\documentclass[reqno]{elsarticle}

\usepackage{hyperref}
\usepackage[mathlines]{lineno}

\usepackage{color}
\usepackage{multicol}
\usepackage{amsfonts}
\usepackage{amsfonts,amssymb}
\usepackage{float}
\usepackage{algorithm}
\usepackage{amsmath}
\usepackage{algorithmic,caption}
\usepackage{graphicx,epsfig}
\newtheorem{theorem}{Theorem}[section]
\newtheorem{lemma}[theorem]{Lemma}
\newtheorem{corollary}[theorem]{Corollary}
\newdefinition{remark}{Remark}
\newdefinition{definition}{Definition}
\newproof{proof}{Proof}

\newcommand{\N}{\mathbb{N}}

\newcommand{\R}{\mathbb{R}}

\newcommand{\norm}[2]{\left\|#1\right\|_{#2}}

\renewcommand{\L}{\mathrm{L}}

\renewcommand{\O}{\Omega}

\newcommand{\Lr}{\L^2_r}
\newcommand{\LrO}{\Lr(\O)}

\newcommand{\drdz}{\dr\dz}
\newcommand{\dr}{\;dr}
\newcommand{\dz}{dz}

\newcommand{\ds}{\;ds}

\newcommand{\rH}{\mathrm H}

\newcommand{\rL}{\mathrm L}

\newcommand{\pdr}{\partial_r}

\newcommand{\pdz}{\partial_z}
\newcommand{\Dt}{\Delta t}

\newcommand{\Om}{\Omega}

\newcommand{\hists}{\mathcal{F}}
\newcommand{\Preit}{\mathcal{F}_{_{\!\!D}}}

\newcommand{\PTr}{\mathcal{T}}
\newcommand{\xin}{\xi}

\newcommand{\vpu}{\rho_1}
\newcommand{\vpd}{\rho_2}
\newcommand{\vpc}{\rho_0}

\newcommand{\hvp}{h_{\rho}}
\newcommand{\dhvp}{\eta_{\rho}}

\newcommand{\FndD}{F}

\newcommand{\cta}{\omega}
\newcommand{\pdisf}{p}

\newcommand{\dro}{\,d\rho}

\newcommand{\calV}{\mathcal{V}}

\def\rou{\rho_1}
\def\rod{\rho_2}

\def\Preic{\mathcal{F}_{_{\!\!S}}}

\begin{document}

\begin{frontmatter}

\title{Mathematical analysis and numerical solution of models with dynamic
Preisach hysteresis}
\tnotetext[mytitlenote]{The work of the authors from Universidade de Santiago de Compostela
was supported by Ministerio de Econom\'{\i}a y Competitividad (Spain)
under the research project ENE2013--47867--C2--1--R and by Ministerio de Econom\'{\i}a,
Industria y Competitividad (Spain) under the research project MTM2017-86459-R.
 P.~Venegas was partially supported by FONDECYT project 11160186.}

\author[USC]{A.~Berm\'udez}
\ead{alfredo.bermudez@usc.es}

\author[USC]{D.~G\'omez}
\ead{mdolores.gomez@usc.es}

\author[UBB]{P.~Venegas}
\ead{pvenegas@ubiobio.cl}

\address[USC]{Departamento de Matem\'atica
Aplicada, Universidade de Santiago de Compostela,
E-15782 Santiago de Compostela, Spain}

\address[UBB]{GIMNAP, Departamento de Matem\'atica,
Universidad del B{\'\i}o-B{\'\i}o,  Concepci\'on, Chile}



\begin{abstract}
 Hysteresis is a phenomenon that is observed in a great variety of physical systems,
which leads to a nonlinear and multivalued behavior, making their modeling
and control difficult. Even though the analysis and mathematical properties of
classical or rate-independent hysteresis models are known, this is not the case
for dynamic models where  current approaches
lack a proper functional analytic framework which is essential to
formulate optimization problems and develop stable numerics, both being crucial in
practice. This paper deals with the description and mathematical analysis of the
dynamic Preisach hysteresis model. Toward that end, we complete a widely
accepted definition of the dynamic model  commonly
used to describe the constitutive relation between the magnetic field
H and the magnetic induction B, in which, the values of B not only depends on
the present values of H but also on the past history and its velocity.
We first analyze mathematically some important properties of the
model and compare them with known results for the static Preisach model.
Then, we consider a parabolic problem with dynamic hysteresis motivated by electromagnetic field equations. Under suitable assumptions, we show the well posedness of a
weak formulation of the problem and solve it numerically.
 Finally, we report a numerical test in order to assess
the order of convergence and to illustrate the behavior
of the numerical solution for different
configurations of the dynamic Preisach model.
\end{abstract}


\end{frontmatter}


\section{Introduction}

Hysteresis is a nonlinear behavior exhibited by some media characterized by a special memory-based property according
to which their response to particular changes is a function of the preceding responses.
This characteristic is known as \textit{memory effect}.
 A survey of hysteresis models may be found, for instance, in  \cite{Macki1993,Hassani2014,Visintin2006}.
 Detailed studies of these models have been done by different authors. From the mathematical point of view,
we refer to the pioneering work of  Krasnosel'ski\u{i} and Pokrovski\u{i}~\cite{KP89}, who introduced the
fundamental concept of hysteresis operator and conducted a systematic analysis of the mathematical properties of these objects.

Models of hysteresis can be divided into two main classes:  \textit{static} or \textit{rate-independent} and
\textit{dynamic} or \textit{rate-dependent} models.
In static models, the values of output depend just on the range of the input and on the order in which
values have been attained; in particular, the speed of the input has no influence~\cite{Vi94}. As a consequence,
the model cannot reflect the dependence with frequency or field waveform. This is a fundamental property of
classical hysteresis phenomenon; in this sense, static models are also known as \textit{rate-independent}
models and some authors~\cite{Vi94} define hysteresis as a rate-independent memory only.  Conversely, in
 dynamic hysteresis, the effect of the speed of changes of the applied input is added to the model.
That is why dynamic models are also known as \textit{rate-dependent} models.

The hysteresis phenomenon has been observed for a long time in many different areas of science and engineering.
 The term was initially coined in the area of magnetism~\cite{Ewing1885} given that many ferromagnetic materials
present hysteresis behavior that is reflected in the magnetization curves describing the magnetic response of the
material to an applied magnetic field.  From the electrical engineering point of view, having a good hysteresis model
is fundamental to, for instance, correctly estimate the energy losses in electrical machines, a very important
characteristic to take into account when designing an electrical device; in particular, the so-called hysteresis and
excess losses~\cite{BDGV2017}. Consequently, building a mathematical model of this relation is a very important
(and difficult) task and numerical simulation of devices involving ferromagnetic materials is still quite a challenge.

One the of the most popular hysteresis operators among scientists and engineers is the  classical Preisach
model~\cite{PreisachF}, a rate-independent  model based on physical assumptions motivated by the concept
of magnetic domains. Even if this model was first suggested in the area of ferromagnetism~\cite{M91,Bertotti1998,PDCM94,T99},
nowadays Preisach type operators~\cite{Bott2000} are recognized as a fundamental tool for describing a wide range
of hysteresis phenomena in different subjects as  elastoplasticity~\cite{Lubarda93}, solid phase transitions~\cite{BS96},
shape memory alloys~\cite{RAO2013}, hydrology~\cite{K2011}, fluid flow in porous
media~\cite{Schweizer2017}, infiltration \cite{Ma2008}, batteries~\cite{BFSVZ2014}, economics or
biology~\cite{Kreji_okane_2012}, among others.

In the case of ferromagnetism, the original formulation of the classical Preisach model allows us to simulate
scalar and rate-independent hysteresis relationship between the magnetic field $H$ and the magnetic flux
density $B$ (or the magnetization $M$). Nevertheless, in many applications, the magnetization evolution
was found to be dependent on the rate of applied field and thus, at  present, there are several extensions of
this classical Preisach model to behave like a dynamic model. They are generically called \textit{dynamic Preisach models}
(see, for instance,  \cite{Mayergoyz1988,B92,Yu2002,DBBFM_2000}).
 The analysis and mathematical properties of classical Preisach  hysteresis models are well-known~\cite{BS96,Vi94}.
However, rigorous mathematical analysis of rate-dependent hysteresis models is still largely open, despite of their
obvious importance in applications. In particular, we focus on  the dynamic Preisach model proposed by Bertotti in
 \cite{B92}. As we will see later, this model is based on a  nonlinear ordinary differential equation, but it was
not clear that a function satisfying this problem existed.

In recent years, mathematical analysis of rate-independent models of hysteresis  coupled with partial differential
 equations has been progressing; see among others the works by Visintin~\cite{Vi94,Visintin2006}, Eleuteri~\cite{M2007,Eleuteri_Kopfova_2010},
Showalter~\cite{Showalter96}, Berm\'udez~\cite{BGRV2014}, Krej\v{c}\'i~\cite{Kr99} and Brokate and
Sprekels~\cite{BS96}, Mielke~\cite{MIELKE2007}, Gurevich~\cite{Gurevich2009}.
for a physical point of view. In particular, the former authors focus on the study of hysteresis in the area of
magnetism where the classical Preisach model in considered.
Even though there are several publications devoted to the numerical solution of partial differential equation
with dynamic Preisach models~\cite{Dupre_dyn_pre_1998,Kuczmann2014,Basso1997},  to the best of the
author's knowledge, mathematical analysis of the dynamic hysteresis models presented in this paper has not been done yet.

The goal of this work is twofold:  the mathematical study of the dynamic Preisach model of hysteresis
 proposed in \cite{B92} and the mathematical analysis  and numerical solution of
parabolic problems with dynamic hysteresis motivated by electromagnetic field equations. With this in mind,
 we first formalize the definition of the dynamic Preisach model, more specifically we focus on the \textit{dynamic relay} which
 is introduced as the solution of a multi-valued ordinary differential equation. The mathematical analysis of this
 equation  is performed and  we prove some properties of the dynamic relay that allow us to obtain  mathematical
 properties of the dynamic Preisach model. From these properties and by applying the same techniques as for
problems modelled by the classical Preisach operador (see, for instance \cite{Vi94}),  we prove existence of
solution of a parabolic equation including dynamic hysteresis.

The paper is organized as follows:  in Section~\ref{sec:relays} we briefly recall the definition of the
rate-independent relay operator and the classical Preisach operator; then, a new formulation of the
dynamic (rate-dependent) relay operator is provided and some properties are proved. In Section~\ref{DPM},
by using the  dynamic relay, we recall the definition of the dynamic Preisach operator and prove some of its
properties. An  abstract parabolic problem with dynamic hysteresis is stated in Section~\ref{PWH}. From the
properties proved for the dynamic Preisach model, an existence result is derived. Finally, in Section~\ref{Num}
we introduce a numerical scheme to approximate the parabolic problem. We report two numerical tests: one in
order to assess the order of convergence and another one to illustrate the behavior of the numerical solution for
 different configurations of the dynamic Preisach model. Finally, some conclusions are drawn.

\section{Relays and Preisach operators}\label{sec:relays}
\setcounter{equation}{0}
 The  Preisach model describes the hysteresis using a superposition of elementary hysteresis operators called
\textit{relay operators}. The model assumes that the material consists of an infinite number of (magnetic)
particles each one characterized by a relay so the whole system can be modeled by a weighted parallel
connections of these relays. The weight function works as a local influence of each operator in the overall
hysteresis model and it is estimated from measured data.

Depending on the characteristics of the relay and on the nature of the connection between them, different
Preisach (or Preisach-type) models can be obtained. In what follows we briefly recall the definition of the
rate-independent relay which is the basis  of the classical Preisach model. Then, we introduce a
rate-dependent or dynamic relay leading to the definition of the \textit{dynamic Preisach} model introduced
in \cite{B92}.
\subsection{The static (rate-independent) relay}
In the classical (rate-independent) Preisach model, the output of each relay is represented by an elementary
rectangular loop on the input-output diagram $(u,v)$ (see Figure~\ref{h_fig:relay} (left)), with transition
thresholds at $\rou$ and $\rod$.
\begin{figure}[H]
\begin{center}
$\vcenter{\hbox{\includegraphics*[width=0.4\textwidth]{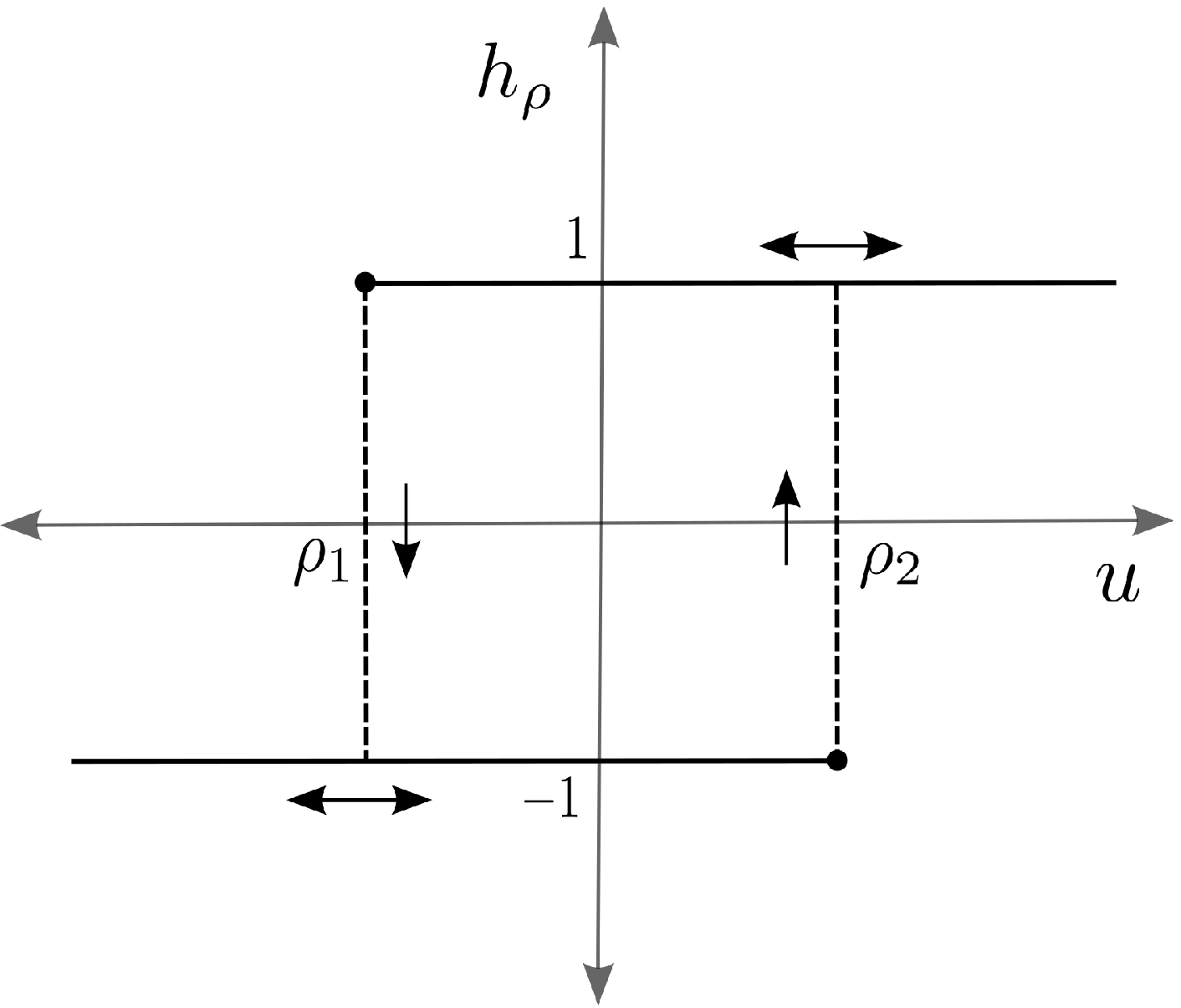}}}$ \qquad
$\vcenter{\hbox{ \includegraphics*[width=0.4\textwidth]{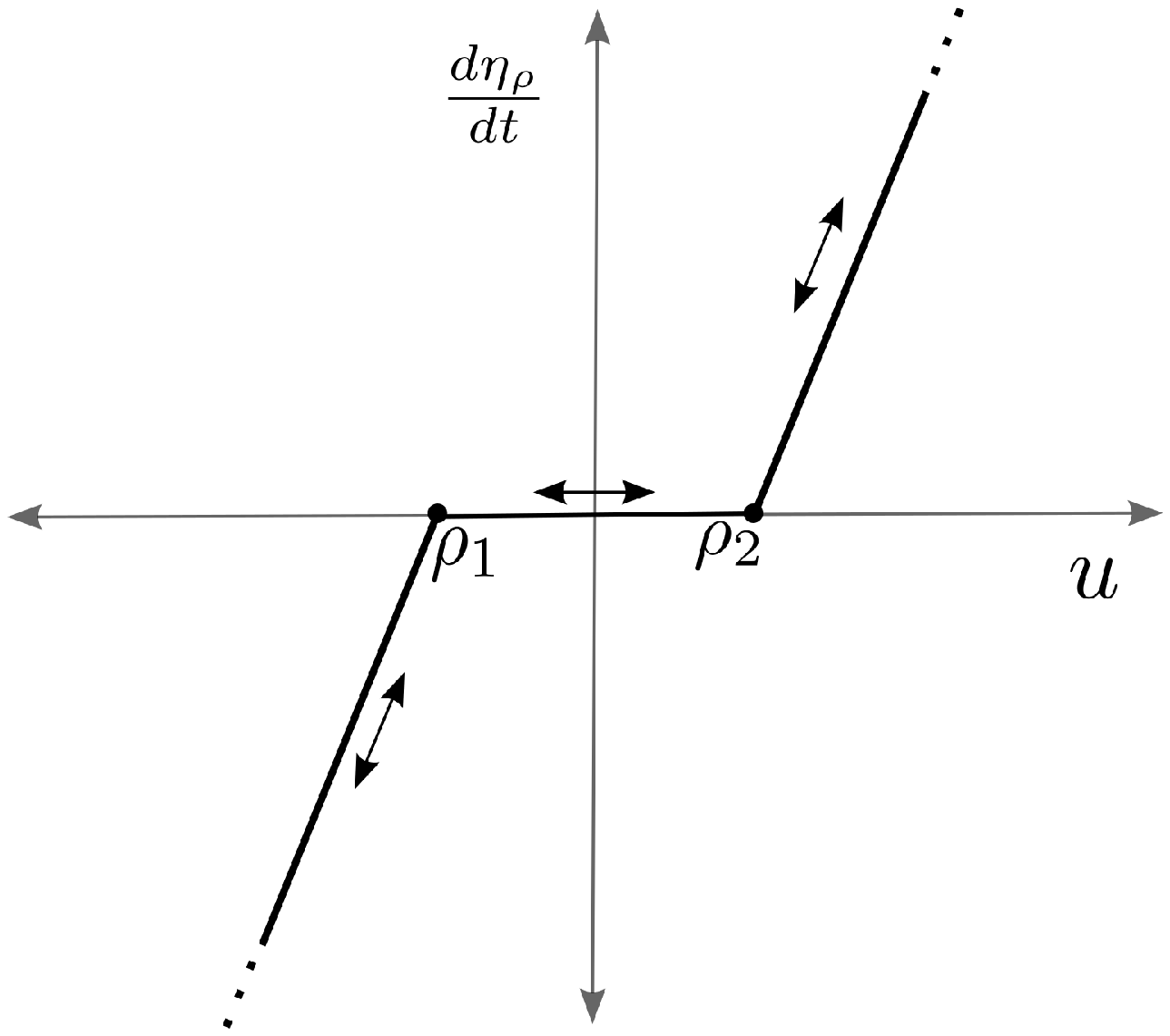}}}$
\caption{Relay in the $(u,\hvp)$-plane for the classical Preisach model (left)  and relay slope
 for the dynamic Preisach model for the case $|\dhvp| < 1$(right).
The arrows mark the authorized paths.
}
\label{h_fig:relay}
\end{center}
\end{figure}
Formally, given any couple $\rho=(\vpu,\vpd)$, such that $\vpu<\vpd$,  the corresponding relay operator
$\hvp$ is defined as follows: for any $u \in C([0,T])$ and $\xi \in\{1,-1\}$ (an initial condition),
$\hvp(u,\xi)$ is a function from $[0,T]$ to $\R$ such that,
$$
\hvp(u,\xi)(0):=\left\{\begin{array}{rl}
                  -1 & \quad\mathrm{if}\; u(0)\leq \vpu, \\
                  \xi &\quad \mathrm{if}\; \vpu<u(0)< \vpd, \\
                  1 &\quad  \mathrm{if}\; u(0)\geq \vpd. \\
                \end{array}\right.
$$
Then, for any $t \in (0,T]$, let us set $X_u(t):=\{\tau \in \,(0, t] : \,\,u(\tau) =  \rou \,\,\mbox{or}\,\, \rod\}$.
This set  keeps account of the previous instants in which $u$ has got the thresholds
$\vpu$ or $\vpd$. We define
\begin{equation}\label{h_rho}
\hvp(u,\xi)(t):=\left\{\begin{array}{cl}
                  h_{\rho}(u,\xi)(0) & \quad \mathrm{if}\,\, X_u(t)=\emptyset\,, \\
-1 & \quad \mathrm{if}\,\, X_u(t) \neq \emptyset \quad\mbox{and}\quad u(\max
X_u(t))=\vpu, \\
~~~1 & \quad \mathrm{if}\,\, X_u(t) \neq \emptyset \quad\mbox{and}\quad u(\max X_u(t))=\vpd.
\\
                \end{array}\right.
\end{equation}

We notice that  $\hvp$ can only be equal to $\pm 1$  depending on the past history of the system,
with instantaneous ``switch-down'' and ``switch-up''  when $u$ takes the values $\vpu$ and $\vpd$, respectively.
The value of the relay operator remains at the last value ($\pm1$) until
$u$ takes the value of one opposite switch, that is, switch to value $1$ when
$u$ attains the value $\vpd$ from below, and to $-1$ when it attains $\vpu$ from
above.

To define the \textit{classical Preisach operator} $\Preic$, it is useful to
introduce the  so-called \textit{Preisach triangle}
$\PTr:=\{ \rho=(\vpu,\vpd) \in \mathbb{R}^2: -\vpc\leq \vpu \leq
\vpd\leq \vpc\}$ where $\vpc>0$ is given.
Let us denote by $Y$ the family of Borel measurable  functions ${\PTr}\rightarrow \{-1, 1\}$
and by $\xi$ a generic element of $Y$. Then, $\Preic$ is given by:
\begin{align}
& \Preic:  C([0,T])\times Y  \longrightarrow  C([0,T]), \nonumber\\
& (u,\xi) \longmapsto
[\Preic(u,\xi)](t)=\int_{\PTr}[h_\rho(u,\xi(\rho))](t)\pdisf(\rho)\dro,
\label{h_FS}
\end{align}
where $\pdisf\in \rL^1({\PTr})$ with $\pdisf>0$ is termed \textit{Preisach density function}.
Thus, the classical Preisach model can be understood as the ``sum'' of a family of static relays,
distributed with a certain density $p$. Considering the definition above, this operator is a hysteresis
operator in the mathematical sense established by Visintin~\cite{Vi94}.
\begin{remark}\label{eq:demag}
  If $u=0$ at $t=0$, then we can consider the following initial condition:
\begin{equation}
\hvp(u,\xi)(0):=\left\{\begin{array}{cl}
                  -1 & \mathrm{if}\,\, \vpu+\vpd>0 \\
                  1 &  \mathrm{if}\,\, \vpu+\vpd<0. \\
                \end{array}\right.
\end{equation}
In electromagnetism, this initial configuration
is usually called ``demagnetized'' or  ``virginal'' state because it leads to
a null magnetic induction when $p$ is symmetric~\cite{M91}.
\end{remark}

Currently, the classical Preisach model serves as a basis for generalizations that try to overcome
some of the lacks of the model. In particular, the fact that the form of the hysteresis diagram does
not reflect the frequency content of the input $u(t)$ (see \eqref{h_rho} and \eqref{h_FS}). In other
words, the model is rate-independent~\cite{M91};  this means that at any instant $t$,
$[\Preic(u,\xi)](t)$ only depends on the image set $u([0,t])$ and on the order in which these values of
 $u$ have been attained, but not on the rate of change of $u$. Formally, this can be expressed
as follows~\cite{Vi94}:
\begin{definition}\label{def:rate_indep}
$\zeta: \mbox{Dom}(\zeta)\subset C([0,T])\times[-1,1]$ is said
\textit{rate independent} if the path of the pair
$(u,\zeta(u,\xi))$ is invariant with respect to any increasing
diffeomorphism  $\varphi: [0, T] \rightarrow
[0, T]$, i.e.,
\begin{align*}
&   \forall (u,\xin) \in \mathrm{Dom}(\zeta),\nonumber \\
&\zeta(u \circ\varphi,\xin) =\zeta(u,\xin) \circ \varphi ~~\mbox{in}~ [0, T].
\end{align*}
%
\end{definition}

In order to take into account the rate-dependent effects, in the following section  we will introduce the so-called dynamic relay.

\subsection{The dynamic relay}

In the rate-dependent generalization of the Preisach
model (the so-called \textit{dynamic} Preisach model) introduced by Bertotti~\cite{B92,BP92}, the relays are
assumed to switch at a finite rate proportional to the difference between
$u(t)$ and the  switching values $\rho_1$ and $\rho_2$
(see Figure~\ref{h_fig:relay} (right)). In particular, this means that, contrary to the classical relay where
only two states, $-1$ and $1$ are possible, now all intermediate states in the interval $[-1,1]$ can be attained;
see Figure~\ref{fig:ex_1} and Remark~\ref{rmk_1}. The proportionality factor, denoted by $k$, is a material-dependent parameter.

Formally, for a fixed $\rho=(\rho_1,\rho_2)\in\R^2$, $\rho_1< \rho_2$, and motivated
 by \cite{B92}, we define the dynamic relay operator
 $\dhvp: \rL^2(0,T)\times [-1,1]\to \rH^1(0,T)$  such that, for any $u\in\rL^2(0,T)$
 and $\xi \in [-1,1]$, $\dhvp(u,\xi):[0,T]\to [-1,1]$  is the unique  function
 $y\in \rH^1(0,T)$ such that $-1\leq y(t)\leq 1$ and solves the nonlinear Cauchy problem:
\begin{align}
\label{eq:carq1}
\frac{dy}{dt}(t)&=F(t,y(t)):=\left\{ \begin{array}{l}
k(u(t)-\rho_2)^+ -k(u(t)-\rho_1)^- \,\, \mbox{ if } -1<y(t)<1,
\\
0\phantom{(u(t)-\rho_2)}  \quad \mbox{ if } y(t)=-1 \mbox{ and }u(t)\leq\rho_2,
\\
 k(u(t)-\rho_2)  \quad \mbox{ if } y(t)=-1 \mbox{ and }u(t)\geq \rho_2,
\\
k(u(t)-\rho_1) \quad \mbox{ if } y(t)=1 \mbox{ and }u(t)\leq\rho_1,
\\
0\phantom{(u(t)-\rho_2)}  \quad \mbox{ if } y(t)=1 \mbox{ and }u(t)\geq \rho_1,
\end{array}\right.
\\
\label{eq:initial}
y(0)&=\xi.
\end{align}
We have used the standard  notations:
$$
x^+= \max\{x,0\} \mbox{ and } x^-=\max\{-x,0\},
$$
so that $x=x^+-x^-$.
Notice that \eqref{eq:carq1} partially coincides with the definition given in \cite{B92} when $-1< y < 1$.
Nevertheless, in order to perform a suitable mathematical analysis, we need to consider  all the other
cases included in \eqref{eq:carq1}. We also notice that this formulation and that of Bertotti are consistent
because cases third and fourth cannot happen, as we will see in the sequel.

Let $g_\rho:\mathbb{R}\rightarrow \mathbb{R}$ be defined by
$$
g_\rho(u):=  -k(u-\rho_1)^- + k(u-\rho_2)^+.
$$
Then $g$ is Lipschitz-continuous:
$$
|g_\rho(u_1)-g_\rho(u_2)|\leq k |u_1-u_2|.
$$
In terms of $g_\rho(u(t))$, (\ref{eq:carq1}) can be  rewritten as follows
\begin{align}
\label{eq:carq1p}
\frac{dy}{dt}(t)&=\left\{
\begin{array}{ll}
g_\rho(u(t)) & \mbox{ if } -1<y(t)<1,
\\
0  & \mbox{ if } y(t)=-1 \mbox{ and }g_\rho(u(t))\leq0,
\\
g_\rho(u(t))  & \mbox{ if } y(t)=-1 \mbox{ and }g_\rho(u(t))\geq0,
\\
 g_\rho(u(t)) & \mbox{ if } y(t)=1 \mbox{ and }g_\rho(u(t))\leq0,
\\
0  & \mbox{ if } y(t)=1 \mbox{ and }g_\rho(u(t))\geq0,
\end{array}\right.
\\
\label{eq:initial2}
y(0)&=\xi,
\end{align}
To prove the existence and uniqueness of such a function $y$
we will consider another  \textit{apparently different} initial-value
problem, namely, find $y\in \rH^1(0,T)$ satisfying  $-1\leq y(t)\leq 1$
and $q\in \mathrm{L}^2(0,T)$ such that
\begin{equation}
\label{eq:dyn_relay}
\left\{\begin{array}{l}
\dfrac{d y}{dt}(t)=g_\rho(u(t))-q(t),
\\
y(0)=\xi,
\end{array}\right.
\end{equation}
and
\begin{equation}
\label{eq:q}
\left\{\begin{array}{ll}
q(t)=0 & \mbox{ if } |y(t)|< 1, \\
q(t)\in (-\infty,0] &    \mbox{ if } y(t)=-1, \\
q(t)\in [0,\infty) & \mbox{ if } y(t)=1.
\end{array}\right.
\end{equation}
Function $q(t)$ is a priori unknown and can be considered as a Lagrange
multiplier associated with the constraint $|y(t)|\leq 1$. In fact, problem  \eqref{eq:dyn_relay}-\eqref{eq:q} can be written in a more
compact but equivalent way as a multi-valued ordinary differential equation:
\begin{equation}
\label{eq:relay}
\left\{\begin{array}{l}
\dfrac{dy}{dt}(t)+\partial \chi_{[-1,1]} (y(t))\ni g_\rho(u(t)),
\\
y(0)=\xi,
\end{array}\right.
\end{equation}
where $\partial \chi_{[-1,1]}$ denotes the sub-differential of
$\chi_{[-1,1]}$ which is the indicator function of the interval
$[-1,1]$. Let us recall that (see, for instance, \cite{B73})
$$
\chi_{[-1,1]}(x)=\left\{
\begin{array}{ll}
\infty  & \mbox{ if } |x|> 1,
\\
0   &  \mbox{ if } |x|\leq 1,
\end{array}
\right.
\qquad
\mbox{and}
\qquad
\partial\chi_{[-1,1]}(x)=\left\{
\begin{array}{ll}
[0,\infty) & \mbox{ if } x=1,
\\
(-\infty,0] & \mbox{ if } x=-1,
\\
0  &\mbox{ if } |x|<1,
\\
\emptyset & \mbox{ if } |x|> 1.
\end{array}
\right.
$$
Thus, it is easy to see that function $q(t)$  belongs to
$\partial \chi_{[-1,1]}(y(t))$ for each $t\in [0,T]$ which means $|y(t)|\leq 1$ and
$$
q(t)\Big(z-y(t)\Big) \leq 0 \quad \forall z\in \mathbb{R} \mbox{ with } |z|\leq 1.
$$
Indeed, we notice that the latter is equivalent to (\ref{eq:q}). Let us
emphasize that function $q(t)$ is also an unknown of the problem
and, as we will see below, it is unique too. Actually, the value
of $q(t)$ accommodates so that the solution of the Cauchy problem satisfies $|y(t)|\leq 1$,
i.e., it is a Lagrange multiplier associated to this constraint.

The existence of a unique solution to the Cauchy problem
(\ref{eq:relay}) has been proved in a much more general setting
(see for instance \cite{B73}). However, for the sake of completeness
 we include  a direct proof for this simpler case in \ref{ap:proof}.
\begin{theorem}
\label{eq:existencerelay}
Let us assume that $u$ is a given function in $\mathrm{L}^2(0,T)$. Then for each
$\xi\in [-1,1]$ there exists a unique function $y\in \rH^1(0,T)$ satisfying
(\ref{eq:relay}). Consequently,  function $q(t)$ defined almost everywhere by
$$
q(t):=-\dfrac{dy}{dt}(t)-k(u(t)-\rho_1)^- + k(u(t)-\rho_2)^+
$$
belongs to $\mathrm{L}^2(0,T)$ and $q(t)\in \partial \chi_{[-1,1]}(y(t)) \mbox{ a.e. in } (0,T)$.
\end{theorem}

We also have the following characterization of $q(t)$.
\begin{lemma}
For a.e. in $[0,T]$ we have,
\begin{equation}
\label{eq:carq2}
q(t)=P_{\partial\chi_{[-1,1]}(y(t))}\big(g_\rho(u(t))\big)=\left\{ \begin{array}{ll}
0 & \mbox{ if } -1<y(t)<1,
\\
g_\rho(u(t))  & \mbox{ if } y(t)=-1 \mbox{ and }g_\rho(u(t))\leq0,
\\
0  & \mbox{ if } y(t)=-1 \mbox{ and }g_\rho(u(t))\geq0,
\\
0  & \mbox{ if } y(t)=1 \mbox{ and }g_\rho(u(t))\leq0,
\\
g_\rho(u(t))  & \mbox{ if } y(t)=1 \mbox{ and }g_\rho(u(t))\geq0,
\end{array}\right.
\end{equation}
where $P_{\partial\chi_{[-1,1]}(y(t))}$ denotes the
projection on the set $\partial\chi_{[-1,1]}(y(t))$.
\end{lemma}
\begin{proof}
It follows immediately from  Remark 3.9 in \cite{B73}.\qed
\end{proof}

As a consequence of the previous characterization of $q(t)$ we obtain an equivalent
expression for \eqref{eq:dyn_relay}.
\begin{corollary}
We have, a.e. in $[0,T]$,
\begin{equation}
\label{eq:carq3}
\frac{dy}{dt}(t)=g_\rho(u(t))- q(t) 
=\left\{ \begin{array}{ll}
g_\rho(u(t)) & \mbox{ if } -1<y(t)<1,
\\
0  & \mbox{ if } y(t)=-1 \mbox{ and }g_\rho(u(t))\leq0,
\\
g_\rho(u(t))  & \mbox{ if } y(t)=-1 \mbox{ and }g_\rho(u(t))\geq0,
\\
 g_\rho(u(t)) & \mbox{ if } y(t)=1 \mbox{ and }g_\rho(u(t))\leq0,
\\
0  & \mbox{ if } y(t)=1 \mbox{ and }g_\rho(u(t))\geq0.
\end{array}\right.
\end{equation}
and, therefore, the solution of \eqref{eq:relay}  is also the solution of the
original problem \eqref{eq:carq1}-\eqref{eq:initial}. This follows from
\eqref{eq:carq3} and the fact that $g_\rho(u(t))\leq 0$ is equivalent to $u(t)\leq  \rho_2$,
and the same is true for $g_\rho(u(t))\geq  0$ and $u(t)\geq \rho_1$.
\end{corollary}

From the previous analysis it follows that, for a given input
$u$ and initial state $\xi$, the dynamic relay can be computed by
solving the multi-valued ordinary differential equation \eqref{eq:relay}.
This can be done, for instance, by using semi-smooth Newton~\cite{HK02} or Berm\'udez-Moreno method~\cite{BM81}, just to name a few.
However, motivated by \cite{B92}, we notice that for an input $u\in C([0,T])$
such that  $X_{u}(T)=\{t_u^1,\ldots, t_u^M\}, M \geq 1$, if we  define
$t_u^0:=0$ and $t_u^{M+1}:=T$, then the dynamic relay $y:=\dhvp(u,\xi)$
can be obtained by integrating \eqref{eq:carq3} (equivalently \eqref{eq:carq1})
with respect to $t$ and taking into account the saturation $-1\leq y\leq 1$.
Notice that, for each $t \in (t_u^m,t_u^{m+1}], m=0,\ldots,M+1$,  $y(t)$
is either nondecreasing  (if $u(t)\geq \rho_2$), or
nonincreasing (if $u(t)\leq \rho_1$) or constant (if $\rho_1 \leq u(t)\leq \rho_2$).
For ease of presentation, in order to obtain an expression for $y(t)$, we consider
the first interval, namely, $t \in (t_u^0,t_u^1]$. In this case it follows that,
\begin{itemize}
\item  If $ \rho_1\leq u(t) \leq \rho_2 $ then, from \eqref{eq:carq3}, the slope of $y$ vanishes
along the interval. Thus,  $y(t)=y(t_u^0)=\xi$ on $(t_u^0,t_u^1]$.
\item If $u(t)\geq\rho_2$  then $y(t)$ increases but cannot exceed $1$. Notice that the third
case in \eqref{eq:carq3} cannot be fulfilled on $(t_u^0,t_u^1]$, because, when $y(t_u^0)=-1$
the slope of $y$ is positive and $y(t)>-1, t\in (t_u^0,t_u^1]$ .
Thus, If $|y|< 1$ or  $y=1$,  from the first and fifth cases in \eqref{eq:carq3}, respectively,
we arrive at
\[
y(t)=\min\left\{1, y(t_u^0)+\displaystyle\int_{t_u^0}^t k(u(s)-\rho_2)ds\right\}.
\]
\item If $u(t)\leq\rho_1$,  then $y(t)$ decreases but cannot be less than $-1$.
Similar to the previous case, we notice that the fourth
case in \eqref{eq:carq3} cannot be fulfilled on $(t_u^0,t_u^1]$
because when $y(t_u^0)=1$ the slope of $y$ is negative and $y(t)<1, t\in (t_u^0,t_u^1].$
Thus,  if $|y|< 1$ or  $y=-1$, from the first and second cases in \eqref{eq:carq3}, respectively,
we arrive at
\[
y(t)=\max\left\{-1, y(t_u^0)+\displaystyle\int_{t_u^0}^t k(u(s)-\rho_1)ds\right\}.
\]
\end{itemize}
Thus, by proceeding similarly with the remaining intervals, it follows that
for $t\in~(t_u^m,t_u^{m+1}], m=0,\ldots,M$:
\begin{align}
y(t)
&=
\left\{\begin{array}{cl}
\min\left\{1, y(t_u^m)+\displaystyle\int_{t_u^m}^t k(u(s)-\rho_2)ds\right\}
& \mbox{ if } u(s)\geq\rho_2 \mbox{ for } s \in (t_u^m,t],\\
\max\left\{-1, y(t_u^m)+\displaystyle\int_{t_u^m}^t k(u(s)-\rho_1) ds\right\}
&\mbox{ if } u(s)\leq\rho_1 \mbox{ for } s \in (t_u^m,t], \\
y(t_u^m) & \mbox{otherwise.}
\label{eq:dyn_relay_int2}
\end{array}\right.
\end{align}
%
%
To give an idea of the plane curve $t\in [0,T]\rightarrow (u(t),y(t))$,
we consider two examples with a sinusoidal input
$u(t)=200\sin(2\pi f t )$  and initial condition $\xi=-1$. This input is depicted in Figure~\ref{fig:ex_1} (top and bottom left) for  frequency $f=20$Hz,
where the dotted lines represent different $(\rho_1,\rho_2)$ values.
Figure~\ref{fig:ex_1} (top right) shows the  curve $t\in [0,T]\rightarrow (u(t),y(t))$ when the   relay $\dhvp$
is characterized by switching values  $(\rho_1,\rho_2)=(50,100)$ and slopes $k\in \{1,50,10^8\}$. For the same slopes, in
Figure~\ref{fig:ex_1} (bottom right) we show this curve for switching values $(\rho_1,\rho_2)=(-50,50)$.
Notice that $(u(t),y(t))$ may vary in an asymmetric way when the input is not symmetric with
respect to $(\rho_1,\rho_2)$ (see dashed line in Figure~\ref{fig:ex_1} (top right)).
We also notice that the dash-dotted line in Figure~\ref{fig:ex_1} (top and bottom right) becomes
similar to the discontinuous  relay of the classical Preisach
model (see Figure~\ref{h_fig:relay} (left)), that is, the   classical relay
can be seen as the limit as $k\rightarrow \infty$ of the dynamic relay.
\begin{figure}[h]
\begin{center}
\includegraphics[width=0.42\textwidth]{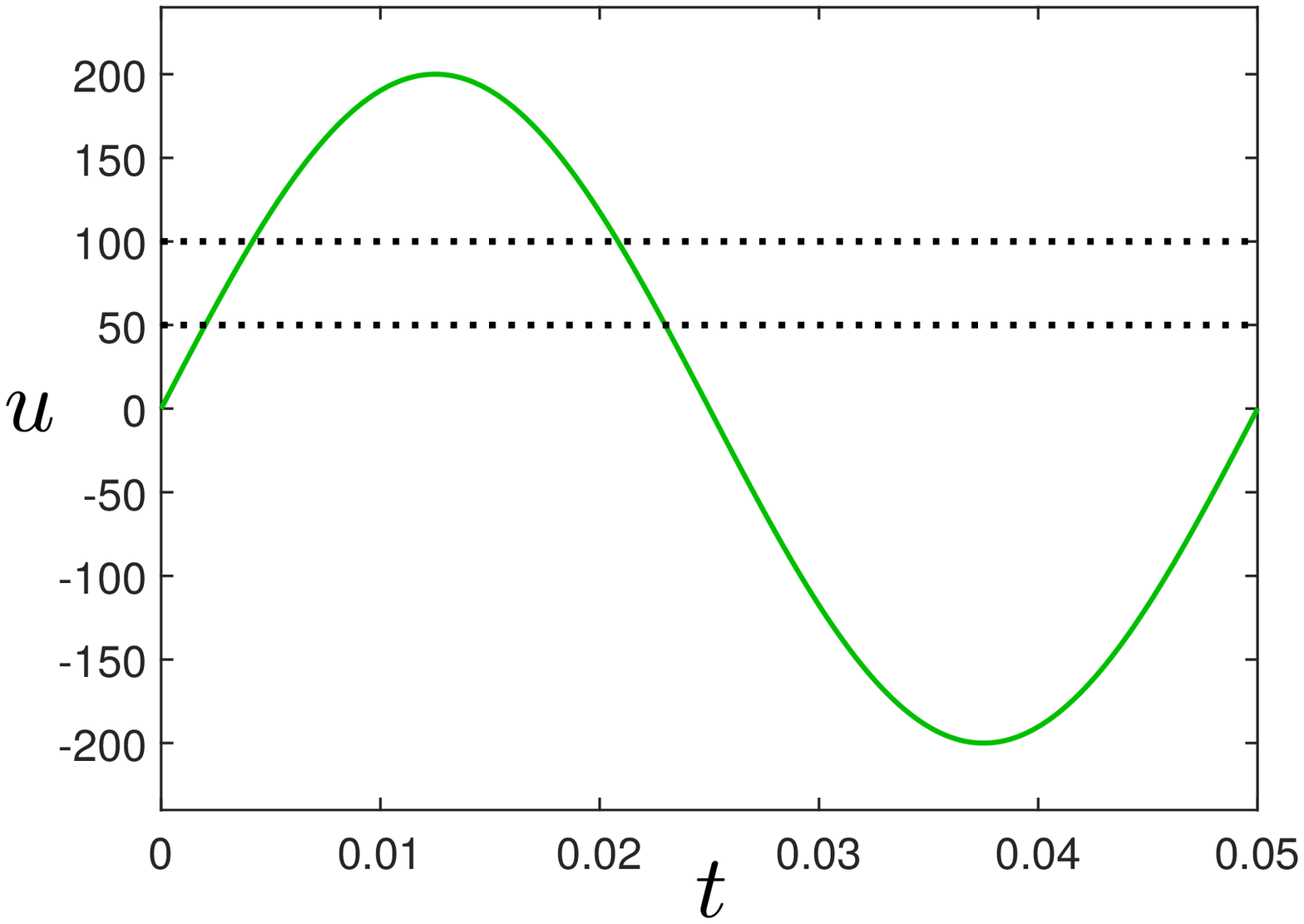}
\includegraphics[width=0.42\textwidth]{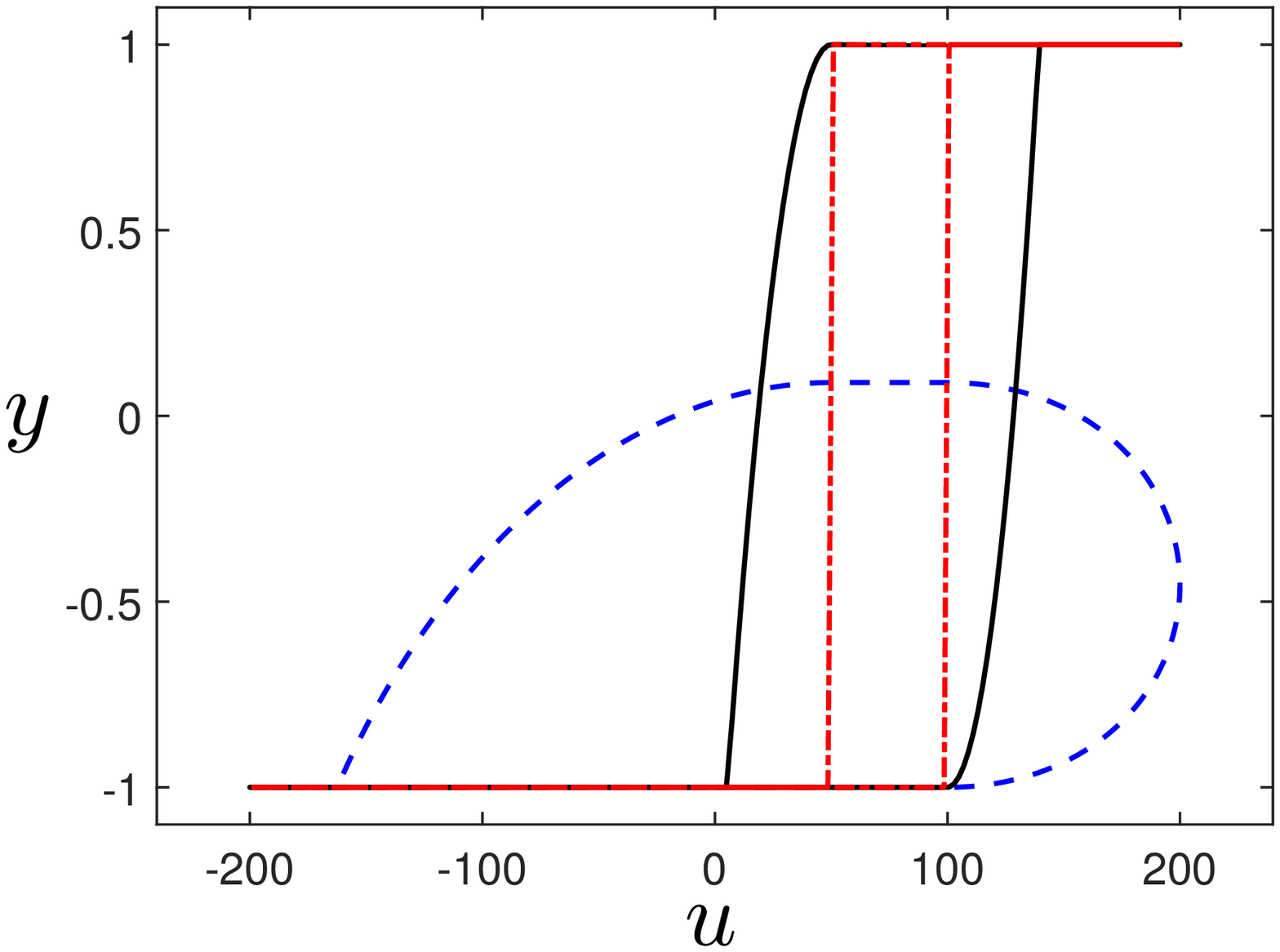}\\
\includegraphics[width=0.42\textwidth]{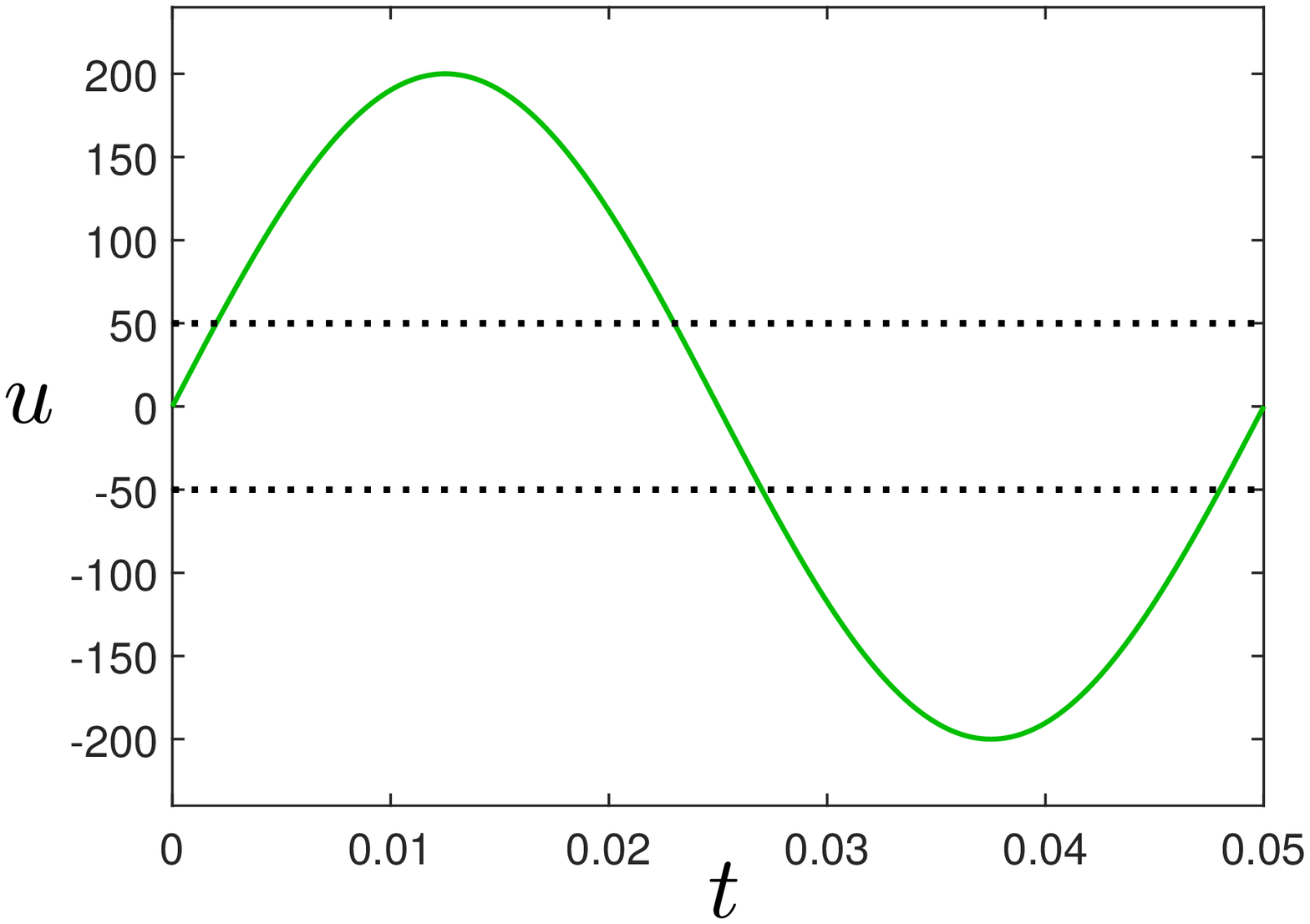}
\includegraphics[width=0.42\textwidth]{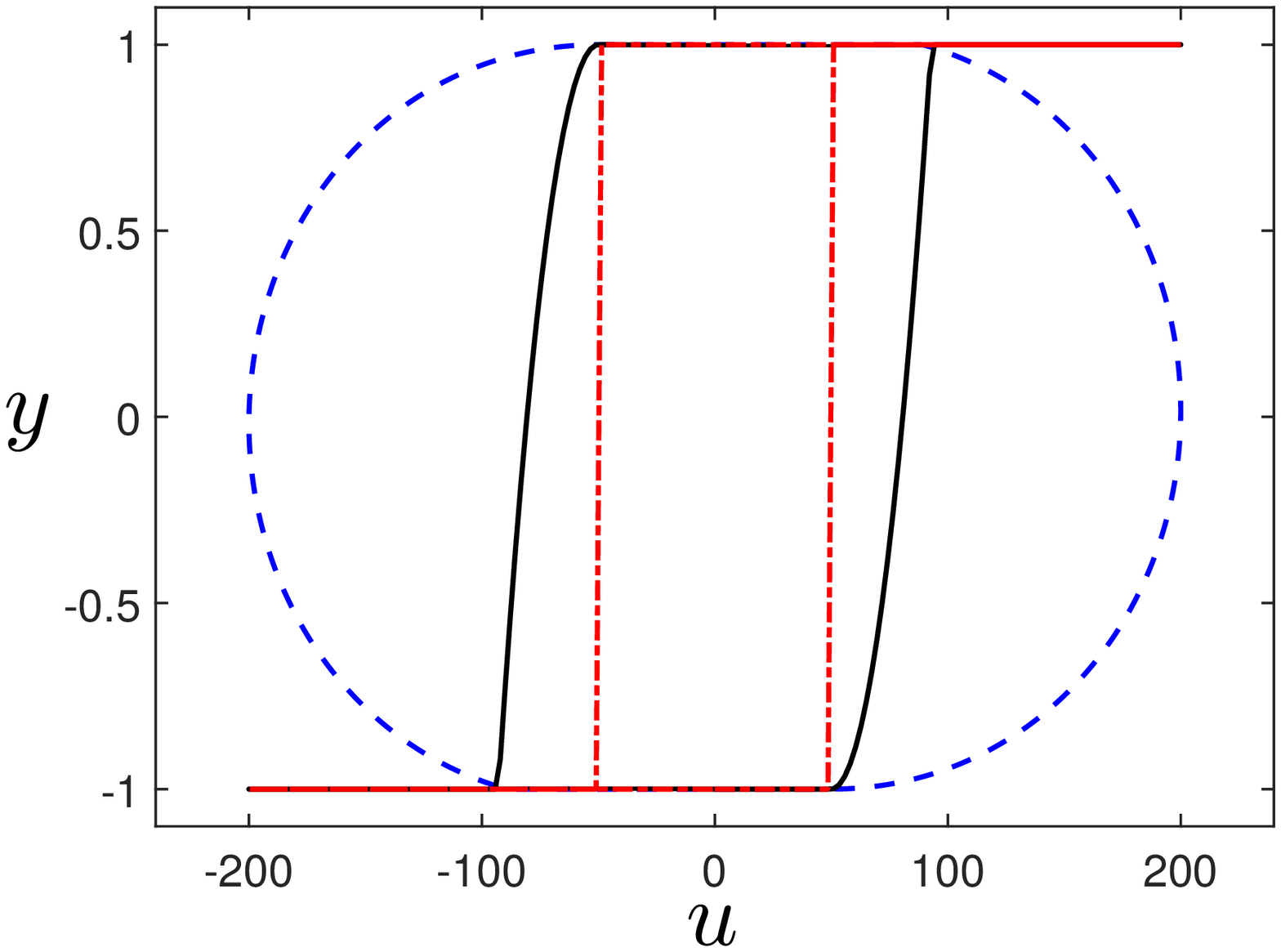}\\
\end{center}
\caption{Input function $u(t)=200 \sin(2\pi ft)$ (left) for frequency $f=20$Hz.
Three  dynamic relays are presented for
switching values  $(\rho_1,\rho_2)=(50,100)$ (top), $(\rho_1,\rho_2)=(-50,50)$ (bottom)  and initial state $\xi=-1$.
The switching values are represented by dotted lines on the left panels.
The right panels show the different
relay  with slopes  $k=1$ (dashed line), $k=50$ (solid line)  and $k=10^8$ (dash-dotted line). }
\label{fig:ex_1}
\end{figure}

For the second example we analyze the evolution of the relays with a sinusoidal input
$u(t)=200\sin(2\pi f t )$ for  $k=50$ and different frequencies $f$.
Figure~\ref{fig:ex_2} shows the dynamic relays
characterized by $(\rho_1,\rho_2)=(50,100)$ (top right) and $(\rho_1,\rho_2)=(-50,50)$ (bottom right)
for frequencies 50, 500 and 5000 Hz.
From these examples we can see the variation of the dynamic relay with respect to $k$
and the input rate.   In particular, Figure~\ref{fig:ex_2} shows that the  dynamic relay is rate-dependent.
\begin{figure}[h]
\begin{center}
\includegraphics[width=0.42\textwidth]{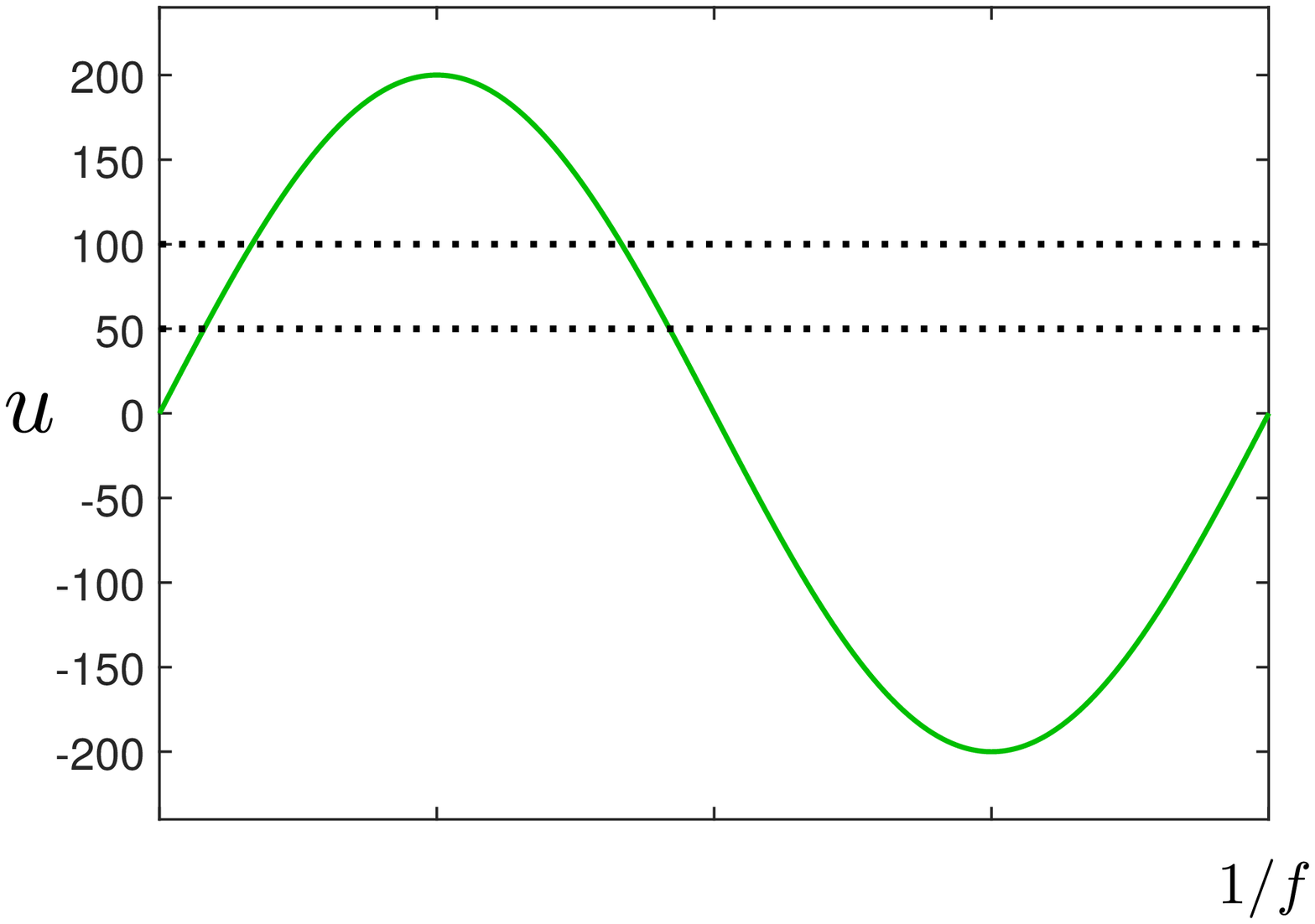}
\includegraphics[width=0.42\textwidth]{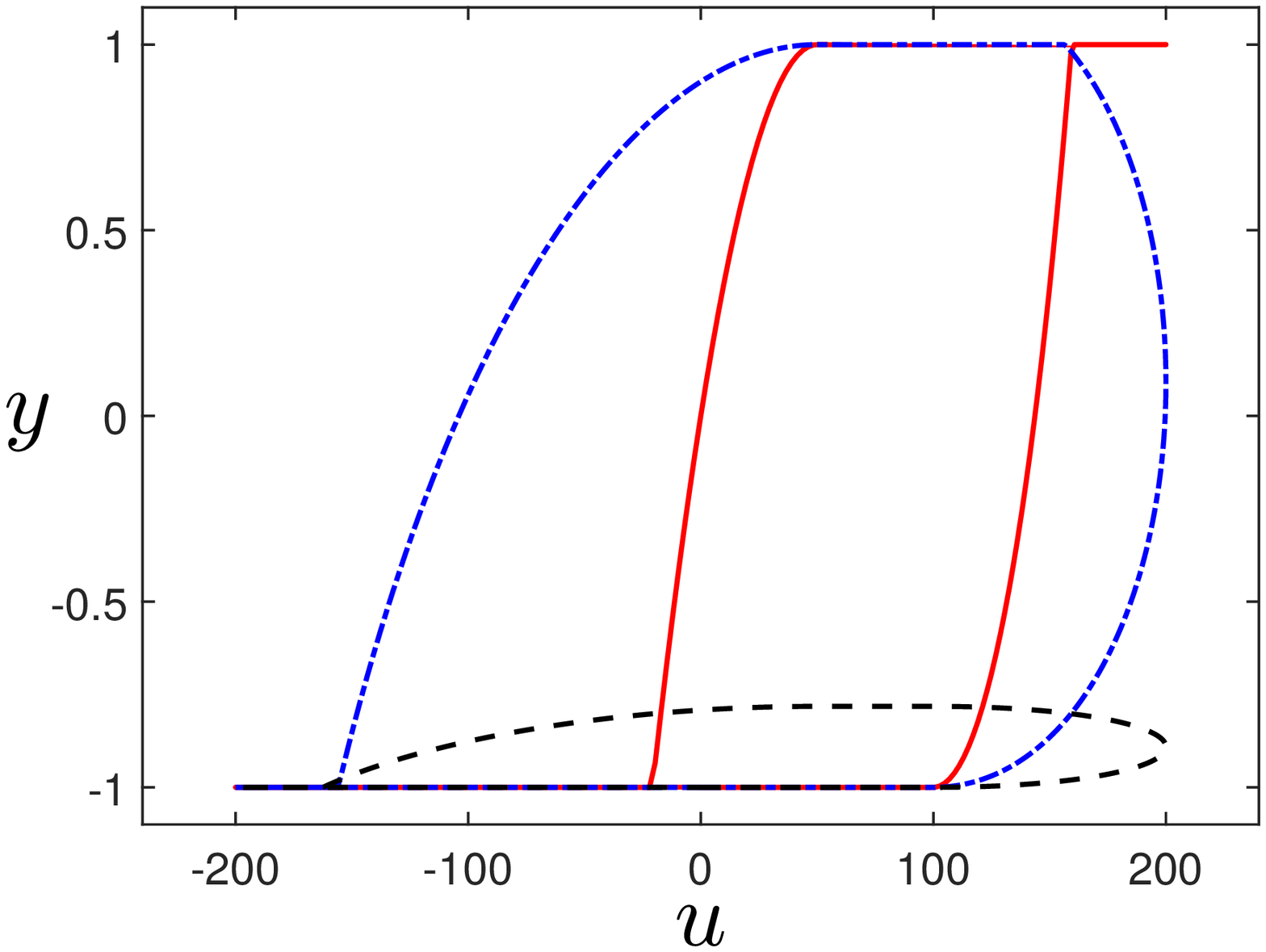}\\
\includegraphics[width=0.42\textwidth]{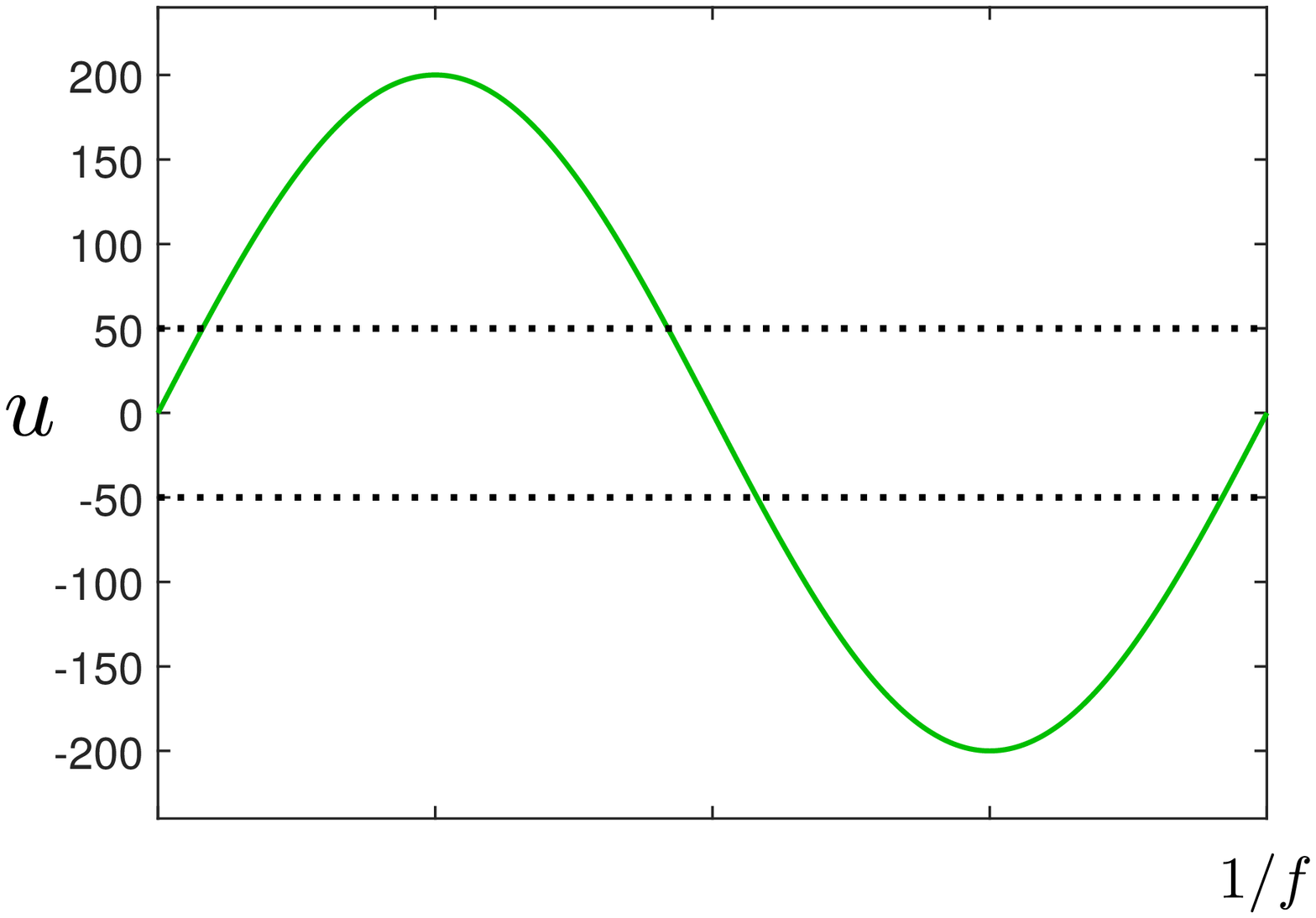}
\includegraphics[width=0.42\textwidth]{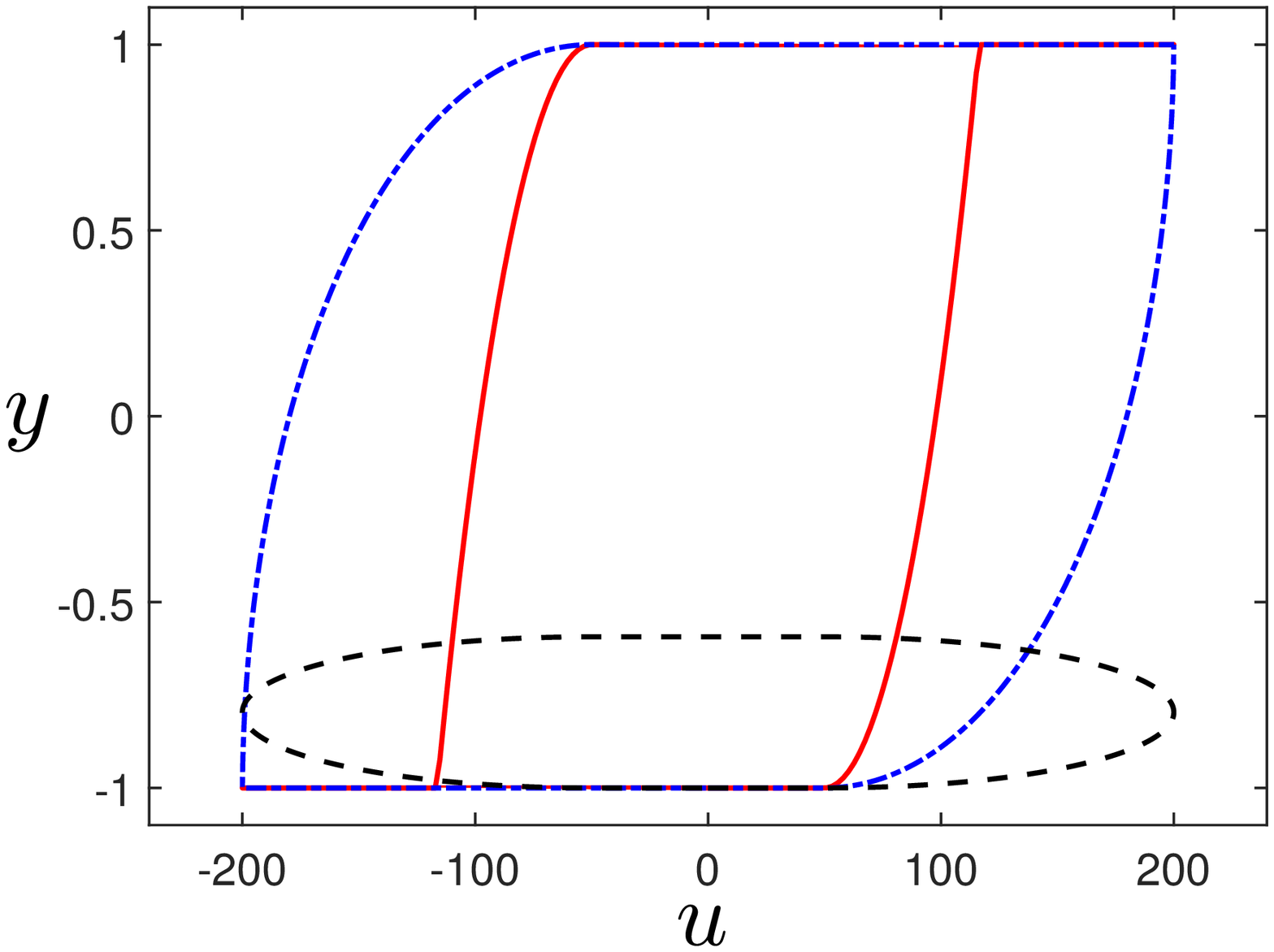}\\
\end{center}
\caption{Input function $u(t)=200\sin(2\pi f t )$ (left) for a fixed
frequency denoted by $f$. Three  dynamic relays are presented for
switching values  $(\rho_1,\rho_2)=(50,100)$ (top), $(\rho_1,\rho_2)=(-50,50)$ (bottom) and initial state $\xi=-1$.
The switching values are represented by dotted lines on the left panels. The  figures at the right show
 the $(u,y)$ curve corresponding to the dynamic relay for $k=50$ and three frequencies: $f=50$Hz (dashed line),
 $f=500$Hz (dash-dotted line) and $f=5000$Hz (solid line).}
\label{fig:ex_2}
\end{figure}
\begin{remark}\label{rmk_1}
Unlike the classical Preisach model, in the dynamic model the initial
state $\xi$ belongs to  $[-1,1]$ instead of $\{-1,1\}$.
Figure~\ref{fig:rmk_1} shows the relays for $u_1(t)=150\sin(2\pi f t )$
and $u_2(t)=150\sin(2\pi f t )+75$, for $f=20$ Hz and two different initial states:
$\xi_1=-0.5$ (left) and $\xi_2=0.5$ (right). Two slopes are considered
for each initial  state: $k=5$ and $k=50$.
  \begin{figure}[h]
\begin{center}
\includegraphics[width=0.42\textwidth]{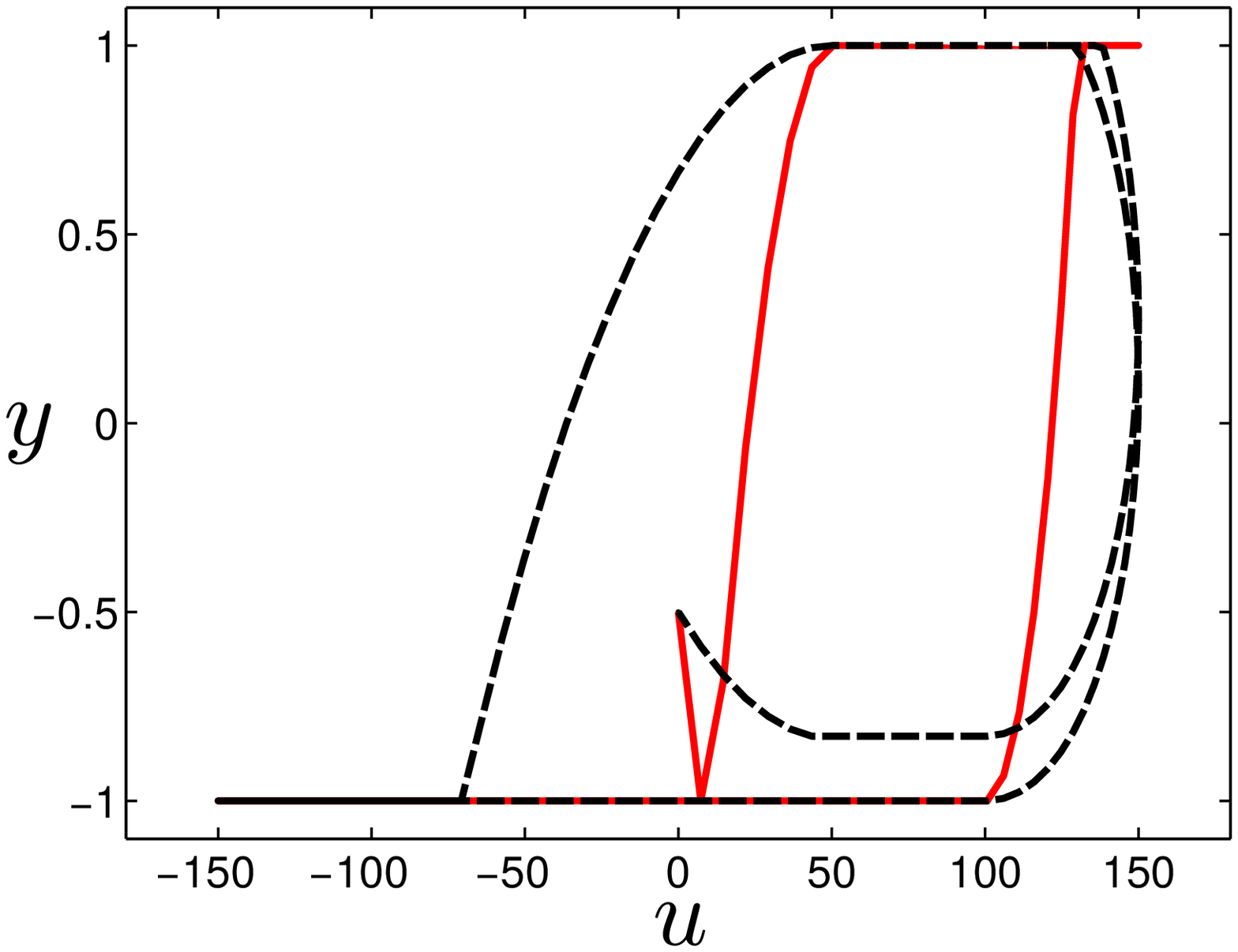}
\includegraphics[width=0.42\textwidth]{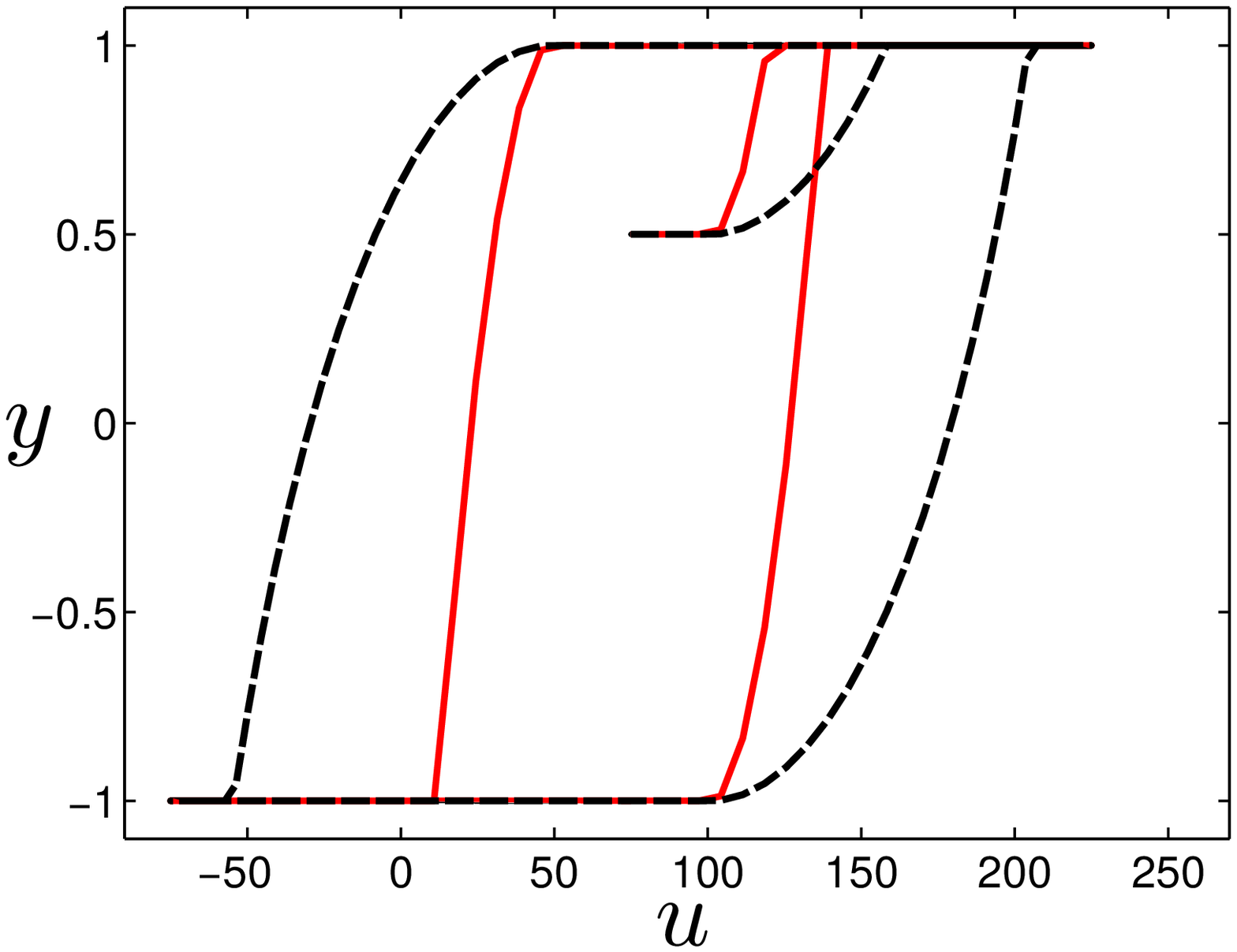}
\end{center}
\caption{The left panel shows
a dynamic relay corresponding to the input $u_1(t)=150\sin(2\pi f t )$  and initial state $\xi=-0.5$.
The right panel shows a dynamic relay corresponding to input $u_2(t)=150\sin(2\pi f t )+75$
 and initial state $\xi=0.5$.
 For both relays $f=20$ Hz; two slopes,  $k=50$ (solid line) and $k=5$ (dashed line) are considered. }
\label{fig:rmk_1}
\end{figure}
\end{remark}
From the previous examples we notice that, unlike the classical
static relay,  the dynamic relay $\dhvp$ does not satisfy either
the so-called \textit{rate independence} (see Definition~\ref{def:rate_indep}) or  the \textit{piecewise monotonicity}
properties defined as follows:
\begin{definition}
$\zeta$ is said \textit{piecewise monotone} if when $ u$ is
either nondecreasing or nonincreasing in $[t_1,
t_2] \subset [0,T]$ then so is $\zeta(u,\xin)(t)$.
\end{definition}
Notice that the dashed curve on Figure~\ref{fig:ex_2} (right)  shows that the latter is
not satisfied by the dynamic relay.

In the sequel we will show some properties of the dynamic
relay solution to \eqref{eq:carq1}-\eqref{eq:initial}.
The following  lemma shows that, like the classical relay,
the dynamic relays ``preserves the order'' of the inputs.
\begin{lemma}\label{lemma:eta_order} Let $u_1,u_2 \in \mathrm{L}^2(0,T)$,
$\xin_1, \xin_2 \in [-1,1]$ and $k>0$. The dynamic relay $\dhvp$ satisfies the following
order preservation property:
\begin{equation}\label{eq:order_pre}
\mbox{If } u_1 \leq u_2 \mbox{  a.e. in  }  [0,T] \mbox{  and  } \xin_1\leq \xin_2,
\mbox{  then  } [\dhvp(u_1,\xin_1)](t) \leq [\dhvp(u_2,\xin_2)](t)\,\,
\forall t \in [0,T].
\end{equation}
\end{lemma}
\begin{proof}
  Let $y_i(t)$ be the solution of problem
 \eqref{eq:relay} for $u=u_i(t)$, $i=1,2$ and respective initial condition
 $y_i(0)=\xi_i$, $i=1,2$. Then we have
\begin{equation}
\label{eq:order1}
\frac{dy_1}{dt}-\frac{dy_2}{dt} +q_1(t)-q_2(t)=g_{\rho}(u_1(t))-g_{\rho}(u_2(t)),
\end{equation}
$y_1(0)-y_2(0)=\xi_1-\xi_2$ and $q_i(t)\in \partial\chi_{[-1,1]}(y_i(t))$,
$i=1,2$. Let us multiply (\ref{eq:order1}) by $\big(y_1(t)-y_2(t)\big)^+$.  We get
\begin{multline*}
 \frac{1}{2} \frac{d}{dt}|\big(y_1(t)-y_2(t)\big)^+|^2
 +(q_1(t)-q_2(t))\big(y_1(t)-y_2(t)\big)^+\\
=\big(g_{\rho}(u_1(t))-g_{\rho}(u_2(t))\big)(y_1(t)-y_2(t)\big)^+\leq 0
\end{multline*}
By integrating from $0$ to $T$ and using the monotonicity
of $\partial\chi_{[-1,1]}$ we deduce
$$
 \frac{1}{2}|\big(y_1(t)-y_2(t)\big)^+|^2\leq
 \frac{1}{2}|\big(y_1(0)-y_2(0)\big)^+|^2= \frac{1}{2}|\big(\xi_1-\xi_2\big)^+|^2=0
$$
Hence $\big(y_1(t)-y_2(t)\big)^+=0$ from which the result follows.\qed
\end{proof}
As a consequence of the previous result,  the following property,
to be used in the sequel, holds true.
\begin{corollary}
 Let $\xin \in [-1,1]$ and $u,v\in \rL^2(0,T)$ such that $u(t)=v(t)$  a.e. in $[0,t_1]$.
If $u\geq v$ a.e. in $[t_1,t_2]$, $t_1\leq t_2$, then
\begin{equation}\label{eq:pwl_eta}
  \left([\dhvp(u,\xin)](t)-[\dhvp(v,\xin)](t)\right)\left(u(t)-v(t)\right)
  \geq 0 \mbox{ a.e. in }[t_1,t_2].
\end{equation}
\end{corollary}

The following lemma establishes a continuity property
of dynamic relays.
%
\begin{lemma}\label{lemma:eta_cont}
 Let   $k>0$ be given,
then the operator $\dhvp:\rL^2(0,T)\times [-1,1]\to C([0,T])$ is Lipschitz-continuous.
More precisely, let $u_1,\,u_2$ be any functions in $ \mathrm{L}^2(0,T)$, $\xi_1,\xi_2\in[-1,1]$, and $y_1,\,y_2$
their respective solutions of the Cauchy problem (\ref{eq:relay}). Then we have
\begin{equation}
\label{eq:lip}
\|y_1-y_2\|_{C([0,T])}\leq \mathrm{e}^{\frac{T}{2}} \Big(|\xi_1-\xi_2|+k\|u_1-u_2\|_{\rL^2(0,T)}\Big).
\end{equation}
\end{lemma}
\begin{proof}
Let us denote by $y_1(t)$ and $y_2(t)$ the solutions of \eqref{eq:relay} corresponding to functions $u_1$, $u_2$ in $\rL^2(0,T)$
and initial conditions $\xi_1$ and $\xi_2$, respectively.
Then there must exist $q_i(t)\in \partial \chi_{[-1,1]}(y_i(t)),\ i=1,2 $  such that
$$
\dfrac{dy_i}{dt}(t)+q_i(t)=g_\rho(u_i(t)),\ i=1,2.
$$
Let us subtract the above equations for $i=1,2$ and then make the
scalar product by $y_1-y_2$. We get
$$
\frac{1}{2}\frac{d}{dt}|y_1(t)-y_2(t)|^2 \leq |g_\rho(u_1(t))-g_\rho(u_2(t))|\,|y_1(t)-y_2(t)|
$$
because $ (q_1(t)-q_2(t))(y_1(t)-y_2(t)) \geq 0, $
as $\partial \chi_{[-1,1]}$ is monotone.
By integrating from $0$ to $t$ we obtain
\begin{align*}
\frac{1}{2}|y_1(t)-y_2(t)|^2\leq \frac{1}{2}|\xi_1-\xi_2|^2
+\int_0^t |g_\rho(u_1(s))-g_\rho(u_2(s))|\,|y_1(s)-y_2(s)|\ds\\
 \leq \frac{1}{2}|\xi_1-\xi_2|^2+\left(\int_0^t |g_\rho(u_1(s))
 -g_\rho(u_2(s))|^2\ds\right)^{1/2}\,\left(\int_0^t |y_1(s)-y_2(s)|^2\ds\right)^{1/2}
\\
\leq \frac{1}{2}|\xi_1-\xi_2|^2+\frac{1}{2}\int_0^t |g_\rho(u_1(s))
-g_\rho(u_2(s))|^2\ds+\frac{1}{2}\int_0^t |y_1(s)-y_2(s)|^2\ds
\end {align*}
and then
\begin{align*}
|y_1(t)-y_2(t)|^2
\leq |\xi_1-\xi_2|^2+k^2\int_0^t |u_1(s)-u_2(s)|^2\ds+\int_0^t |y_1(s)-y_2(s)|^2\ds.
\end{align*}
By using the generalized Gronwall's inequality
(see, for instance, Lemma 6.2 in \cite{H80}) we deduce
$$
|y_1(t)-y_2(t)|^2\leq\,\mathrm{e}^t\left(|\xi_1-\xi_2|^2+k^2\int_0^t |u_1(s)-u_2(s)|^2\ds\right)
$$
and finally
\begin{align*}
\max_{t\in [0,T]} |y_1(t)-y_2(t)|\leq \mathrm{e}^{\frac{T}{2}}
\left(|\xi_1-\xi_2|^2+k^2\int_0^T |u_1(s)-u_2(s)|^2\ds\right)^{\frac{1}{2}}
\\
\leq \mathrm{e}^{\frac{T}{2}}\left(|\xi_1-\xi_2|+k\Big(\int_0^T |u_1(s)-u_2(s)|^2\ds\Big)^{\frac{1}{2}}\right)
\end{align*}
which finishes the proof.\qed
\end{proof}


\section{Dynamic Preisach model}\label{DPM}

In this section we introduce the dynamic Preisach operator and prove
some properties which are similar to the ones satisfied by the classical Preisach operator.
Given $\vpc>0$, let us consider again the \textit{Preisach triangle}
$\PTr:=\{ \rho=(\vpu,\vpd) \in \mathbb{R}^2: -\vpc\leq \vpu \leq
\vpd\leq \vpc\}$ (see Figure~\ref{fig:triangle_p} (left)) and the \textit{Preisach function}
$\pdisf\in \rL^1(\PTr)$ with $\pdisf>0$.  We denote by $Y$ the convex set of initial
configurations which can be defined by
$$
Y:=\left\{v\in L^1_{\pdisf}(\PTr) : |v(\rho)| \leq 1, \mbox{  a.e.  }  \rho \in \PTr  \right\}
$$
where
 $$L^1_{\pdisf}(\PTr):=\{\xi:\PTr \to \mathbb{R}
\mbox{ Lebesgue-measurable such that } \int_{\PTr}|\xi(\rho)|\pdisf(\rho)\dro <\infty\}$$
endowed with the norm
$$
\|\xi\|_{L^1_{\pdisf}(\PTr)}:=\int_{\PTr}|\xi(\rho)|\pdisf(\rho)\dro.
$$

Let us   define the dynamic Preisach operator (see \cite{B92})
\begin{align}
& \Preit:  \rL^2(0,T)\times Y  \longrightarrow  \rH^1(0,T), \nonumber\\
& (u,\xi) \longmapsto
[\Preit(u,\xi)](t)=\int_{\PTr}[\eta_{\rho}(u,\xi(\rho))](t)\pdisf(\rho)\dro,
\label{h_Fmu}
\end{align}
Notice that, if $\eta_\rho$ is replaced by $h_\rho$ given by \eqref{h_rho} then we obtain the classical, rate-independent, Preisach model.

Mathematical properties of the classical Preisach operator are well known
(see, for instance, \cite{Vi94}) but this is not the case for the dynamic model.
From Lemmas \ref{lemma:eta_order} and \ref{lemma:eta_cont}
we obtain the following properties of the dynamic Preisach model.
\begin{lemma}\label{lemma:Preisach_cont} The dynamic Preisach operator
$\Preit:  \mathrm{L}^2(0,T) \times Y \longrightarrow  \rH^1(0,T)$
is order preserving $(\textrm{cf}.~\eqref{eq:order_pre})$  and satisfies the
following properties:
\begin{itemize}
 \item It is Lipschitz-continuous; more precisely, for all
 $u_1,u_2\in \mathrm{L}^2(0,T)$ and $\xi_1,\xi_2\in Y$,
 \begin{multline*}
\|\Preit(u_1,\xin_1)-\Preit(u_2,\xin_2)\|_{C([0,T])}\\
\leq \mathrm{e}^{\frac{T}{2}}\left(\|\xi_1-\xi_2\|_{L^1_{\pdisf}(\PTr)}+
k\|u_1-u_2\|_{\mathrm{L}^2(0,T)}\int_{\PTr}\pdisf(\rho)\dro\right).
\end{multline*}
\item
It is bounded in the following sense: for all $v\in \mathrm{L}^2(0,T)$ and $\xi\in Y$,
\begin{equation}
\label{eq:cota}
\norm{\Preit(v,\xin)}{C([0,T])}\leq \displaystyle\int_{\PTr}\pdisf(\rho)\dro.
\end{equation}
\end{itemize}
\end{lemma}
\begin{proof} The order preserving property is a consequence of Lemma \ref{lemma:eta_order}
and the positivity of the Preisach function $p$.

The Lipschitz-continuity follows from Lemma~\ref{lemma:eta_cont}. Indeed, we have
\begin{align*}
\|\Preit(u_1,\xin_1)-\Preit(u_2,\xin_2)\|_{C[0,T]}&\leq
\int_{\PTr}\max_{t\in [0,T]}\left\{|[\eta_{\rho}(u_1,\xi_1(\rho))](t)
-[\eta_{\rho}(u_2,\xi_2(\rho))](t)|\right\}\pdisf(\rho)\dro
\\
&\leq \mathrm{e}^{\frac{T}{2}}\left(\int_{\PTr}|\xi_1(\rho)-\xi_2(\rho)
|\pdisf(\rho)\dro+ k\|u_1-u_2\|_{\mathrm{L}^2(0,T)}\int_{\PTr}\pdisf(\rho)\dro\right),
\end{align*}
from which the result follows.

Finally \eqref{eq:cota} is a consequence of the fact that
$|\eta_{\rho}(u,\xi(\rho))(t)|\leq 1 \ \forall t\in [0,T].$\qed
\end{proof}

As mentioned, in the classical Preisach model the relay $h_\rho$
only takes values $+1$ or $-1$. Thus, at each time $t\geq t^0$,
the Preisach triangle $\PTr$
is subdivided into two sets (one possibly empty):
\begin{align*}
S_u^{-}(t):=\left\{(\vpu,\vpd)\in\PTr: [\hvp(u,\xi)](t)=-1\right\},
\, S_u^{+}(t):=\left\{(\vpu,\vpd)\in\PTr: [\hvp(u,\xi)](t)=1\right\},
\end{align*}
and consequently
$$
[\Preic(u,\xi)](t)=\int_{S_u^{+}(t)}\pdisf(\rho)\dro-\int_{S_u^{-}(t)}\pdisf(\rho)\dro.
$$
However, this is not the case for the dynamic model where the
relays vary at finite rate between
$-1$ and $1$. Figure~\ref{fig:ex_3}  shows the
classical and dynamic relay configuration  with respect to the same input
$u$ depicted in Figure~\ref{fig:ex_3} (top left). For this example we have considered a  Preisach triangle
 characterized by $\rho_0 =300$,  a constant $k=50$ and the demagnetized state as initial condition (cf.~\eqref{eq:demag}).
\begin{figure}[h]
\centering
\includegraphics[width=0.46\textwidth]{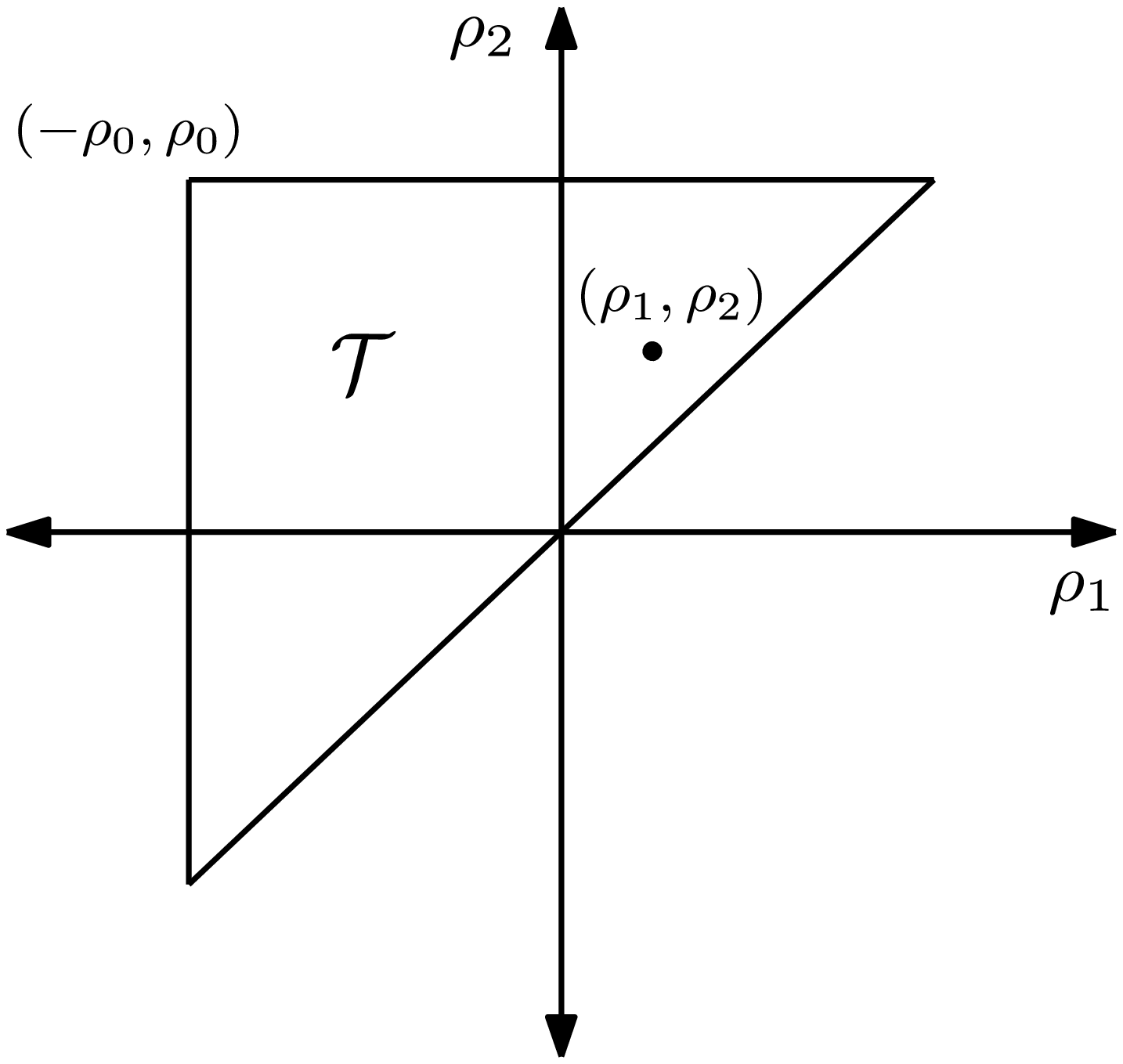}
\raisebox{0.15\height}{\includegraphics[width=0.46\textwidth]{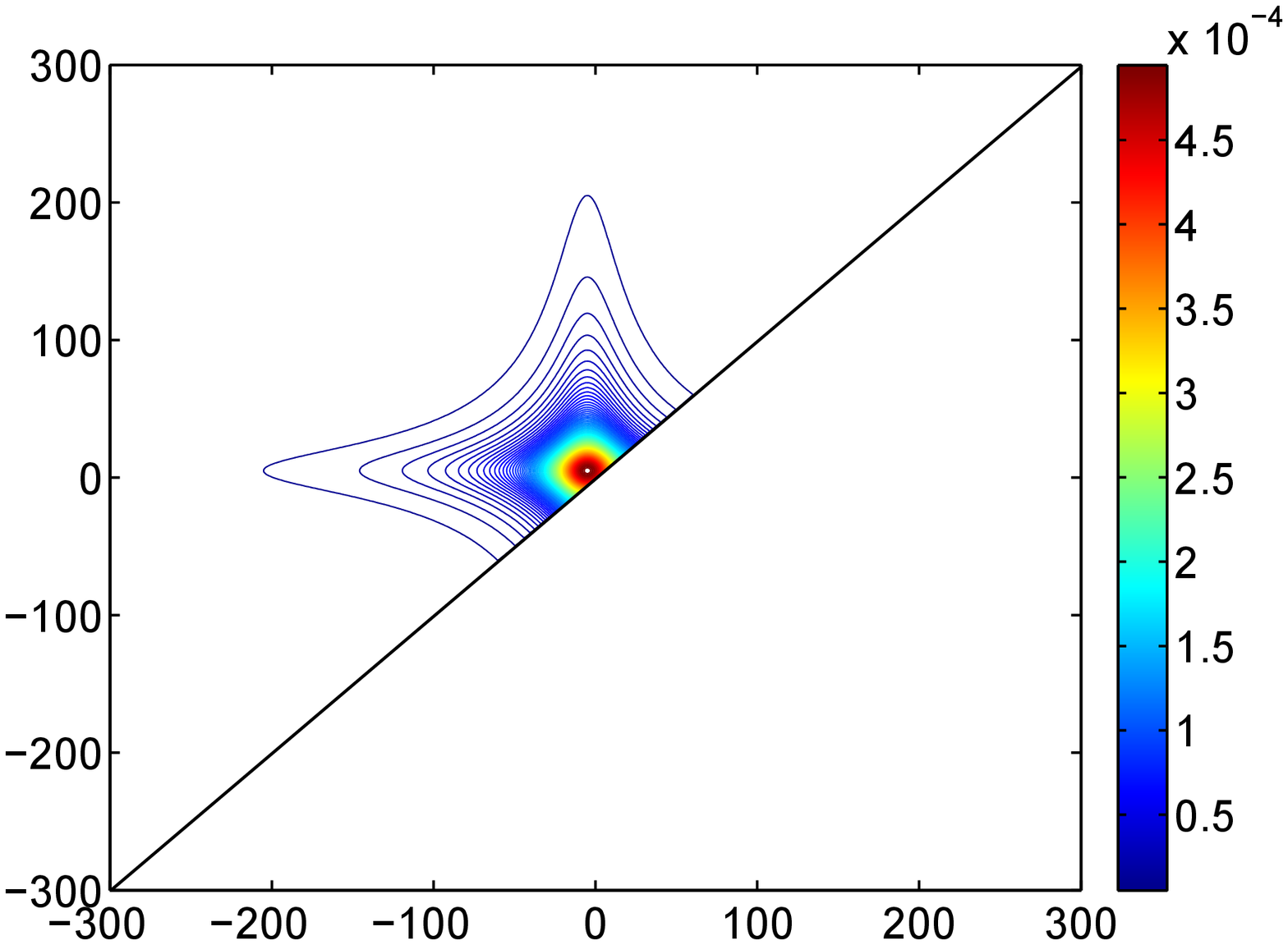}}
\caption{Preisach triangle (left) and Factorized-Lorentzian  distribution $p$ (right)
with $N=1/2000$, $\omega=5$ and $\gamma=4$.}
\label{fig:triangle_p}
\end{figure}

We also compute the dynamic Preisach operator $\Preit(u,\xi)$ for $u$ depicted in
Figure~\ref{fig:ex_3} (top left) for different $k$ values and frequencies. For the dynamic relay we have considered $k=50$ and the
Preisach function $p$ is given by the Factorized-Lorentzian
distribution~\cite{Bertotti1998} (see Figure~\ref{fig:triangle_p} (right)):
\begin{align}\label{facLor}
\pdisf(\vpu,\vpd):=N\left(1+\left(\dfrac{\vpd-\cta}{\gamma \cta}\right)^2\right)^{-1}
\left(1+\left(\dfrac{\vpu+\cta}{\gamma \cta}\right)^2\right)^{-1}
\end{align}
with $N=1/2000$, $\omega=5$ and $\gamma=4$. Figure~\ref{fig:BH_curve} shows
the dynamic curve $(u,\Preit(u,\xi))$ for different $k$ values and input velocities.

\begin{figure}[h]
\begin{center}
\includegraphics[width=0.44\textwidth, height=0.33\textwidth]{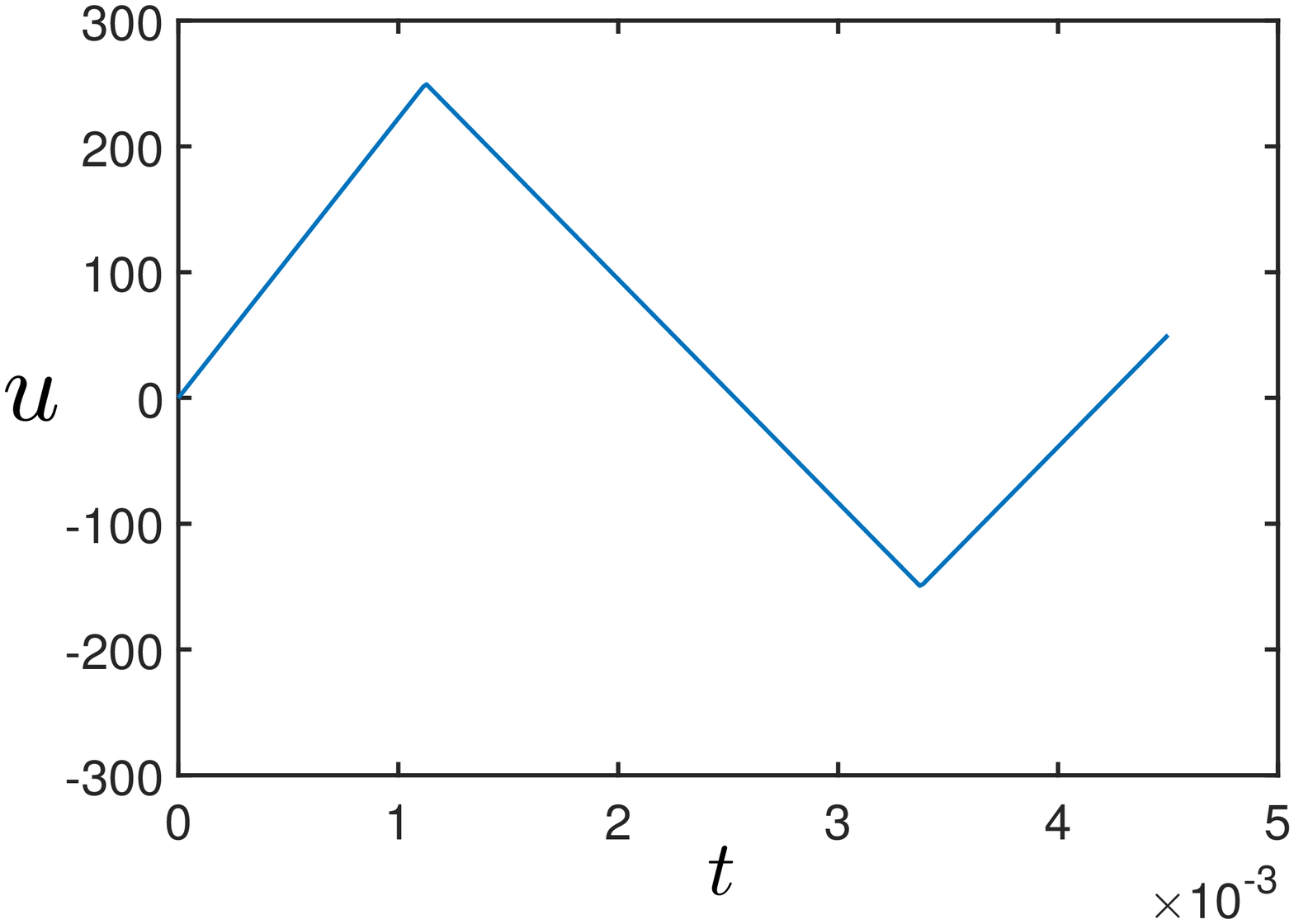}
\includegraphics[width=0.4\textwidth]{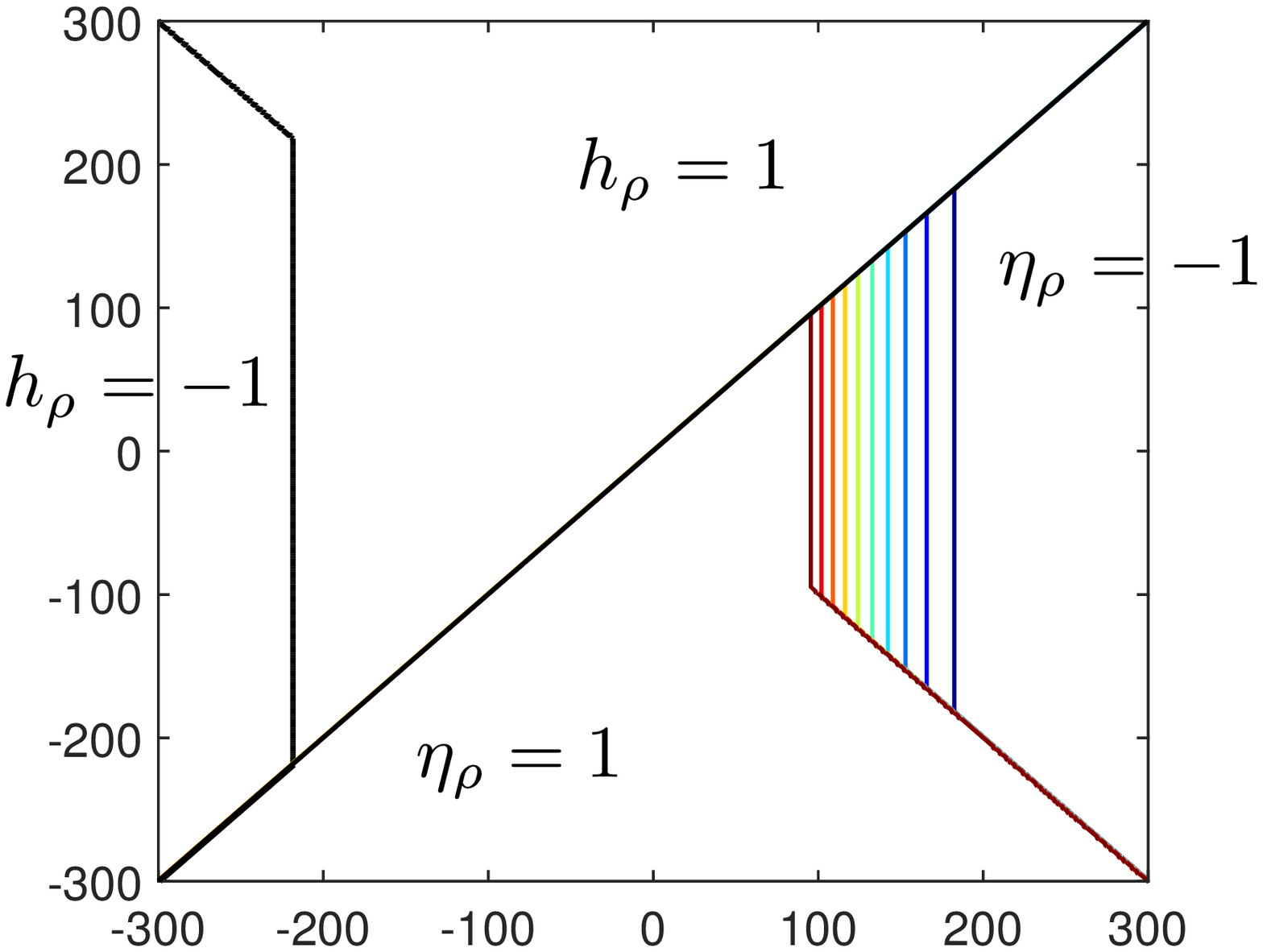}\\
\includegraphics[width=0.41\textwidth]{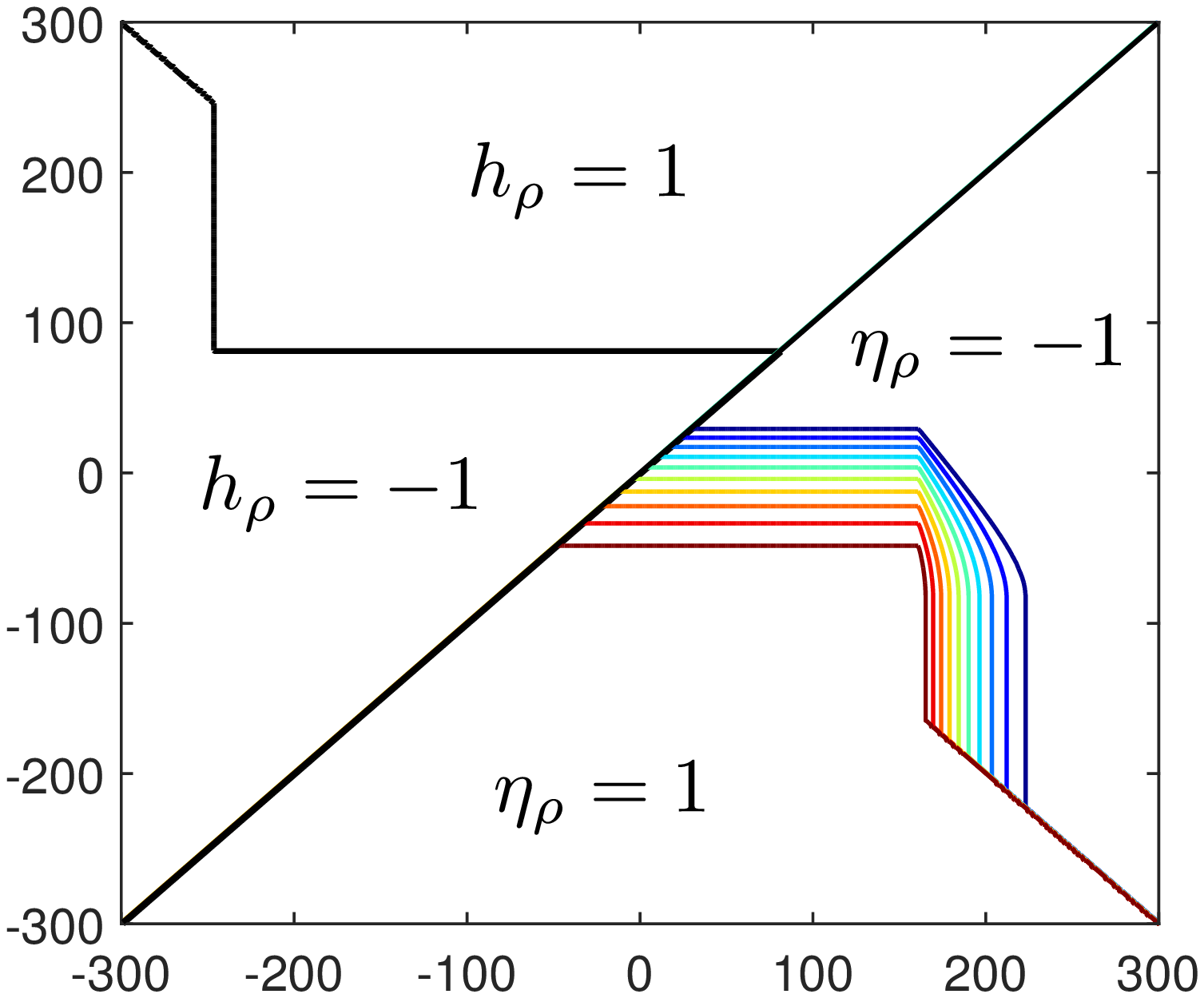}
\quad\includegraphics[width=0.4\textwidth]{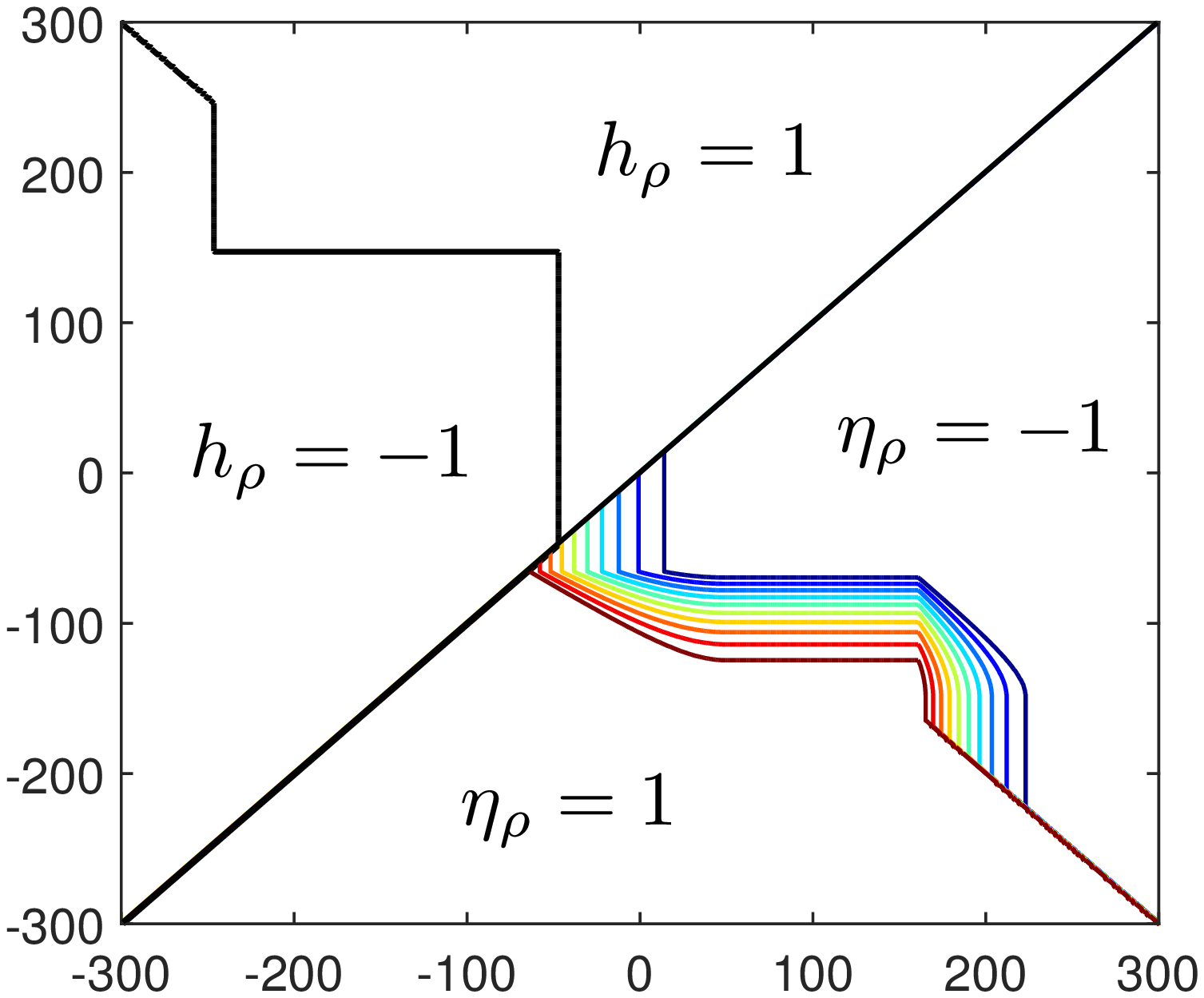}
\end{center}
\caption{Input function $u(t)$ (top left) defined in $[0,0.0045]$. The isolines represent the
corresponding dynamic relay values $\eta_{\rho}$ for $k=50$  and the classical relay
$h_{\rho}$  at $t=0.001$ (top right), $t=0.003$ (bottom left) and t= 0.0045 (bottom right). For compactness we have
considered the dynamic and classical relay in the same figure.}
\label{fig:ex_3}
\end{figure}
\begin{figure}[h]
\begin{center}
\includegraphics[width=0.45\textwidth]{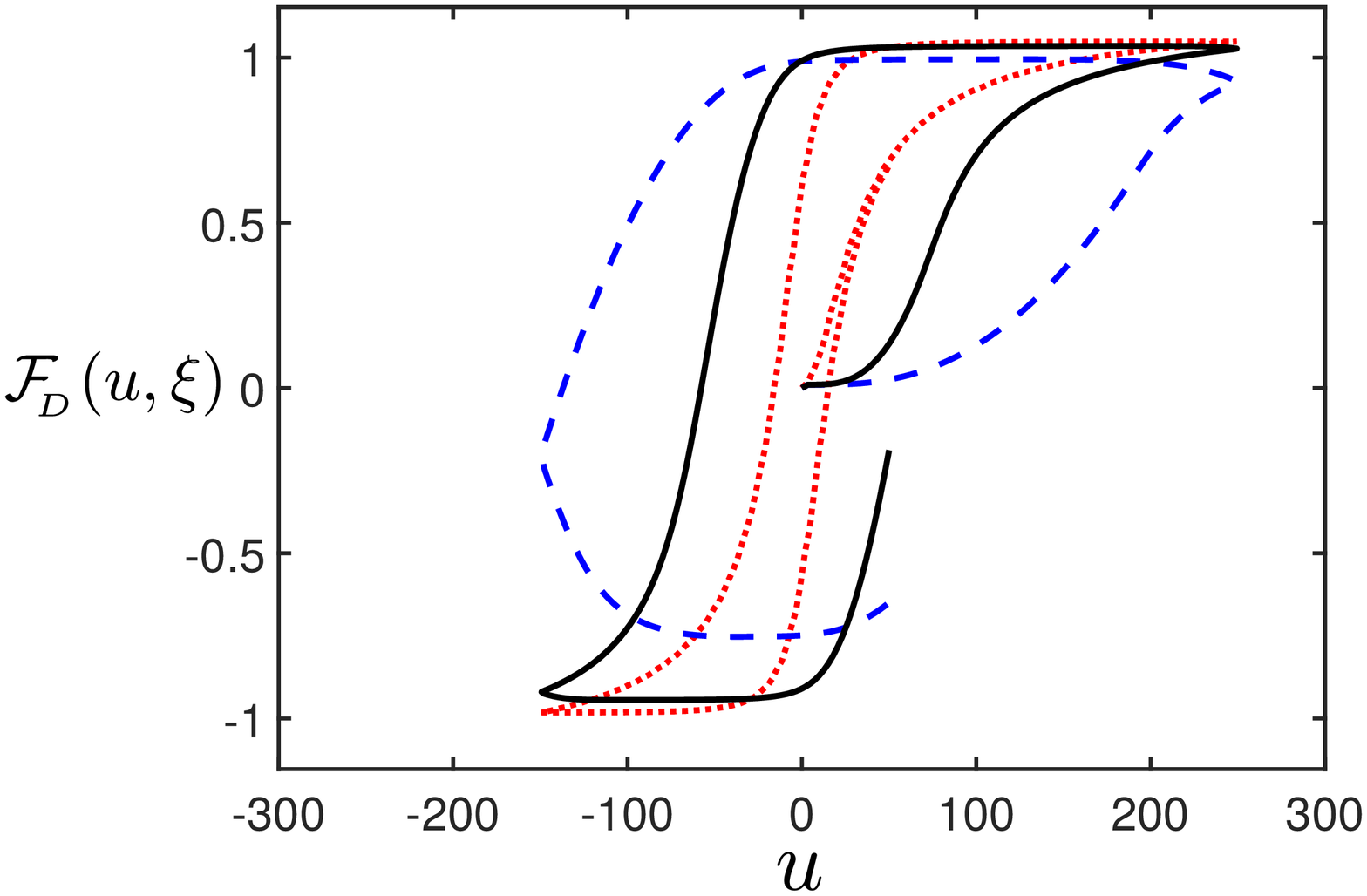}\quad
\includegraphics[width=0.45\textwidth]{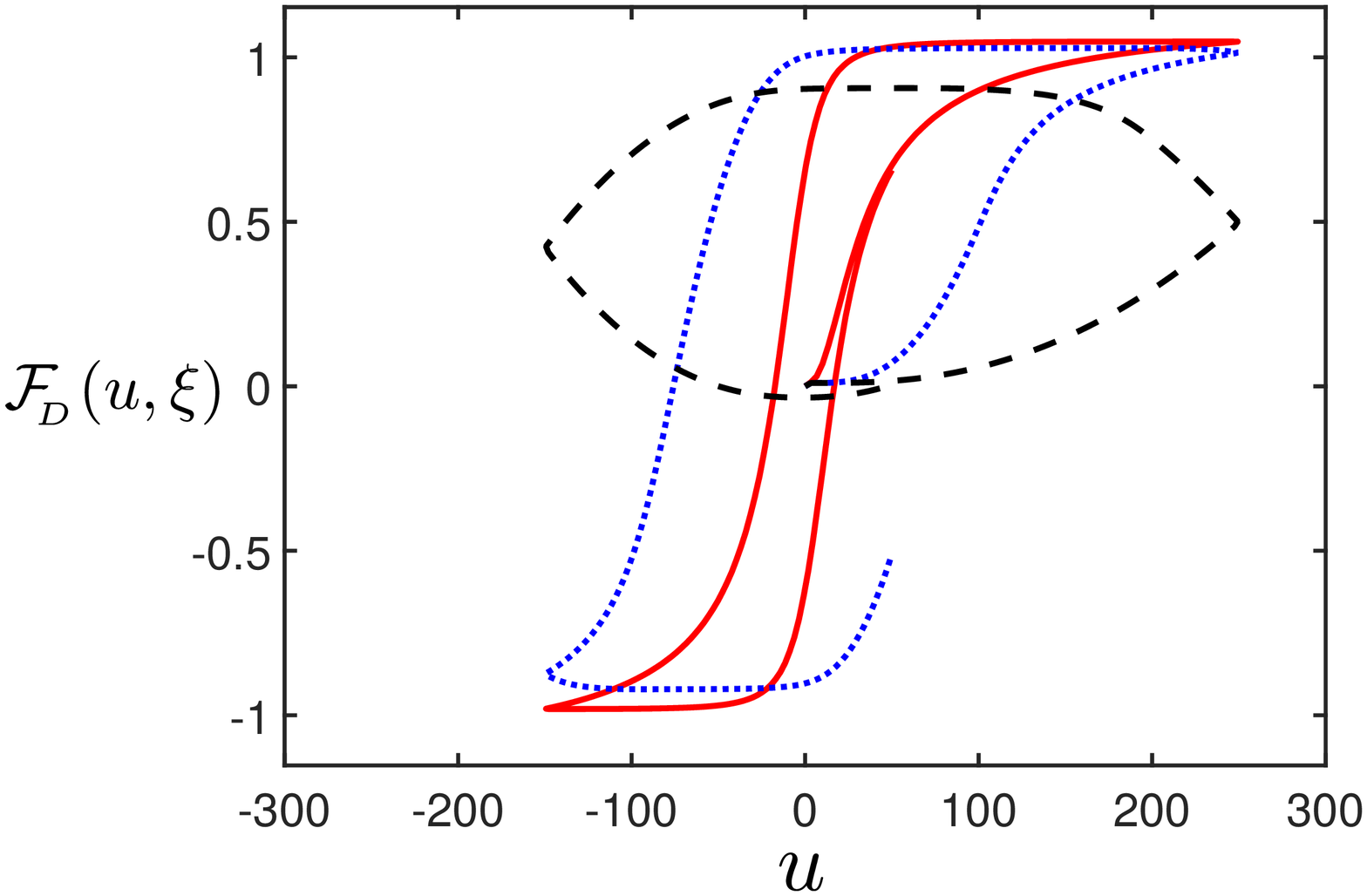}
\end{center}
\caption{ Dynamic Preisach hysteresis curve $\Preit$ for a
Factorized-Lorentzian  distribution, diferent $k$ values and input $u(t)$ depicted on Figure~\ref{fig:ex_3} (top left).
Left: input $u$ defined in
$[0,0.0045]$ and slope values  $k=25$ (dashed  line), $k=200$ (solid line) and
 $k=10^8$ (dotted line). Right: slope value $k=100$ and input function $u(t)$
 defined in $[0,4.5]$ (solid line), $[0,0.0045]$ (dotted line) and  $[0,0.00045]$  (dashed  line).
}
\label{fig:BH_curve}
\end{figure}
Finally, since we are interested in the mathematical analysis and computation of
distributed electromagnetic models (also called field models),
following \cite{Vi94} we introduce a space-time
dependent operator. Given  a time dependent input field
$u(x,\cdot)\in \rL^2(0,T)$ and an initial state field $\xi(x)\in Y,$
we define a space and time
dependent hysteresis operator $\hists:\rL^2(0,T;\rL^2(\Om))\times \rL^2(\Omega;Y)\to C(0,T;\rL^2(\Om))$ as follows:
\begin{equation}\label{eq:space_hyst}
[\hists(u,\xi)](x,t):=[\Preit(u(x,\cdot),\xi(x))](t)\quad ,\mbox{  a.e. in } [0,T]\times\Omega
\end{equation}
where $\rL^2(\Omega;Y) $ is a subset of  $\rL^2(\Omega;\rL^1_{\pdisf}(\PTr))$,
the space of all function $v:\Om\to \rL^1_{\pdisf}(\PTr)$ such that
$$
\norm{v}{\rL^2(\Omega;\rL^1_{\pdisf}(\PTr))}:=\left(\int_{\Omega}\norm{u }{\rL^1_{\pdisf}(\PTr)}^2\right) ^{1/2}<\infty.
$$
Let us emphasize that operator $\mathcal{F}$ is local in $x$ but non-local in $t$.
We end this section with  the following properties of
 dynamic operator $\hists$ which follow from Lemma~\ref{lemma:Preisach_cont}.
\begin{lemma}\label{lemma:F_prop}  The dynamic Preisach operator
$\hists:\rL^2(0,T;\rL^2(\Om))\times\rL^2(\Om;Y)\to C(0,T;\rL^2(\Om))$ is  uniformly bounded, order
preserving $(\textrm{cf}.~\eqref{eq:order_pre})$ and Lipschitz-continuous in the following sense: there
exists $C$ depending on $k$ such that, for all $u_1,u_2 \in \rL^2(0,T;\rL^2(\Om))$ and $\xi_1,\xi_2 \in \rL^2(\Om;Y)$
\begin{equation*}
\|\hists (u_1,\xi_1)-\hists(u_2,\xi_2)\|_{C(0,T;\rL^2(\Om)))}\leq
C\left(\norm{\xi_1-\xi_2}{\rL^2(\Om;\rL^1_{\pdisf}(\PTr))}+\norm{u_1-u_2}{\rL^2(0,T;\rL^2(\Om))}\right).
\end{equation*}
\end{lemma}


\section{Parabolic problem with dynamic hysteresis}\label{PWH}

In this section we introduce a parabolic problem with dynamic
hysteresis for which we state an existence result by using the properties proved on the previous section.

 Let $T>0$ and $\Omega \in \R^d, d=2,3$
be a bounded domain with smooth boundary $\Gamma=\partial \Omega$.
Let $V\subset  H $ be two Hilbert spaces of scalar functions defined in $\Omega$ with continuous, dense,
compact  embedding. Then we have $V\subset H \equiv H '\subset V'$.
We consider a mapping $ a:(0,T)\times V\times V\to \R$ such that $ a(t,\cdot,\cdot)$ is bilinear a.e. $ t\in (0,T)$.
 We are interested in the
mathematical analysis of the following parabolic problem:
find $u \in\rL^2(0,T;V)\cap\rL^{\infty}(0,T; H )$ with $\partial_t u \in\rL^2(0,T;V')$
and $w\in \rL^2(0,T; H )$ with $\partial_t w\in \rL^2(0,T;V')$, such that
\begin{subequations}\label{VF:main}
\begin{align}
\label{VF:a}
\left\langle\partial_t u+\partial_t w, v
\right\rangle_{V,V'}+ a(t,u,v)
&=\langle f,v\rangle_{V,V'}\quad \forall v \in V,\quad\mbox{a.e. in  } (0,T],\\
\label{VF:b}
w&= \hists(u,\xi) \qquad \mbox{in  } \Omega\times[0,T],\\
\label{VF:c}
(u+w)(0)&=u_0+w_0\,\,\,\,\quad \mbox{in  } \Omega.
\end{align}
\end{subequations}
Here, operator $\hists$ is defined by \eqref{eq:space_hyst}. We introduce the following assumptions that will be used to prove
the existence of a solution to \eqref{VF:main}:
\begin{itemize}
\item[H.1] $ a(\cdot,u,v)$ is a continuous form in
$V\times V$. Moreover,
 it is Lipschitz continuous in $t$ and  satisfies the G{\aa}rding's inequality
 \begin{equation} \label{eq:a_coerc}
 a(t,v,v)+\lambda  \|G\|_{{ H }}^2 \geq \gamma \|G\|_{V}^2\quad
 \forall v\in V, \quad \forall t \in [0,T],
\end{equation}
for some constants  $\lambda,\gamma \geq 0$.
\item[H.2] $f$ belongs to $\rH^1(0,T;V')$, $u_0\in V$ and  $w_0:=\hists(u_0)\in  H$.
\item[H.3] 
 For a fixed initial state
$\xi:\Om\to Y$,
 the mapping $\hists:\rL^2(0,T;H)\to C(0,T;H)$
is well defined  and Lipschitz continuous in the following sense: there exits $C>0$ such  that:
\begin{equation*}
\|\hists (u,\xi)-\hists(v,\xi)\|_{C(0,T;H))}\leq C\norm{u-v}{\rL^2(0,T;H)}.
\end{equation*}
\end{itemize}

The next result shows the existence of solution to problem
\eqref{VF:main}. The proof is carried out through three different
steps: time discretization, a priori estimates and passage to the limit
by using compactness (cf.~H.3). This approximation procedure is often
used in the analysis of equations that include a memory operator since
at any time-step we solve a stationary problem in which this operator
is reduced to a standard nonlinear mapping (see, for instance, \cite{Vi94}).
\begin{theorem}\label{thm:existence}
Let us assume H.1, H.2 and H.3 hold true.
Then, problem~\eqref{VF:main} has a solution.
\end{theorem}

\begin{proof} Let us fix $m\in\N$ and set $\Dt:=T/m$. Now, for
$n=1,\ldots,m$, we define $t^n:=n\Dt$.
The time discretization of problem~\eqref{VF:main} based on backward Euler's
scheme reads as follows: Given $u^0=u_0$ and $w^0=w_0$ in $\Omega$,
find $u^n\in V$ and $w^n\in H$, $n=1,\ldots,m$, satisfying
\begin{align}\label{eq:disc_pb}
\left(u^n+w^n,v\right)_{H,H}+ \Dt a(t^n,u^n,v)
&=\Dt \langle f^n,v\rangle_{ V , V '}+(u^{n-1}+w^{n-1},v),\\
w^n&= [\hists(u_{\Dt^n},\xin)](t^n)\qquad\quad\mbox{in  } \Omega,
\label{eq:disc_pb2}
\end{align}
for all $v \in V$, where  $u_{\Dt^n}$ is the piecewise linear in time
interpolant of $\{u^i\}_{i=0}^n$. In order to study the time-discrete problem
 we introduce an operator $\FndD^n:H \longrightarrow H$ as follows: 
\[
F^n(w):=[\hists(\Lambda^n(w),\xin)](t^n)\quad \forall w\in H,
\]
where $\Lambda^n: H \rightarrow \mathrm{H}^1(0,t^n;H)$ is defined as follows:
for $w\in H$, $\Lambda^n(w)$ is the continuous piecewise linear function in time
such that $\Lambda^n(w)(t^i)=u^i,\ i=0,\ldots,n-1,$ and $\Lambda^n(w)(t^n)=w$.
From Lemma~\ref{lemma:F_prop} it follows
that $F^n$ is Lipschitz continuous in $H$, uniformly bounded and,
from the order preservation property, we have $ \forall w_1,w_2\in H$
 \begin{equation}\label{eq:mon_F}
\left(F^n(w_1)-F^n(w_2)\right)(w_1-w_2)\geq 0 \qquad \mbox{a.e. in } \Omega.
 \end{equation}
Thus, from H.1  it follows that \eqref{eq:disc_pb} has a unique solution (see, for instance, \cite{R05}).
The next step is to prove an a priori estimate for the solution of \eqref{eq:disc_pb}.
 Let us apply \eqref{eq:disc_pb} to
$v=u^n-u^{n-1}$. For $n=1,\ldots,m$ we obtain
\begin{align}\label{h_n_eqd}
\Dt \left\|\dfrac{u^n-u^{n-1}}{{\Dt}}\right\|_{H}^2
+\left(\dfrac{w^n-w^{n-1}}{\Dt},u^n-u^{n-1}\right)_{H,H}+
a(t^n,u^n,u^n-u^{n-1})\nonumber\\
=\langle f^n,u^n-u^{n-1}\rangle_{ V , V '}.
\end{align}
It is well known that, when the classical Preisach model is considered,
the second term on the left hand side of the previous equality is positive.
 This is a consequence of the order preservation
property and the fact that $w^{n-1}=F^{n}(u^{n-1})$  in the rate-independent setting. However, this does
not hold true for  the dynamic Preisach operator. In order to estimate this
term, from \eqref{eq:mon_F} we first notice that, a.e. in $\Omega$,
\begin{align}
  (w^n-w^{n-1})(u^n-u^{n-1})&=(F^n(u^{n})-F^n(u^{n-1}))(u^n-u^{n-1})\nonumber\\
  & -(w^{n-1}-F^n(u^{n-1}))(u^n-u^{n-1})  \nonumber \\
  &\geq - ([\hists(u_{\Dt^{n}},\xin)](t^{n-1})-F^n(u^{n-1})) (u^n-u^{n-1})   \nonumber\\
  &=  ([\hists(\widetilde{u}_{\Dt^{n}},\xin)](t^{n})
  -[\hists(\widetilde{u}_{\Dt^{n}},\xin)](t^{n-1})) (u^n-u^{n-1})
  \label{eq:F_diff1}
\end{align}
where  $\widetilde{u}_{\Dt^n}$ is the continuous piecewise linear in
time function such that $\widetilde{u}_{\Dt^n}(t^i)=u^i$, $i=0,\ldots,n-1$ and
$\widetilde{u}_{\Dt^n}(t^n)=u^{n-1}$ a.e. in $\Omega$. Moreover,
from \eqref{h_Fmu}  and \eqref{eq:space_hyst} it follows that, a.e. in $\Omega$,
\begin{multline}\label{eq:F_diff2}
[\hists(\widetilde{u}_{\Dt^n},\xin)](t^n)-
[\hists(\widetilde{u}_{\Dt^n},\xin)](t^{n-1})=\int_{t^{n-1}}^{t^n}\partial_t[\hists(\widetilde{u}_{\Dt^n},\xin)](s) ds\\
=\int_{t^{n-1}}^{t^n}\int_{\PTr}\partial_t\dhvp(\widetilde{u}_{\Dt^n},\xin)(s) p(\rho) d\rho\, ds
\leq C\Dt\left(|u^{n-1}|+1\right)
\end{multline}
where latter inequality follows from \eqref{eq:carq1} and the fact that
$\widetilde{u}_{\Dt^n}=u^{n-1}$ in $[t^{n-1},t^n]$. Here $C$ depends on $k$ but it is independent of $\Dt$.
On the
other hand, in order to estimate the last term on the left-hand side of
 \eqref{h_n_eqd} we use the identity $2(p-q)p= p^2+(p-q)^2-q^2$
 and the Lipschitz continuity of
 $a(\cdot,v,w):(0,T)\to \R, \forall v,w\in V$ to
obtain that
\begin{align}\label{eq:a_est}
 2a(t^n, u^n,u^n-u^{n-1})\geq&
   a(t^n, u^n,u^n)-
 a(t^{n},u^{n-1},u^{n-1})   \nonumber
\\
\geq &
  a(t^n, u^n,u^n)-
 a(t^{n-1},u^{n-1},u^{n-1})
-C\Dt\norm{u^{n-1}}{ V }^2.
\end{align}
Summing up \eqref{h_n_eqd} for $n=1,\ldots,l$ with $l\in \{1,\ldots,m\}$,
 from \eqref{eq:F_diff1}--\eqref{eq:a_est} we obtain
\begin{multline*}
\sum_{n=1}^{l}\Dt \norm{\dfrac{u^{n}-u^{n-1}}{\Dt}}{H}^2
+\dfrac{1}{2}a(t^l,u^l,u^l)\\
\leq \dfrac{1}{2}a(t^0,u_0,u_0)+\sum_{n=1}^{l}C
\Dt\norm{u^{n-1}}{ V }^2
+\sum_{n=1}^{l}\langle f^n,u^n-u^{n-1}\rangle_{ V , V '}+C.
\end{multline*}
%
%
%
%
By proceeding as in \cite{BGRV2014}  it follows that, for $l=1,\ldots,m$
\begin{equation}\label{eq:apriori}
\Dt\sum_{n=1}^{l}\norm{\dfrac{w^{n}-w^{n-1}}{\Dt}}{H}^2+\Dt\sum_{n=1}^{l}\norm{\dfrac{u^{n}-u^{n-1}}{\Dt}}{H}^2
+\norm{u^l}{ V }^2\leq C,\qquad
\end{equation}
%

Finally, let us define $u_{\Dt}:[0,T]\to V $, $w_{\Dt}:[0,T]\to H $ as
the continuous piecewise linear in time interpolants of $\{w^n\}_{n=0}^m$ and
$\{u^n\}_{n=0}^m$, respectively. We also introduce the
step function $\overline{u}_{\Dt}:[0,T]\to V $ by
\begin{align*}
\overline{u}_{\Dt}(t^0):=u_0;\quad \overline{u}_{\Dt}(t):=u^n, \quad t\in (t^{n-1},t^n],\quad
i=n,\ldots,m,
\end{align*}
and define the step functions  $\overline{a}(t)$ and $\overline{f}$ in a similar way.

Using the above notation we rewrite equation \eqref{eq:disc_pb} as follows:
\begin{eqnarray}\label{h_pp}
\left(\partial_t u_{\Dt}+\partial_t w_{\Dt},v\right)_{H,H}+ \overline{a}(t,\overline{u}_{\Dt},v)
=\langle \overline{f},v\rangle_{V,V'}\quad\forall v \in  V,\,\,\mbox{a.e. in  } (0,T].
\end{eqnarray}
{}From \eqref{eq:apriori} we deduce that there exists $C>0$ such that
\begin{align*}
\norm{w_{\Dt}}{\rL^{\infty}(0,T; H)}+\norm{\dfrac{\partial w_{\Dt}}{
\partial t}}{\rL^{2}(0,T; V ')}
+\norm{u_{\Dt}}{\rH^1(0,T; H )\cap \rL^\infty(0,T; V )}+
\norm{\overline{u}_{\Dt}}{\rL^\infty(0,T; V )}\leq C.
\end{align*}
Thus, there exists $u,w$ such that $w_{\Dt}\to w$ and $u_{\Dt}\to u$ weakly in the corresponding spaces.
Passing to the limit  in \eqref{h_pp} we obtain
\begin{align*}
\langle\partial_t u+\partial_t w,v\rangle_{V,V'}+ a(t,u,v)
=\langle f,v\rangle_{V,V'}\qquad\quad\forall v \in  V ,\quad\mbox{a.e. in  } (0,T].
\end{align*}
The next step is to prove that $w=\hists(u,\xin)$. This equality follows from the
compact embedding of $\rH^1(0,T; H )\cap \rL^2(0,T; V )$ in $\rL^2(0,T; H )$  (see, for instance, \cite[Th.~51]{JLL})
and the Lipschitz continuity of $\hists$ (see H.3).\qed
\end{proof}

\begin{remark} The previous result applies, for instance, to the following parabolic equation:
\begin{align*}
\left\langle\partial_t u+\partial_t w, v
\right\rangle_{ V , V '}+ (\nabla u,\nabla v)_{ H , H }
&=\langle f,v\rangle_{ V , V '}\quad \forall v \in  V ,\quad\mbox{a.e. in  } (0,T],\\
w&= \hists(u,\xi) \qquad \mbox{in  } \Omega\times[0,T],\\
(u+w)(0)&=u_0+w_0 \,\,\,\,\quad \mbox{in  } \Omega,
\end{align*}
where $ V =\rH_0^1(\Omega)$ and $ H =\rL^2(\Omega)$ (see Chapter IX in \cite{Vi94} for the rate-independent case).

It can also be applied to the axisymmetric eddy current model
\begin{align*}
\left\langle \partial_t w,rv\right\rangle_{ V , V '}\!\! + \int_\Omega\dfrac{1}{\sigma
r}\left(\dfrac{\partial(ru)}{\partial r}\dfrac{\partial
(rv)}{\partial r}+\dfrac{\partial(ru)}{\partial z}\dfrac{\partial
(rv)}{\partial r}\right)\drdz
&= b'(t) \left(rv\right)|_{\Gamma}\quad
\forall v \in  V,\\
w&= u+\hists(u,\xin)\quad \mbox{in }
\Om\times(0,T),\\
w(0)&= w_0 \quad \mbox{in  }\Om.
\end{align*}
where  $ V =\big\{v\in \rL^2_r(\Omega): \pdr(rv)\in \rL^2_{1/r}(\Om),\pdz v\in
\LrO \mbox{ and } rv|_{\Gamma} \mbox{ is constant}\big\}$ and $ H =\rL^2_r(\Omega)$.
Here $u$ represents the magnetic field,  $w$ is the magnetic induction, $b$ is
the magnetic flux and $\sigma$ is the electrical conductivity (see \cite{VKDM}).
\end{remark}


\section{Numerical approximation and examples}\label{Num}

The aim of this section is twofold: first, to introduce and analyze the
convergence properties of a numerical scheme to approximate a partial differential equation (PDE) with hysteresis.
For this purpose, we apply the numerical approximation to a test problem where several successively refined
meshes and time-steps have been considered. Second, to illustrate the behavior of the
numerical solution for different configurations of the dynamic Preisach model.

With this end, let us consider the following weak formulation:
find $u \in\rH^1(0,T;\rL^2(\Omega))\cap\rL^{\infty}(0,T;\rH^1(\Omega))$
and $w\in \rL^2(0,T; \rL^2(\Omega) )$ with $\partial_t w\in \rL^2(0,T;\rH^1(\Omega)')$, such that
\begin{subequations}\label{VF:num}
\begin{align}
\label{VF:num_a}
\left\langle \partial_t  u+\partial_t w, v
\right\rangle_{\rH^1(\Omega),\rH^1(\Omega)'}+ \sigma^{-1}(\nabla u,\nabla v)
&=0\quad\qquad\quad\forall v \in \rH_0^1(\Omega),\\
\label{VF:num_b}
w&= \hists(u,\xi)\quad\mbox{ in  } \Omega\times[0,T],\\
\label{VF:num_c}
u&= g \qquad\qquad\mbox{on  } \partial \Omega\times[0,T],\\
\label{VF:num_d}
(u+w)(0)&=0\qquad\qquad\mbox{in  } \Omega,
\end{align}
\end{subequations}
where $\Omega \in \R^2$, $\sigma > 0 $ and $g\in \rH^1(0,T;\rH^{1/2}(\Gamma))$.
This problem arises, for instance, in the computation of
2D electromagnetic field in a cross-section of laminated media  (see \cite{VKDM}).
This field is important for the evaluation of the electromagnetic losses.

Notice that this problem does not lie exactly in the same framework as the previous one  because of Dirichlet,  instead of Neumann, boundary condition;
nevertheless, the existence of solution can be proven with the same techniques as those in \cite{BGRV2014}.

Next, we introduce a fully discrete approximation of problem~\eqref{VF:num}.
From now on we will assume that $\Omega$ is a convex polygon. We associate a family of partitions $\{\mathcal{T}_h \}_{h > 0}$ of
$\Om$ into triangles, where $h$ denotes the mesh size. Let $\calV_h$ be
the space of continuous piecewise linear finite elements.
We also consider the finite-dimensional space $\calV^0_h:=\calV_h\cap \rH_0^1(\Omega)$
and denote by $\calV_h(\Gamma)$ the space of traces
on $\Gamma$ of functions in $\calV_h$. We introduce a uniform partition
$\{t^i := i\Delta t, i = 0, \ldots , m\}$ of $[0, T]$, with time step $t := T/m, m \in \mathbb{N}$.
By using the above finite element space
for space discretization and the backward Euler
scheme  for time discretization, we are led to the following Galerkin
approximation  of problem~\eqref{VF:num}:
Given $u_h^0=w_h^0=0$ in $\Omega$,
find $u_h^n\in \calV_h$ and $w_h^n\in \calV_h$, $n=1,\ldots,m$, satisfying
\begin{subequations}\label{VF:num_dis}
\begin{align}
\left( u_h^n+w_h^n,v_h\right)+ \Dt \sigma^{-1}(\nabla u_h^n,\nabla v_h)
&=(u_h^{n-1}+w_h^{n-1},v_h)\quad \forall v_h \in \calV^0_h,\label{VFd_1}\\
w_h^n&= [\hists(u^h_{\Dt^n},\xin)](t^n)\quad\mbox{in  } \Omega,\label{VFd_2}\\
u_h^{n}&=g^n_h\quad \mbox{in  } \Gamma,\label{VFd_3}
\end{align}
\end{subequations}
where $g^n_h\in\calV_h(\Gamma)$ is a convenient approximation of
$g(t^n), n=1, \ldots , m$  and $u^h_{\Dt^n}$ is the piecewise linear in time
interpolant of $\{u_h^i\}_{i=0}^n$.

At each time step of the above algorithm, we must solve a non-linear problem.
With this purpose, and given the history dependence of the nonlinear operator,
we have considered a Newton-like method.
To complete the proposed numerical scheme, a particular hysteresis
operator must be considered (cf. \eqref{VFd_2}). In view of applications we have considered the
dynamic Preisach model described in Section~\ref{DPM} characterized by the
Factorized-Lorentzian distribution \eqref{facLor} (see Figure~\ref{fig:triangle_p} (right))
and different values of slopes $k$ (cf. \eqref{eq:carq1}).

\subsection{Convergence analysis}
Given the difficulties related with the numerical analysis of the problem, we
estimate experimentally the order
of convergence of the scheme presented in the previous section.
With this aim, we have solved
 problem \eqref{VF:num_dis} in a square domain $\Omega = [0, 0.02]^2$ along the
time interval $[0, 0.01]$, with $\sigma=100$,  non-homogeneous Dirichlet boundary condition
$g(x, y, t) = 200 \sin(2 \pi t/0.01)$ and the dynamic Preisach model with slope $k=10$.

Since there is no analytical solution to this problem, we asses the performance of the
method by comparing the computed results with those obtained with a very fine uniform
mesh of size $h_0/15$ and time step $\Delta t_0/512$. The solution to this problem is
taken as the ``exact'' solution $u$.
The method has been used on several successively refined meshes
chosen in a convenient way in order to analyze convergence.  We denote by $h_0=0.0033$ the corresponding
mesh size and we have taken as coarser time step $\Delta t_0 =0.01$. The rest of the meshes
are uniform refinements of this one. The numerical approximations are compared with the
``exact'' solution by computing the percentage error for $u$
in both a discrete $\rL^2(0, T ; \rL^2(\Om))$-norm and
a discrete $\rL^2(0, T ; \rH^1(\Om))$-semi norm, respectively.

Figure~\ref{fig:error_bh} (left) shows the percentage error for $u$ versus the mesh-size $h$ for a fixed time-step.
We observe a linear order of convergence in norm $\rL^2(0, T ; \rH^1(\Om))$  and a
quadratic order in $\rL^2(0, T ; \rL^2(\Om))$. A similar behavior is also observed when
the rate-independent Preisach model is considered (see \cite{BDGV2017}).
\begin{figure}
\centering
\includegraphics*[width=0.5\textwidth]{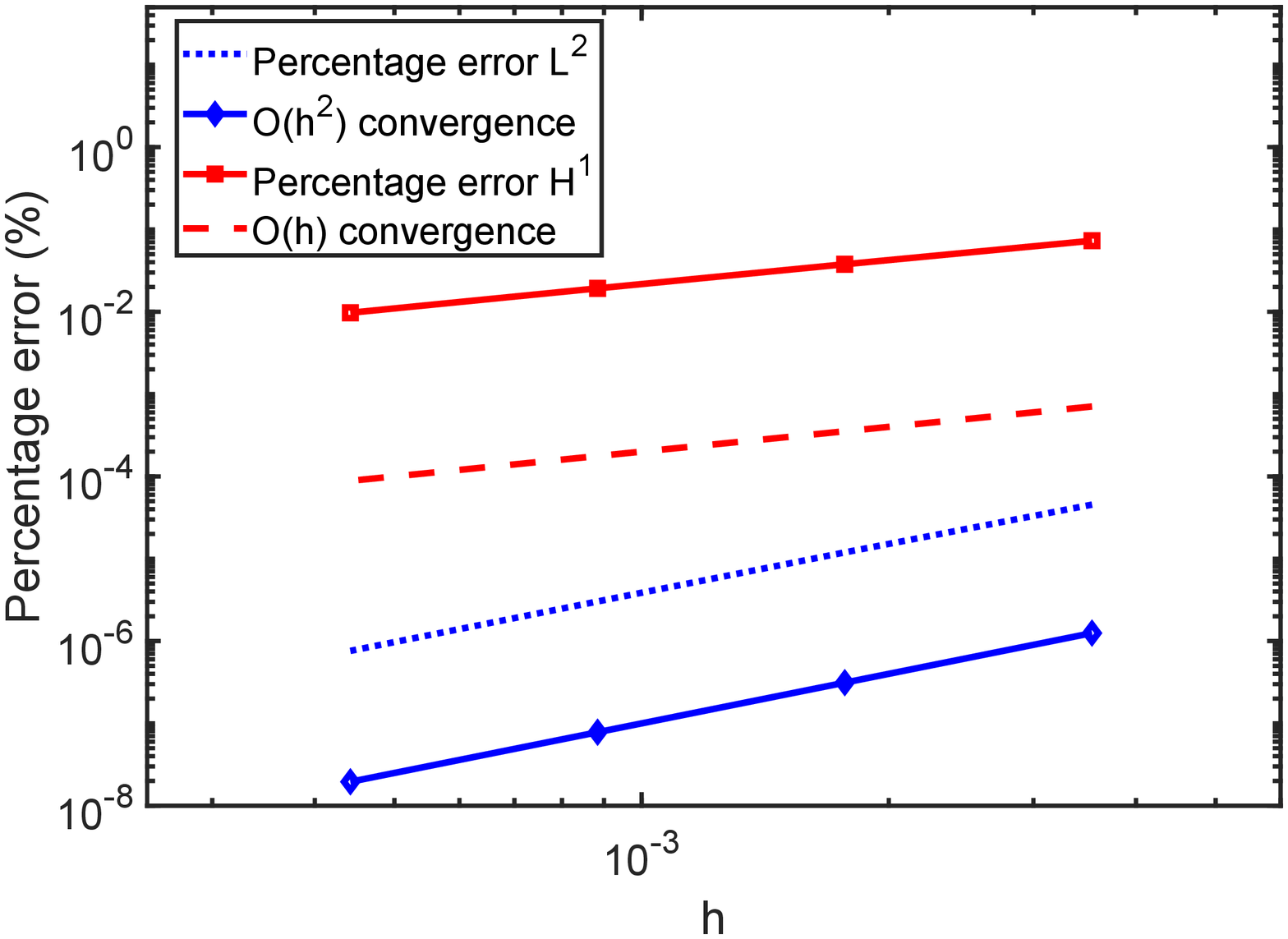}
\includegraphics*[width=0.49\textwidth]{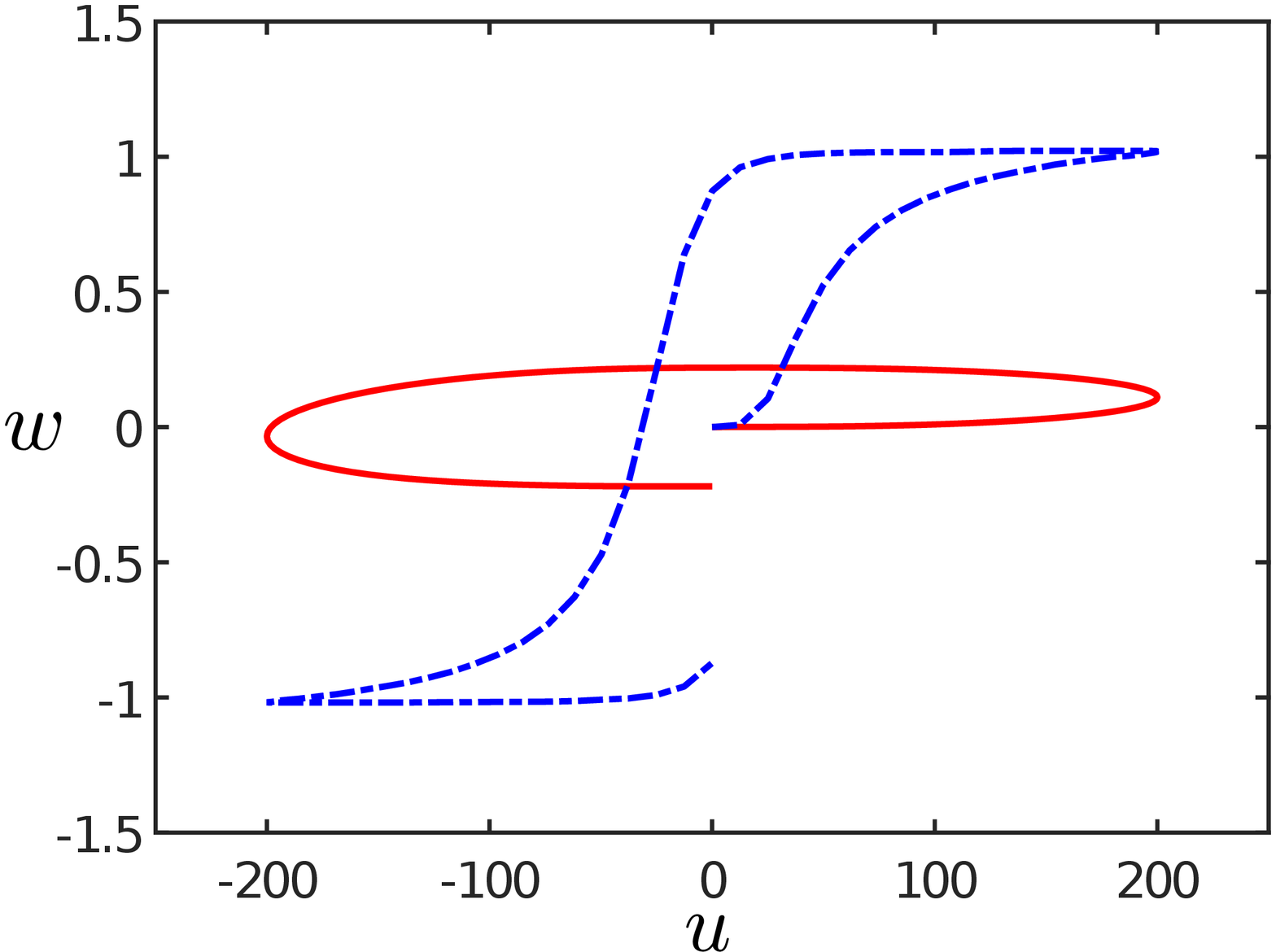}
\caption{Left.  Percentage errors in norms  and  versus the mesh-size $h$ for a fixed time-step $\Delta t/128$ (log-log scale).
Right. Curve $w-u$ on a boundary point for $k=10^{3}$ (dashed line) and $k=0$ (solid line). }
\label{fig:error_bh}
\end{figure}

\subsection{Numerical solution for different k-values}

In this section we illustrate the behavior of the numerical solution to problem \eqref{VF:num_dis}
for different configurations of the dynamic Preisach model. As we notice in Section~\ref{DPM},
the evolution of the dynamic relay and, accordingly, the dynamic Preisach model, varies with respect
to the velocity of the input and the relay slope $k$ (cf.~Figures~\ref{fig:ex_1},
\ref{fig:ex_2} and \ref{fig:BH_curve}). We have computed the solution of problem \eqref{VF:num_dis}
 where we have only changed the slope $k$ of the dynamic Preisach model. For the considered examples
 $k=1$ and $k=10^3$. Figures \ref{fig:u_field} and \ref{fig:w_field} show  fields $u$ and $w$ solution
to problem~\eqref{VF:num_dis}. From Figure~\ref{fig:w_field} (bottom) it can be seen that changes in $w$
are smaller when $k=1$. This behavior is expected as  the size of the $u-w$ cycle decreases
 when the slope, k,  decreases (see Figure~\ref{fig:ex_1}). This is not the case
for the $w$ field for $k=10^3$: it reaches values close to  saturation
(see Figures~\ref{fig:w_field} (top) and \ref{fig:error_bh} (left)).

    \begin{figure}
   \centering
   \includegraphics*[width=0.3\textwidth]{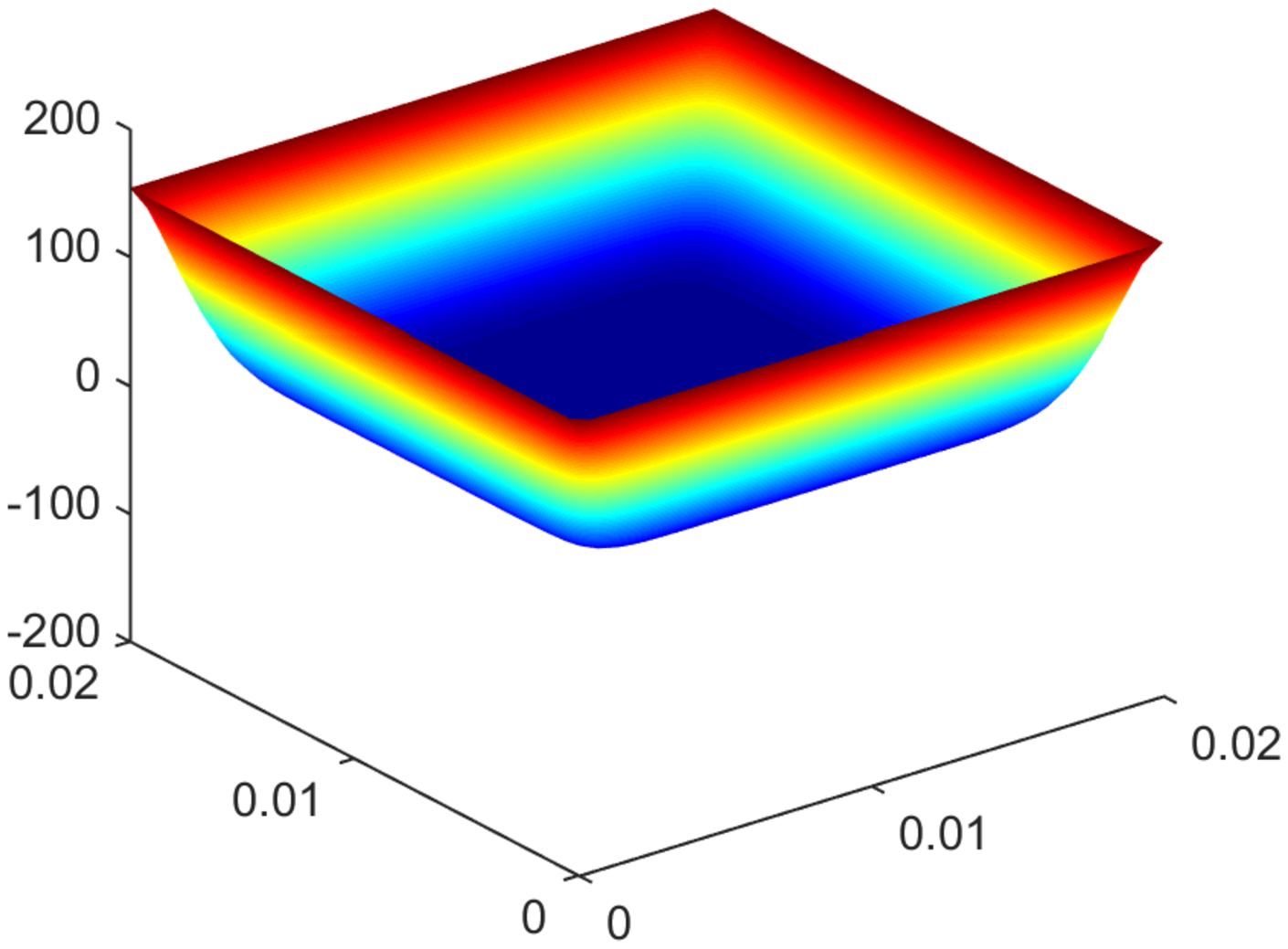}
    \includegraphics*[width=0.3\textwidth]{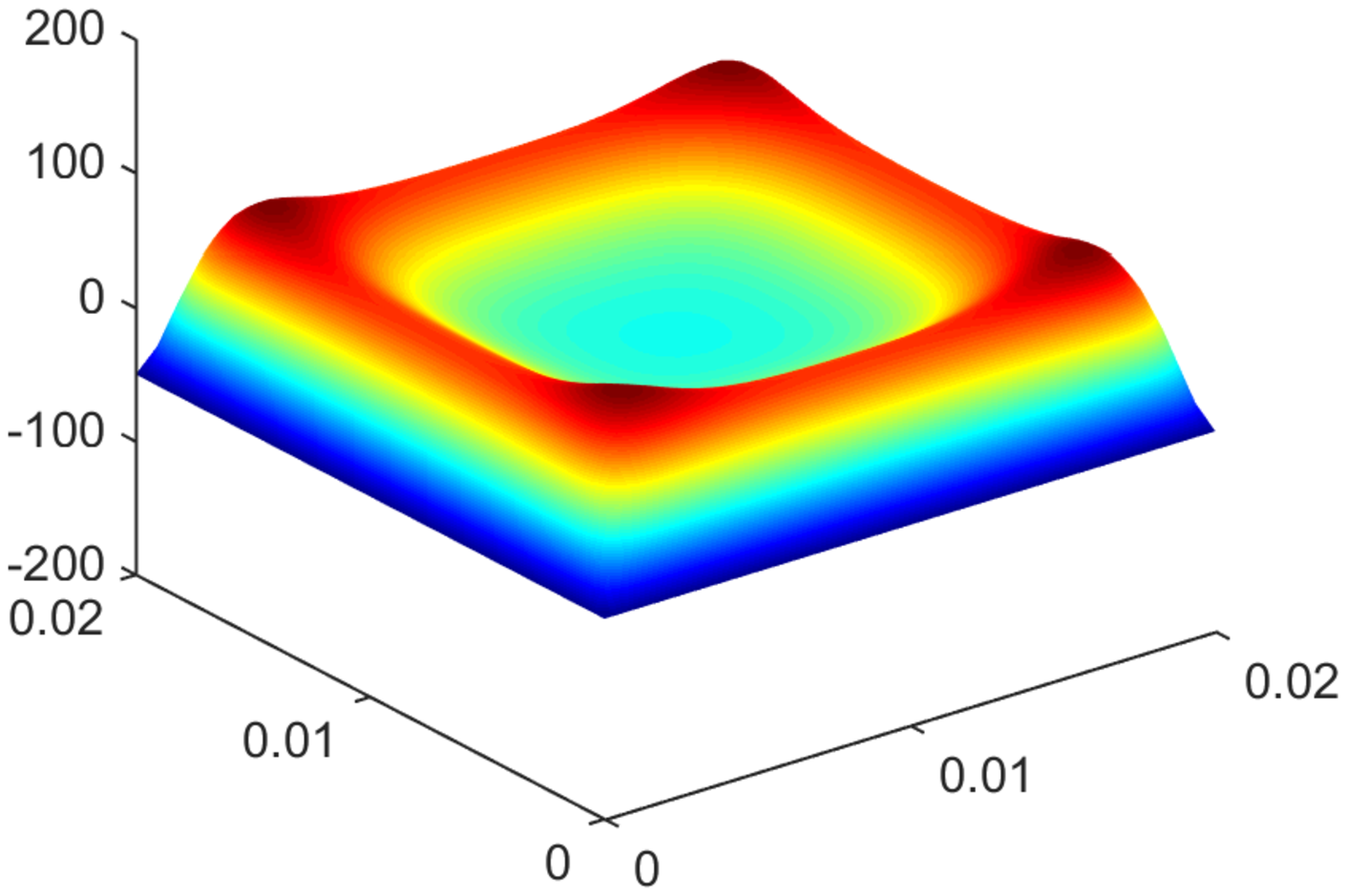}
    \includegraphics*[width=0.3\textwidth]{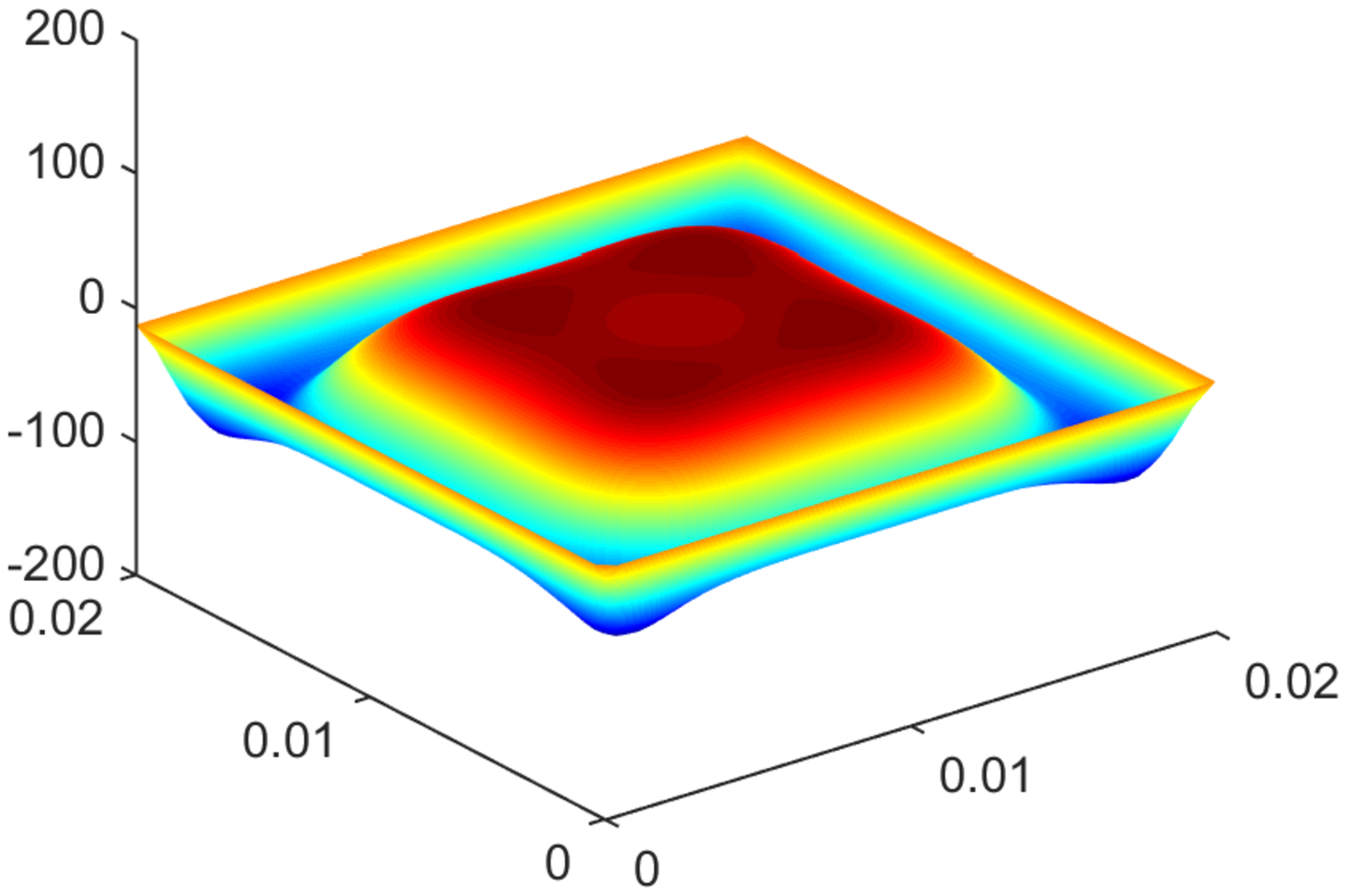}\\
    \includegraphics*[width=0.3\textwidth]{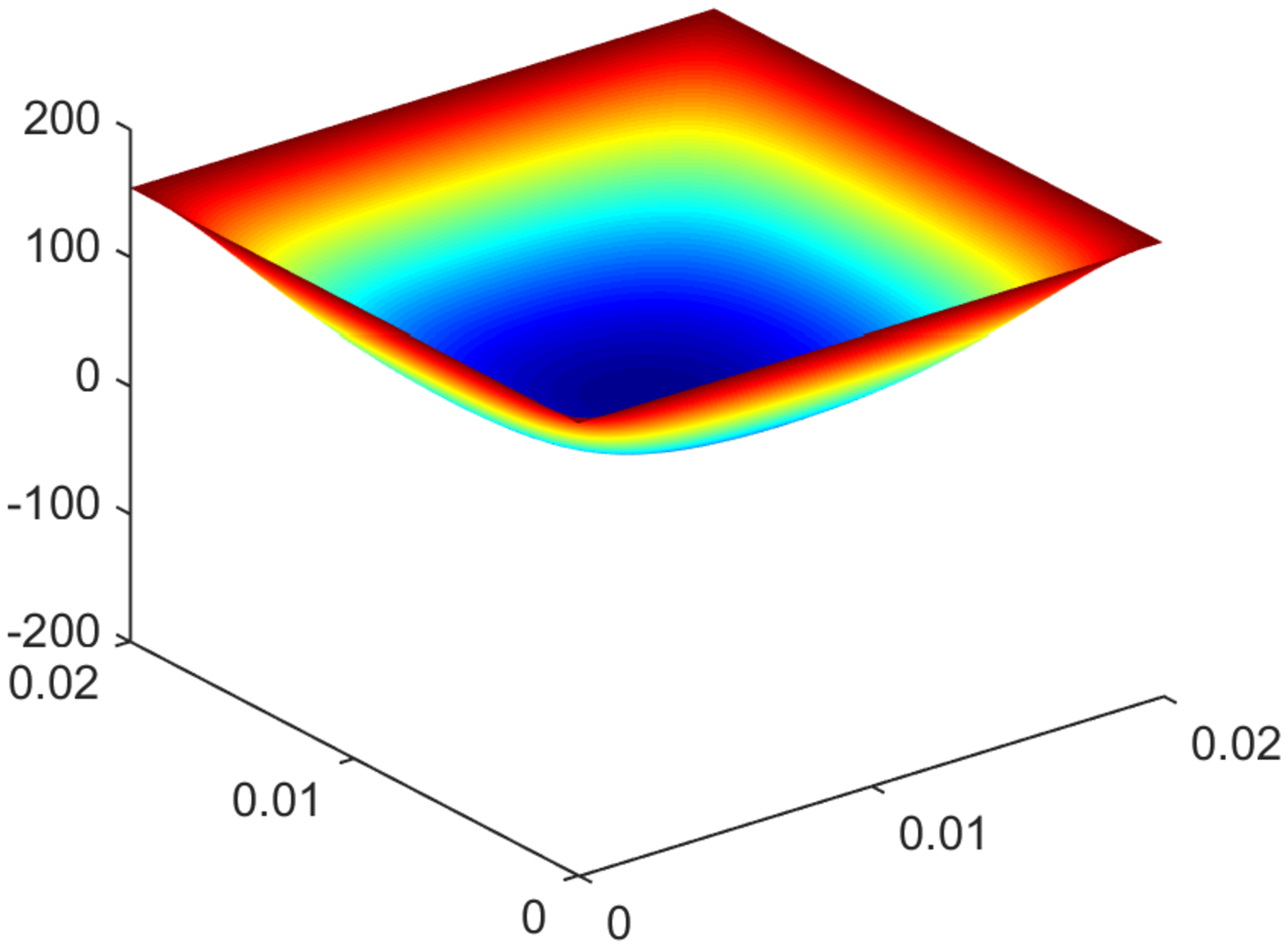}
    \includegraphics*[width=0.3\textwidth]{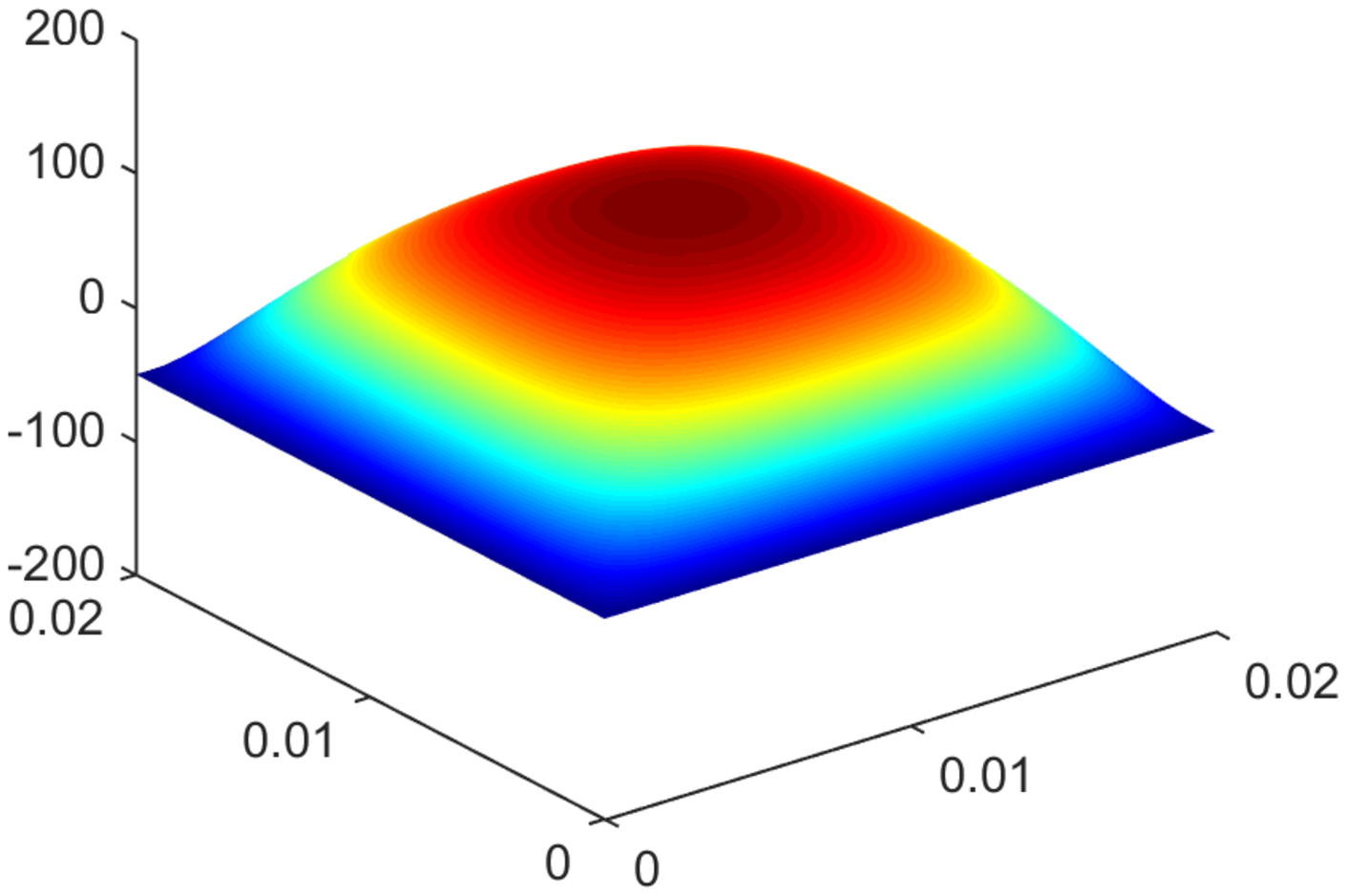}
    \includegraphics*[width=0.3\textwidth]{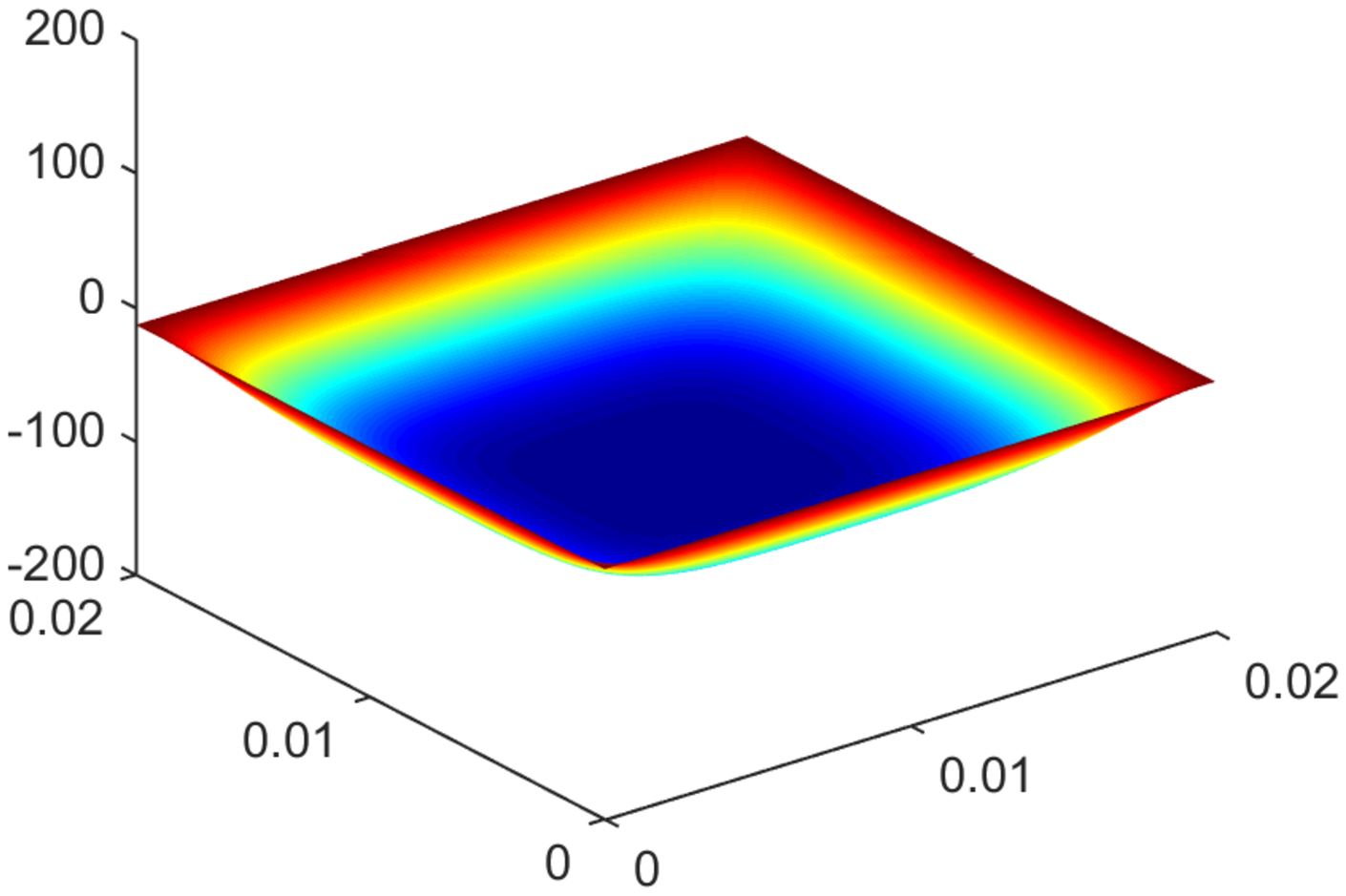}\\
    $t=0.002$s\hspace{0.2\textwidth} $t=0.0055$s \hspace{0.2\textwidth}$t=0.01$s
    \caption{$u$-field solution to problem~\eqref{VF:num_dis} for $k=10^3$ (top) and $k=1$ (bottom).}
    \label{fig:u_field}
    \end{figure}
%
\begin{figure}
    \centering
    \includegraphics*[width=0.3\textwidth]{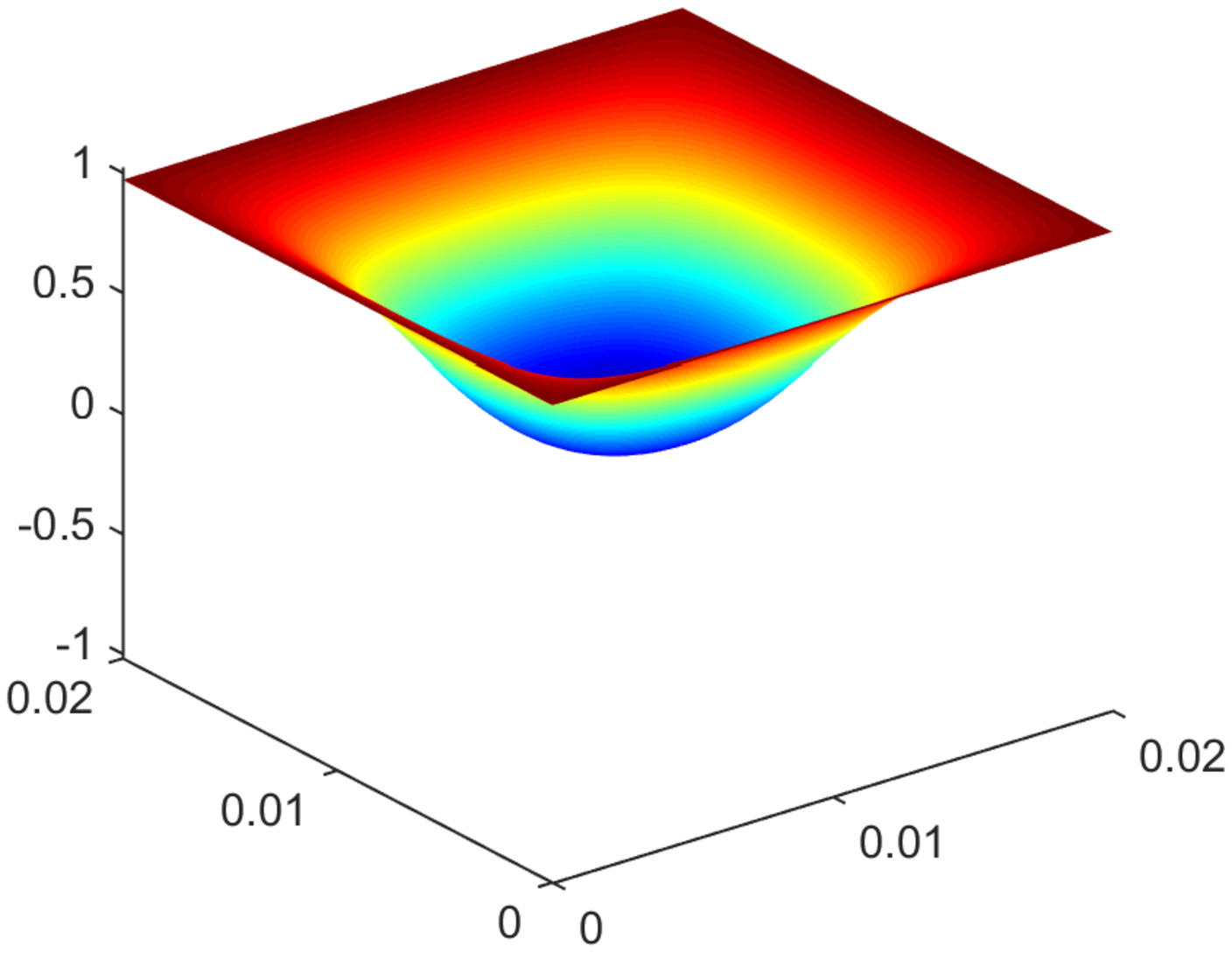}
    \includegraphics*[width=0.3\textwidth]{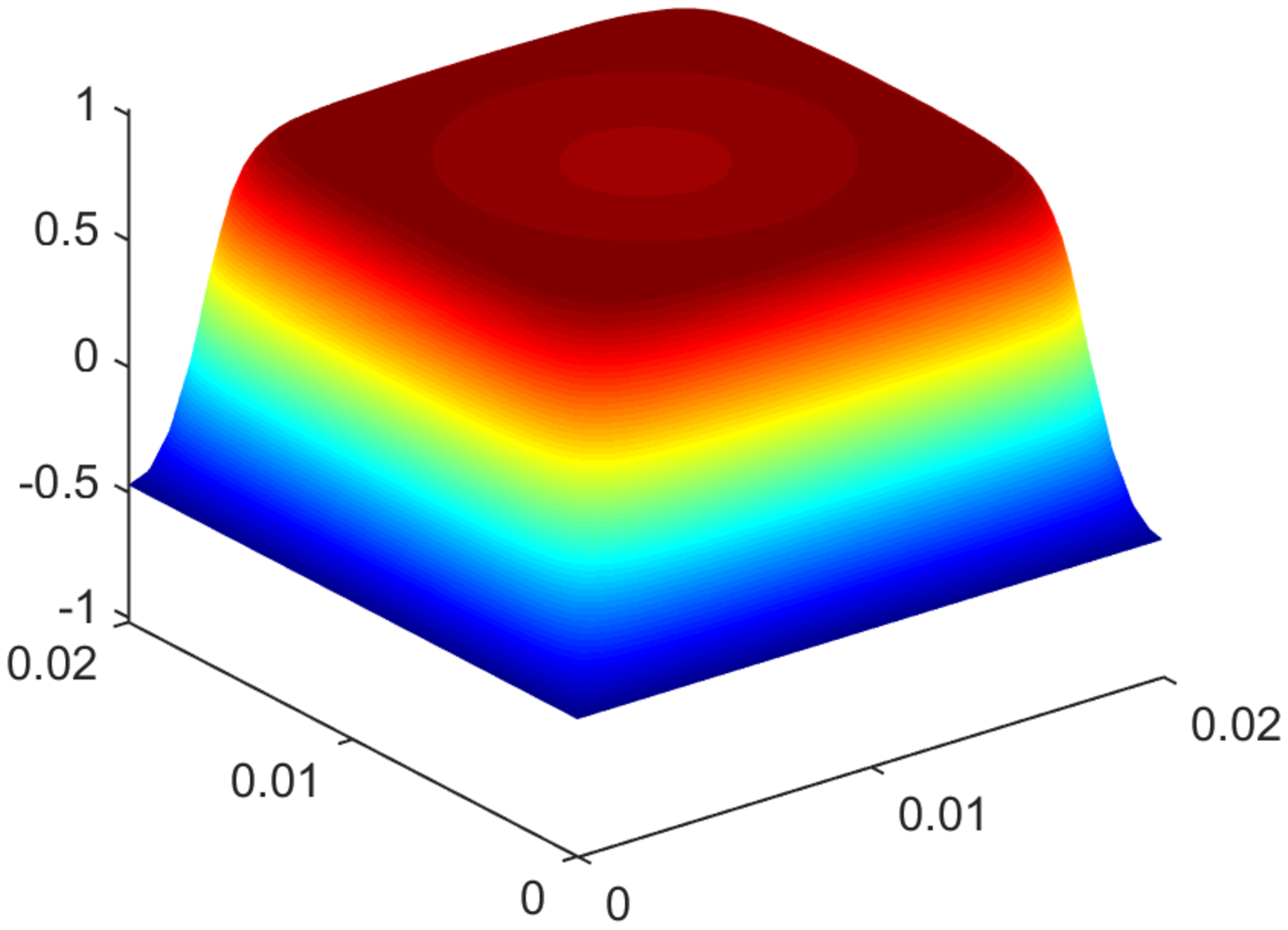}
    \includegraphics*[width=0.3\textwidth]{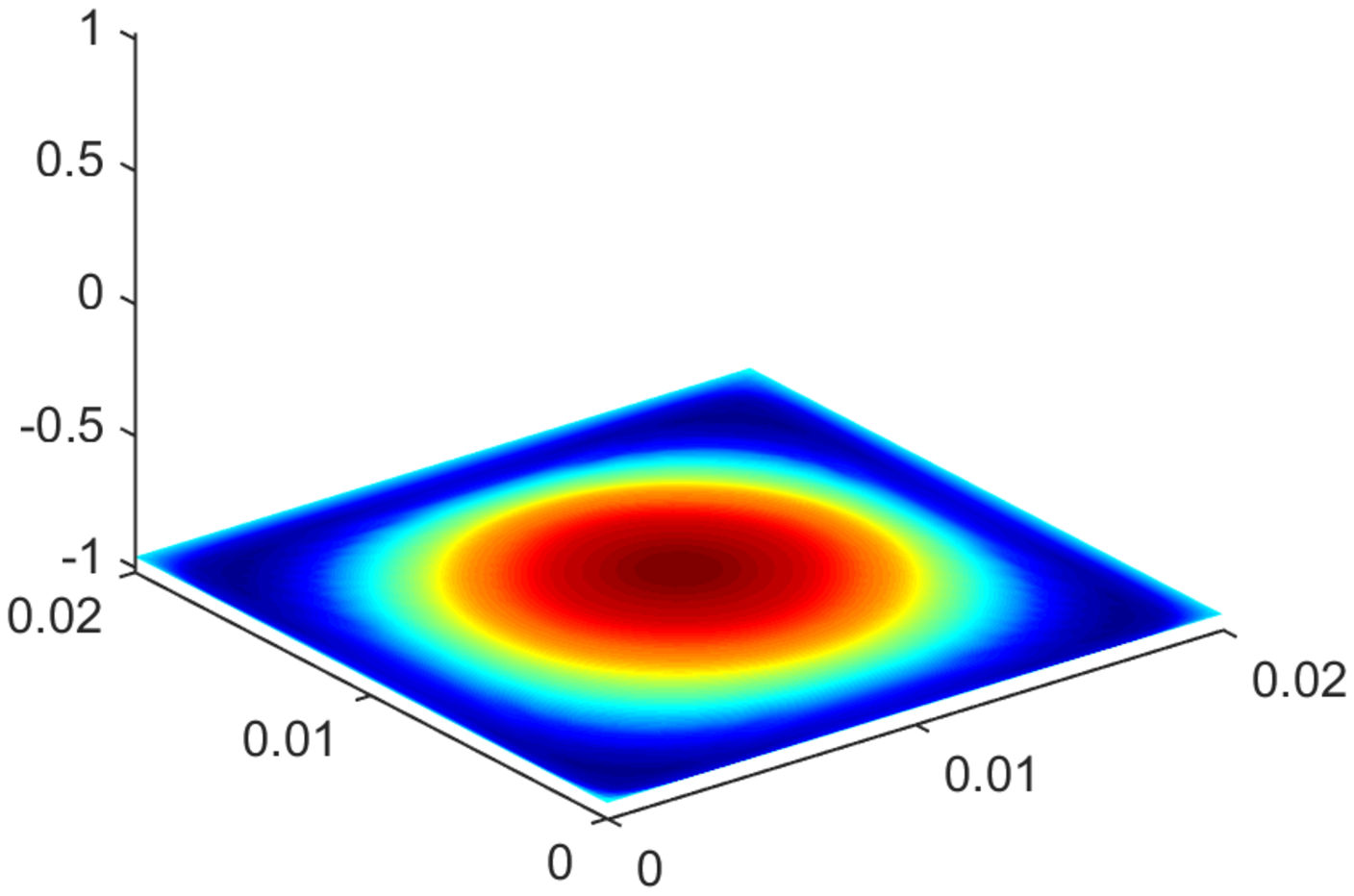}\\
    \includegraphics*[width=0.3\textwidth]{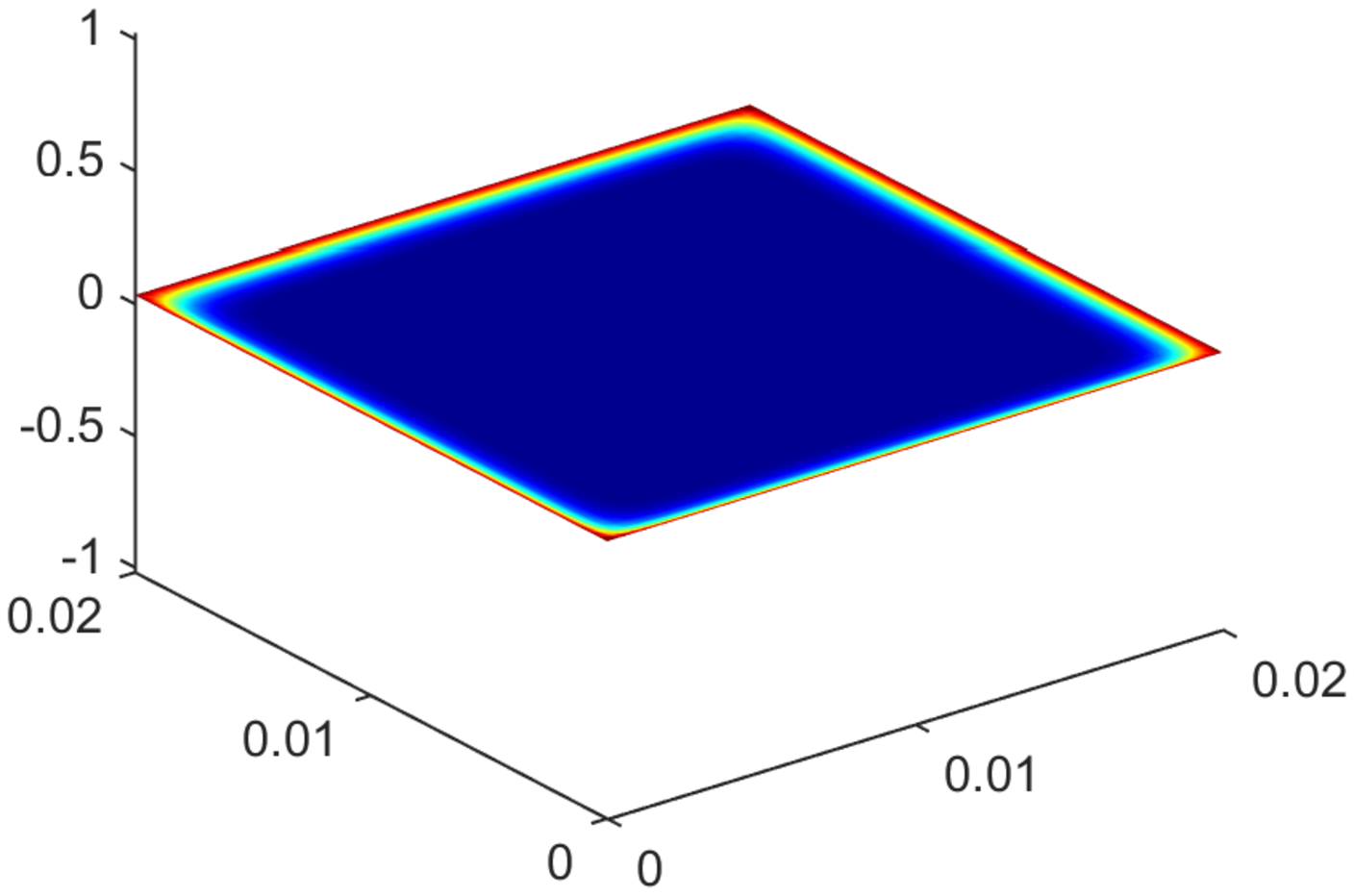}
    \includegraphics*[width=0.3\textwidth]{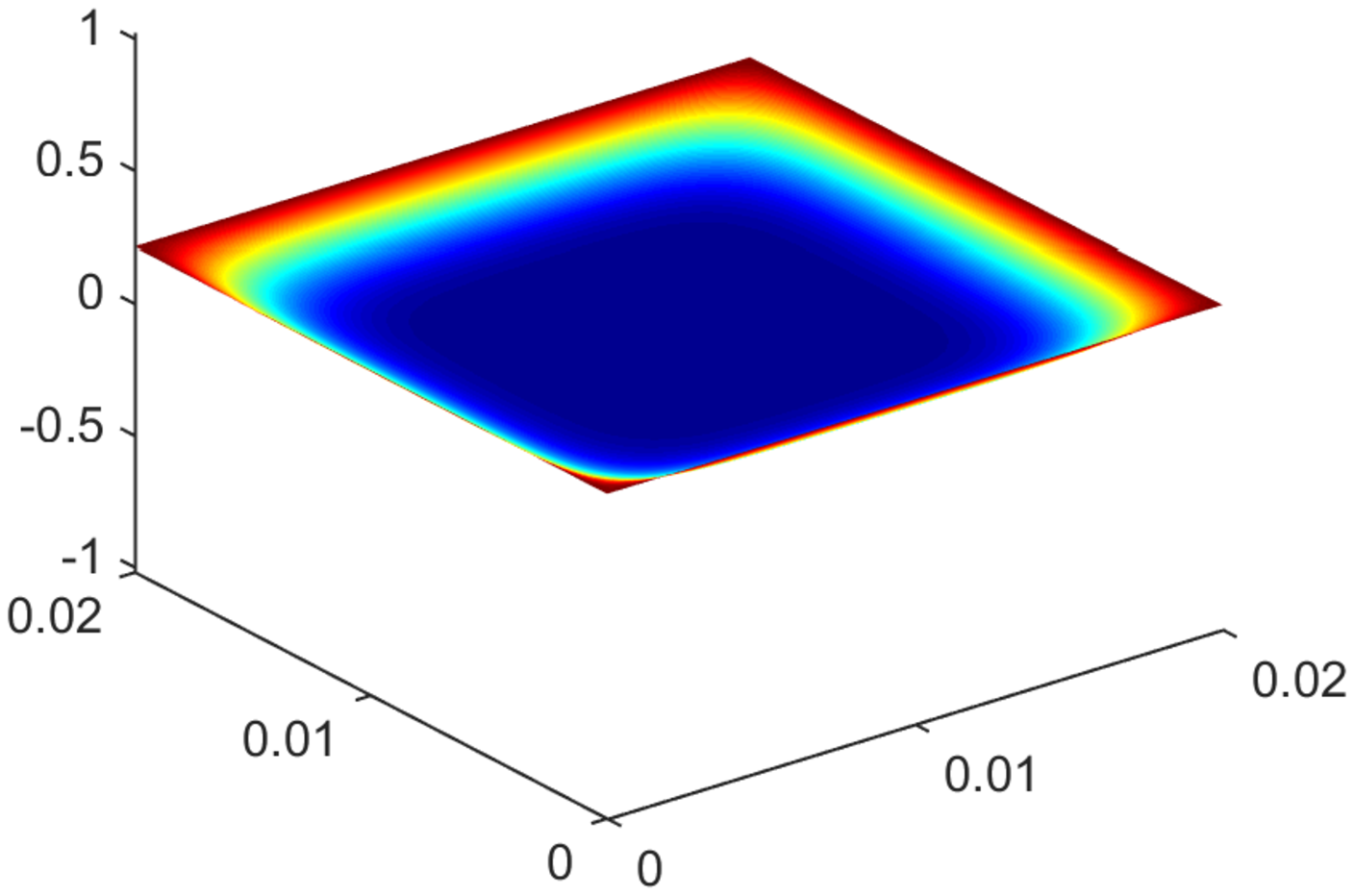}
    \includegraphics*[width=0.3\textwidth]{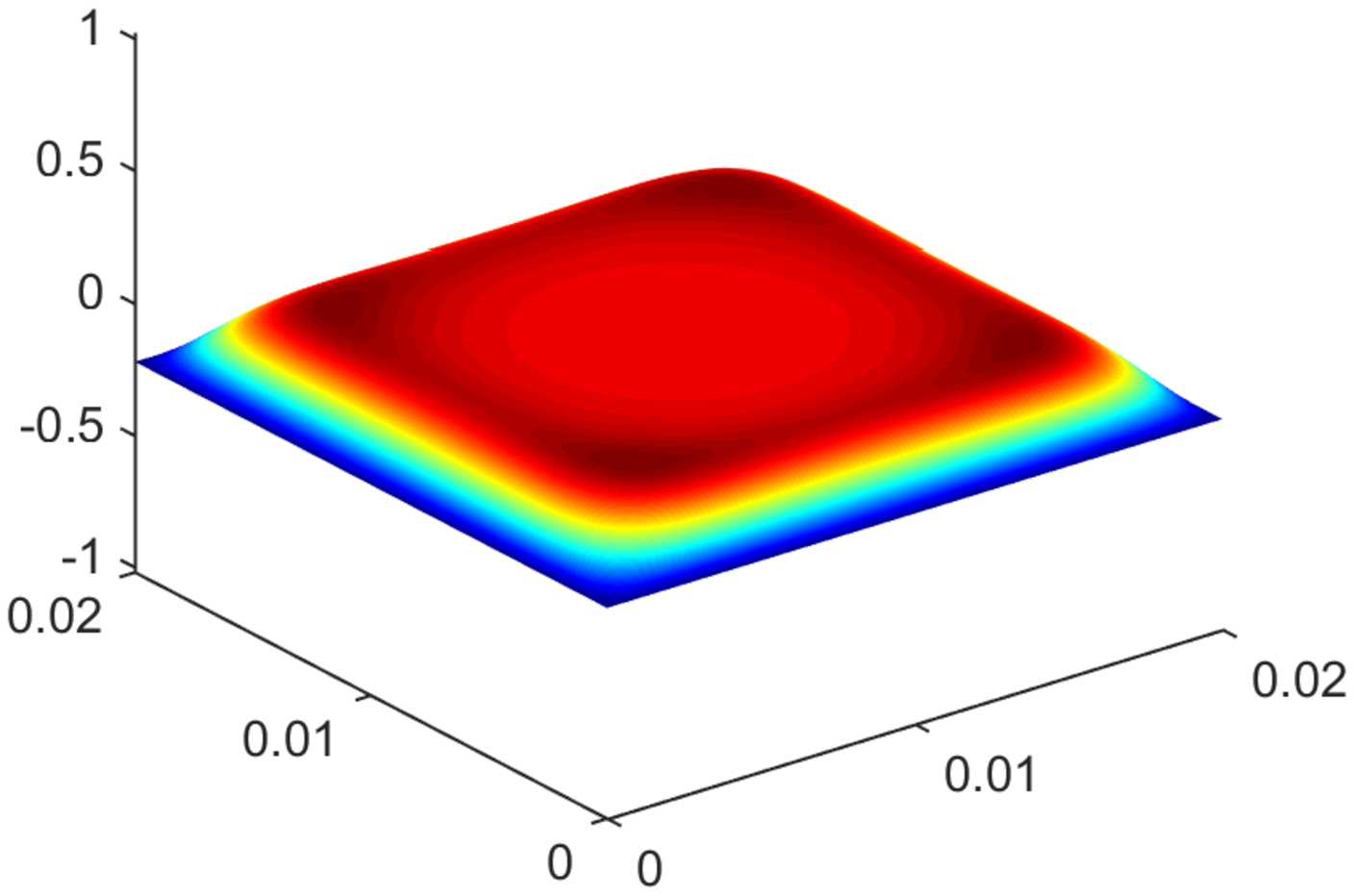}\\
    $t=0.002$s\hspace{0.2\textwidth} $t=0.0055$s \hspace{0.2\textwidth}$t=0.01$s
    \caption{$w$-field solution to problem~\eqref{VF:num_dis} for $k=10^3$ (top) and $k=1$ (bottom).}
    \label{fig:w_field}
    \end{figure}

\section{Conclusion}

In this paper, we have introduced  a functional analytic framework
to study the dynamic Preisach model presented in \cite{B92}.
This framework allow us to prove properties of the operator that
can be used to study a wide range of hysteresis phenomena.
In particular, and  motivated by electromagnetic field equations,
we study the well posedness of a family of PDE's involving hysteresis.
We have noticed that, even though the analysis is similar  to the one
applied for the classical Preisach model,
key changes need to be done for the dynamic case.
We also propose a numerical scheme to approximate a parabolic problem
with dynamic Hysteresis.
The numerical results presented in sections \ref{DPM} and \ref{Num} help us to
become familiar with the role of the $k$-value and the speed of the input
in the dynamic Preisach model and how this parameter affects the solution
of a PDE. A better understanding of the dynamic behavior of solution in processes
with hysteresis is essential to understand the process itself and to develop
appropriate numerical approximations.

\appendix

\section{Proof of Theorem~\ref{eq:existencerelay}}
\label{ap:proof}

Let us introduce the Lipschitz-continuous approximation of the
multi-valued function $\partial\chi_{[-1,1]}$ defined by
$$
(\partial \chi_{[-1,1]})_\mu (x)=\left\{\begin{array}{ll}
\displaystyle \frac{x+1}{\mu} & \mbox{ if } x\leq -1,\\
\displaystyle0  & \mbox{ if } -1<x<1,
\\
\displaystyle \frac{x-1}{\mu} & \mbox{ if } x\geq 1.
\end{array} \right.
$$
Actually, $(\partial \chi_{[-1,1]})_\mu$ is the so-called Yosida
regularization of the multi-valued operator $\partial \chi_{[-1,1]}$. We notice that
 \begin{equation}
 \label{eq:yosida}
(\partial \chi_{[-1,1]})_\mu (x)= \frac{x-P_{[-1,1]}(x)}{\mu}\quad \forall x \in\R
\end{equation}
 where $P_{[-1,1]}$ denotes the projection on the interval $[-1,1]$.

Let us consider an approximation of the  Cauchy problem \eqref{eq:relay} defined as
\begin{align}
\label{eq:reg1}
&\dfrac{dy_\mu}{dt}(t)+(\partial \chi_{[-1,1]})_\mu (y_\mu (t))=g_\rho(u(t)),
\\
&\label{eq:reg2}
y(0)=\xi.
\end{align}
Since $g_{\rho}\in \rL^2(0,T)$ and $(\partial \chi_{[-1,1]})_\mu$
is Lipschitz-continuous this problem has a unique solution from the
Picard-Lipschitz Theorem. Moreover $y_\mu\in \rH^1(0,T)$. In order to
pass to the limit as $\mu\rightarrow 0$ let us obtain some a priori
estimates. For this purpose we multiply both sides by $y_\mu(t)$. We have
$$
\dfrac{dy_\mu}{dt}(t)y_\mu(t)+(\partial \chi_{[-1,1]})_\mu (y_\mu (t))y_\mu(t)=g_\rho(u(t))y_\mu(t)
$$
 and then
 $$
 \frac{1}{2}\frac{d}{dt}|y_\mu(t)|^2\leq |g_\rho(u(t))||y_\mu(t)|.
 $$
 From Gronwall's inequality we deduce
 $$
 |y_\mu(t)|\leq C \quad \forall t\in (0,T)
 $$
 for some constant $C$. Next, we multiply (\ref{eq:reg1}) by
 $\frac{dy_\mu}{dt}(t)$. We get
 \begin{equation}
 \label{eq:reg3}
 |\frac{dy_\mu}{dt}(t)|^2 + (\partial \chi_{[-1,1]})_\mu (y_\mu (t))
 \frac{dy_\mu}{dt}(t)=g_\rho(u(t))\frac{dy_\mu}{dt}(t).
 \end{equation}
 From the chain rule we have
 $$
 (\partial \chi_{[-1,1]})_\mu (y_\mu (t))\frac{dy_\mu}{dt}(t)
 =\frac{d}{dt}(\chi_{[-1,1]})_\mu (y_\mu(t)),
 $$
where $(\chi_{[-1,1]})_\mu(x)$  is a primitive of
$(\partial \chi_{[-1,1]})_\mu(x)$. More precisely,
 $$
( \chi_{[-1,1]})_\mu (x)=\left\{\begin{array}{ll}
\displaystyle \frac{(x+1)^2}{2\mu} & \mbox{ if } x\leq -1,\\
\displaystyle0  & \mbox{ if } -1<x<1,
\\
\displaystyle \frac{(x-1)^2}{2\mu} & \mbox{ if } x\geq 1.
\end{array} \right.
$$
Since $( \chi_{[-1,1]})_\mu$ is non-negative, (\ref{eq:reg3}) yields
 $$
 \int_0^T|\frac{dy_\mu}{dt}(t)|^2\,dt \leq
 \Big(\int_0^T |g_\rho(u(t))|^2dt\Big)^{1/2}\Big(\int_0^T |\frac{dy_\mu}{dt}|^2(t)\,dt\Big)^{1/2}
 $$
 and hence
 $$
 \displaystyle\|\frac{dy_\mu}{dt}\|_{\rL^2(0,T)}\leq \|g_\rho\|_{\rL^2(0,T)}.
 $$
 Thus, the set $\{y_\mu\}$ is bounded in $\rH^1(0,T)$ which implies
 that there exists $y\in \rH^1(0,T)$ and a subsequence $\{y_{\mu_n}\}$ such that
 $$
 \{y_{\mu_n}\} \mbox{ converges to } y \mbox{ weakly in } \rH^1(0,T)
 \mbox{ and strongly in } \rL^2(0,T).
 $$
 Let
 $$q_\mu:= (\partial \chi_{[-1,1]})_\mu (y_\mu (t)).$$
 This means
 \begin{equation}
 \label{eq:qmu}
 q_\mu(t)\Big(x-y_\mu(t)\Big)\leq \Big(\chi_{[-1,1]}\Big)_\mu(x)
-\Big(\chi_{[-1,1]}\Big)_\mu(y_\mu(t)) \quad \forall x \in \R.
 \end{equation}
  Besides, since $\displaystyle q_\mu=g_\rho(u(t))-\dfrac{dy_\mu}{dt}(t)$
   then the set $\{q_\mu\}$  is bounded in $\rL^2(0,T)$ and hence there
   exists $q\in \rL^2(0,T)$ and a subsequence $\{q_{\mu_n}\}$ such that
 $$
 \{q_{\mu_n}\} \mbox{ converges to } q \mbox{ weakly in } \rL^2(0,T).
 $$
 Passing to the limit in (\ref{eq:reg1}) we get
 \begin{equation}
 \label{eq:reg4}
 \dfrac{dy}{dt}(t)+q(t)=g_\rho(u(t)).
 \end{equation}
 Now the goal is to prove that $q(t)\in \partial \chi_{[-1,1]}(y(t))$
 a.e. in $[0,T]$. For this purpose we notice that, since $\{q_\mu\}$
 is bounded in $\rL^2(0,T)$, then (\ref{eq:yosida}) yields
 $$
 \lim_{n\rightarrow \infty}\{ P_{[-1,1]}(y_{\mu_n})\}= y\mbox{ strongly in } \rL^2(0,T),
 $$
 which in particular implies $-1\leq y(t)\leq 1\ \forall t\in[0,T]$.
  Now, let pass to the limit in (\ref{eq:qmu}) for $\mu=\mu_n$. Firstly, we have
 \begin{equation}
 \label{eq:qmulim}
0\leq \int_0^T\Big(\chi_{[-1,1]}\Big)_{\mu_n}(y_{\mu_n}(t))\, dt\leq
-\int_0^T q_{\mu_n}(t)\big(z(t)-y_{\mu_n}(t)\big)\, dt
+\int_0^T \Big(\chi_{[-1,1]}\Big)_{\mu_n}(z(t))\,dt,
 \end{equation}
 for all $z\in \rL^2(0,T)$.
 Then,
 \begin{equation}
 \label{eq:qlim}
 0\leq - \int_0^T q(t)\big(z(t)-y(t)\big)\, dt+\int_0^T\chi_{[-1,1]}(z(t))\,dt
 \end{equation}
because of the strong convergence of $\{y_{\mu_n}\}$, the weak
convergence of $\{q_{\mu_n}\}$ and the fact that
$$
\lim_{n\rightarrow\infty}\int_0^T \Big(\chi_{[-1,1]}\Big)_{\mu_n}(z(t))\,dt
=\int_0^T \chi_{[-1,1]}(z(t))\,dt.
$$
Finally, since $|y(t)|\leq 1$, then
$\chi_{[-1,1]}(y(t))=0\, \forall t$ and we deduce from (\ref{eq:qlim})
$$
 \int_0^T q(t)\big(z(t)-y(t)\big)\, dt\leq
 \int_0^T\chi_{[-1,1]}(z(t))\,dt -\int_0^T\chi_{[-1,1]}(y(t))\,dt
$$
and then $q(t)\in \partial \chi_{[-1,1]}(y(t))\, \mbox{ a.e. in } (0,T)$.

Let us prove uniqueness. Assume that $y_1(t)$ and $y_2(t)$
are two solutions of (\ref{eq:reg1}), (\ref{eq:reg2}).
Firstly $y_1(0)-y_2(0)=0$ and there exist $q_i(t)\in
\partial \chi_{[-1,1]}(y_i(t)),\ i=1,2 $  such that
$$
\dfrac{dy_i}{dt}(t)+q_i(t)=g_\rho(u(t)),\ i=1,2.
$$
Let us subtract the above equations for $i=1,2$ and then
make the scalar product by $y_1-y_2$. We get
$$
\frac{1}{2}\int_0^t\frac{d}{ds}|y_1(s)-y_2(s)|^2 \leq 0
$$
because,
$$
\int_0^t (q_1(s)-q_2(s))(y_1(s)-y_2(s))\,dt \geq 0,
$$
as $\partial \chi_{[-1,1]}$ is monotone.
Hence,
$$
|y_1(t)-y_2(t)|^2\leq |y_1(0)-y_2(0))|^2=0,
$$
which finishes the proof.

\bibliography{BGV_bib}

\end{document}